\documentclass[12pt, a4paper]{amsart}
\usepackage[dvips]{epsfig}
\usepackage{amsmath}
\usepackage{amssymb}
\usepackage{amsfonts}
\usepackage{bbm}
\usepackage{amsthm}
\usepackage{amsbsy}
\usepackage{amsgen}
\usepackage{amscd}
\usepackage{amsopn}
\usepackage{amstext}
\usepackage{amsxtra}
\usepackage[utf8]{inputenc}
\usepackage{enumerate}
\usepackage{hyperref}
\usepackage{lineno}
\usepackage{marginnote}
\usepackage{verbatim}
\usepackage{calrsfs}
\usepackage{times, graphicx}
\usepackage{layout}
\usepackage{subcaption}
\usepackage{eso-pic}
\usepackage{lastpage}
\usepackage{dsfont}
\usepackage{fixmath}

\DeclareMathAlphabet{\pazocal}{OMS}{zplm}{m}{n}
\usepackage[foot]{amsaddr}

  \evensidemargin
\oddsidemargin
\newtheorem{theorem}{Theorem}[section]
\newtheorem{proposition}[theorem]{Proposition}
\newtheorem{lemma}[theorem]{Lemma}
\newtheorem{corollary}[theorem]{Corollary}

\theoremstyle{definition}
\newtheorem{definition}[theorem]{Definition}

\newtheorem{conjecture}[theorem]{Conjecture}

\numberwithin{equation}{section}

\DeclareMathAlphabet{\pazocal}{OMS}{zplm}{m}{n}

\usepackage{accents}
\newcommand{\ubar}[1]{\underaccent{\bar}{#1}}

\newcommand{\id}{\mathds{1}}
\newcommand{\R}{\mathbb{R}}
\newcommand{\C}{\mathbb{C}}

\newcommand{\E}{\mathbb{E}}
\newcommand{\Z}{\mathbb{Z}}

\newcommand{\eps}{\varepsilon}
\newcommand{\Jc}{\pazocal{J}}

\newcommand{\prob}{\pazocal{P}r}
\newcommand {\Pb} {\prob}
\newcommand{\Mcc}{\pazocal{M}}
\newcommand{\Dc}{\pazocal{D}}
\newcommand{\Dcc}{\mathcal{D}}
\newcommand{\area}{\operatorname{Area}}
\newcommand{\inj}{\operatorname{Inj}}
\newcommand{\Bc}{\mathcal{B}}
\newcommand{\Sc}{\pazocal{S}}
\newcommand{\Ec}{\pazocal{E}}
\newcommand{\Ecc}{\mathcal{E}}
\newcommand{\var}{\operatorname{Var}}
\newcommand{\Gc}{\pazocal{G}}

\newcommand{\Wc}{\pazocal{W}}
\newcommand{\Fc}{\pazocal{F}}
\newcommand{\Xc}{\pazocal{X}}
\newcommand{\Hc}{\pazocal{H}}
\newcommand{\Fcc}{\mathcal{F}}
\newcommand{\vol}{\operatorname{Vol}}
\newcommand{\med}{\operatorname{Med}}
\newcommand{\Cc}{\pazocal{C}}
\newcommand{\Ac}{\pazocal{A}}

\begin{document}

\title{Sign-balance of random Laplace eigenfunctions} 
\author{Stephen Muirhead$^{1}$}
\email{stephen.muirhead@monash.edu}
\address{$^{1}$School of Mathematics, Monash University}
\author{Igor Wigman$^{2}$}
\email{igor.wigman@kcl.ac.uk}
\address{$^{2}$Department of Mathematics, King's College London}

\date{} 
\date{\today}

\begin{abstract}
Motivated by the problem of the small-scale sign distribution of Laplace eigenfunctions, we introduce a strong notion of sign-balance for (eigen)functions, and prove that random eigenfunctions are sign-balanced above a precisely determined scale with almost full probability. The scale is proven to be optimal up to a logarithmic power of the energy. Our results include the important case of random spherical harmonics, as well as more general band-limited random waves on smooth Riemannian manifolds. Extending the notion of balance to arbitrary levels, we determine the precise optimum scale above which random eigenfunctions are volume-balanced with respect to non-zero levels. Beyond their intrinsic interest, our results serve as a model for a natural conjecture on the optimal scale at which {\em deterministic} Laplace eigenfunctions are sign-balanced.
\end{abstract}

\maketitle

\section{Introduction}
\label{s:intro}

\subsection{Motivation}
\label{sec:intro mot}

Let $(\Mcc,g)$ be a closed Riemannian $d$-dimensional manifold, and $\Delta$ the Laplace-Beltrami operator on $\Mcc$. It is well-known that the spectrum of $\Delta$ is discrete, consisting of eigenvalues $\lambda_{j}\rightarrow \infty$ corresponding to (real-valued) eigenfunctions $\varphi_j$ satisfying the Helmholtz equation
\begin{equation}
\label{eq:Helmholts}
\Delta\varphi_{j} + \lambda_{j}^{2}\varphi_{j} = 0. 
\end{equation}
We are interested in the nodal geometry of $\varphi_{j}$ for large $\lambda_{j}$. In the case of `generic' chaotic manifolds $\Mcc$, a heuristic argument due to M.\ Berry ~\cite{Berry77} suggests that, in the high energy limit $\lambda_{j}\rightarrow \infty$, the eigenfunctions $\varphi_{j}$ exhibit, on geodesic balls $B_{r}(x) \subseteq \Mcc$ of radius at or above the \textit{Planck scale} $r = \frac{1}{\lambda_{j}}$, various statistical properties similar to those of random monochromatic waves. These are represented by the centred Gaussian isotropic random field $F:\R^{d}\rightarrow \R$ with the covariance function 
\begin{equation}
\label{eq:Berry's RWM covar}
\E\left[F(x)\cdot F(y)   \right] = \gamma_d(\|x-y\|),
\end{equation}
where $\gamma_d(\|\cdot\|)$ is the Fourier transform of the uniform measure on the sphere $\Sc^d$ (e.g.\ $\gamma_2(\|\cdot\|) = J_0(\|\cdot\|)$). Further, it is suggested that the nodal geometry of $\varphi_{j}$ and $F$ should be similarly related.

\vspace{2mm}
In this paper we are interested in the {\em sign distribution} of the eigenfunctions $\varphi_{j}$. One may expect from Berry's ansatz that \textit{every} geodesic ball $B_{r}(x) \subseteq \Mcc$ of sufficiently large radius $r\gg \frac{1}{\lambda_{j}}$ will contain a `balanced' proportion of positive and negative values of $\varphi_{j}$. That is, in the appropriate regime of $r \gg \frac{1}{\lambda_{j}}$, as $j\rightarrow \infty$ one might expect the \textit{uniform} limit 
\begin{equation}
\label{eq:fair share}
\frac{\left|\varphi_{j}^{-1}(0,\infty)\cap B_{r}(x)\right|}{|B_{r}(x)|} \rightarrow \frac{1}{2}, \quad x \in \Mcc,
\end{equation}
where for a subset $\Dcc\subseteq \Mcc$, $|\Dcc|$ is the $d$-dimensional volume measure of $\Dcc$. 
Informally, we say that a sequence of eigenfunctions $\varphi_{j}$ satisfying \eqref{eq:fair share} is {\em sign-balanced} (see Definition \ref{def:sign-balance} below). Our primary objective is to establish a positive result of this nature, thereby lending support to Berry’s ansatz.

\vspace{2mm}
There are only a few results pertaining to the sign distribution of Laplace eigenfunctions in the literature. A classical {\em quasi-symmetry} theorem due to Donnelly-Fefferman ~\cite{DoFeInt88} asserts, under the extra assumption that $g$ is analytic, that there exists a number $a>0$ depending only on $g$, such that for every {\em fixed} geodesic ball $\Dcc \subseteq \Mcc$, and sufficiently large $j$
(depending on $g$ and the radius of $\Dcc$), one has
\begin{equation*}
\frac{\left|\varphi_{j}^{-1}(0,\infty)\cap \Dcc\right|}{|\Dcc|} \ge a.
\end{equation*}
Recently, Logunov-Nazarov ~\cite{LoNa} extended the quasi-symmetry to {\em smooth} closed surfaces.
For $d=2$, Nazarov-Sodin-Polterovich ~\cite{NaPoSoAJM05} posed the question at what scales the quasi-symmetry persists, and proved ~\cite[Theorem 1.4]{NaPoSoAJM05} 
a small-scale analogue of this result: there exists a number $a>0$ depending only on $g$, so that the bound
\begin{equation}
\label{eq:area>=a/log loglog}
\frac{ | \varphi_{j}^{-1}(0,\infty)\cap \Dcc |}{|\Dcc|} \ge \frac{a}{\log{\lambda_{j}}\cdot \sqrt{\log\log{\lambda_{j}}}}
\end{equation}
holds for every metric disk $\Dcc = B_r(x)$ that is `deeply intersected' by $\varphi_{j}^{-1}(0,\infty)$ (meaning that $\varphi_{j}^{-1}(0,\infty)\cap B_{r/2}(x)\ne \varnothing$). The lower bound in \eqref{eq:area>=a/log loglog} is sharp up to the factor $\sqrt{\log\log{\lambda_{j}}}$, as there exists a number $C>0$, a sequence of spherical caps $\Dcc_{j}$ on the round $2$-sphere, and a sequence $\varphi_{j}$ of spherical harmonics vanishing on the centre of $\Dcc$, so that
\begin{equation*}
\frac{ | \varphi_{j}^{-1}(0,\infty)\cap \Dcc_{j} |}{ |\Dcc_{j} | } \le \frac{C}{\log{\lambda_{j}}}.
\end{equation*}

\subsection{Principal result I: Sign-balance for random spherical harmonics}
In this paper we propose a new notion of \textit{sign-balance}, finer than quasi-symmetry, and address whether an appropriately defined `generic' Laplace eigenfunction satisfies this strong property.

\subsubsection{{\bf Sign-balanced sequences of Laplace eigenfunctions}}

\begin{definition}[Sign-balanced sequences of functions]
$\, $
\label{def:sign-balance}
\begin{enumerate}[i.]

\item For $x\in \Mcc$ and $r>0$ below the injectivity radius $\inj(\Mcc)$ of $\Mcc$, the \textit{defect} of a measurable function $f:\Mcc\rightarrow\R$ on $B_{r}(x)$ is 
\begin{equation}
\label{eq:defect def}
\Dc(x;r)=\Dc_{f}(x;r) := \frac{1}{|B_{r}(x)|}\int\limits_{B_{r}(x)} H(f(y))dy \in [-1,1],
\end{equation}
where $H(\cdot)$ is the sign function 
\begin{equation}
\label{eq:H Heaviside}
H(t)= \begin{cases}
1 &t\ge 0 \\ -1 &t<0
\end{cases} = \id_{\cdot \ge 0}(t) - \id_{\cdot < 0}(t).
\end{equation}
 \item Let $\{r_{j}\}\subseteq \R_{>0}$ be a sequence of positive numbers, and $f_{j}:\Mcc\rightarrow\R$ a sequence of measurable functions. The \textit{sign-imbalance} of $f_{j}$ at scale $r_{j}$ is 
\begin{equation}
\label{eq:sign-imbalance def}
\Bc(f_{j};r_{j}):= \sup\limits_{\substack{x\in\Mcc}}|\Dc_{f_{j}}(x;r_{j})| \in [0,1].
\end{equation}
We say that the $f_{j}(\cdot)$ are \textit{sign-balanced} above the scale $r_{j}$ if
\begin{equation}
\label{eq:sign-bal def}
\lim\limits_{j\rightarrow \infty} \sup_{r \ge r_j} \Bc(f_{j};r) = 0.
\end{equation}
\end{enumerate}
\end{definition}

\vspace{2mm}
Our motivating question is whether, in a generic scenario, the Laplace eigenfunctions $\varphi_{j}$ are sign-balanced in the sense of \eqref{eq:sign-bal def} above suitably chosen scales $r_{j} \gg \frac{1}{\lambda_{j}}$. Our main results assert that, in accordance with the discussion in \S~\ref{sec:intro mot}, generic Laplace eigenfunctions are sign-balanced at sufficiently large scales $r \gg 1/\lambda_j$ with almost full probability, in a sense to be made precise. Moreover we identify a second scale $r  \gg 1/\lambda_j$ below which the sign-balance fails decisively. Somewhat surprisingly, both these scales are {\em above} the Planck scale.

\subsubsection{{\bf Random spherical harmonics}}

Our first result concerns the Gaussian ensembles of random spherical harmonics on spheres $\Sc^{d}$, $d\ge 2$. Let us first define this ensemble in the classical setting of $d=2$. Recall that, for $\ell \ge 1$, the space $\Ec_{\ell}$ of spherical harmonics of degree $\ell$, of dimension $\dim\Ec_{\ell} = 2\ell+1$, are the functions $f$ obeying the equation
\begin{equation*}
\Delta f + \ell(\ell+1)f =0,
\end{equation*}
i.e. satisfy \eqref{eq:Helmholts} with $$\lambda=\lambda_{\ell} = \sqrt{\ell\cdot (\ell+1)}.$$ The space $\Ec_{\ell}$ admits a canonical $L^{2}$-orthonormal basis of functions consisting of the Laplace spherical harmonics $\{Y_{\ell,m}\}_{-\ell\le m\le \ell}$. The \textit{random spherical harmonics} $H_{\ell}:\Sc^{2}\rightarrow \R$ of degree $\ell\ge 1$ are the centred Gaussian random fields
\begin{equation}
\label{eq:Tl spher harm d=2 def}
H_{\ell}(x)=H_{2;\ell}(x) = \frac{\sqrt{4\pi}}{\sqrt{2\ell+1}} \sum\limits_{m=-\ell}^{\ell} a_{m}Y_{\ell,m}(x), \quad  x\in\Sc^{2} ,
\end{equation}
where the $a_{m}$ are standard Gaussian i.i.d.\ random variables. Alternatively, the law of $H_{\ell}$ is uniquely prescribed via its covariance kernel (also the reproducing kernel of $\Ec_{\ell}$ as a Hilbert space)
\begin{equation}
\label{eq:covar spher Legendre}
K_{\ell}(x,y) :=\E\left[H_{\ell}(x)\cdot H_{\ell}(y)\right] = P_{\ell}(\cos(d(x,y))), \quad  x,y\in\Sc^{2},
\end{equation} 
where $d(\cdot,\cdot)$ is the spherical distance and $P_{\ell}$ is the Legendre polynomial of degree $\ell$. Since the law of $H_{\ell}$ is invariant under rotations, the Laplace 
spherical harmonics $\{Y_{\ell,m}\}_{-\ell\le m\le \ell}$ in \eqref{eq:Tl spher harm d=2 def} could be replaced by an arbitrary $L^{2}$-orthonormal basis of $\Ec_{\ell}$.

\vspace{2mm}
The definition \eqref{eq:Tl spher harm d=2 def} can be extended to random (ultra-)spherical harmonics on spheres $\Sc^{d}$ of arbitrary dimensions $d\ge 2$: 
\begin{equation}
\label{eq:Tl spher harm d>2 def}
H_{\ell}(x)=H_{d;\ell}(x) = \frac{\big|\sqrt{\Sc^{d}}\big|}{\sqrt{n_{\ell}}} \sum\limits_{m=-\ell}^{\ell} a_{m}Y_{\ell,m}(x), \quad x\in\Sc^{d},
\end{equation}
where the $a_{m}$ are as in \eqref{eq:Tl spher harm d=2 def}, $\{Y_{\ell,m}\}_{1\le m\le n_{\ell}}$ are the Laplace spherical harmonics on $\Sc^{d}$, and
\begin{equation}
\label{eq:n dim spher harm def}
n_{\ell}=n_{d,\ell}=\dim{\Ec_{d;\ell}} = {\ell+d-1 \choose d-1} + {\ell+d-2 \choose d-1}  = \frac{2}{(d-1)!}\ell^{d-1}+O(\ell^{d-2}), 
\end{equation}
is the dimension of the space $\Ec_{d;\ell}$ of spherical harmonics of degree $\ell$. Alternatively, $H_{\ell}$ is the centred Gaussian random field on $\Sc^{d}$ with the covariance (reproducing) kernel 
\begin{equation*}
\begin{split}
K_{\ell}(x,y)=K_{d;\ell}(x,y) :=\E\left[H_{d;\ell}(x)\cdot H_{d;\ell}(y)\right] &= G_{\alpha;\ell}(\cos(d(x,y))) = \frac{P_{\ell}^{\alpha-1/2,\alpha-1/2}(\cos(d(x,y)))}{{\ell+\frac{d}{2}-1\choose \ell}},
\end{split}
\end{equation*}
where, as above, $d(\cdot,\cdot)$ is the spherical distance, $\alpha= \frac{d-1}{2}$, $G_{d;\ell}$ is the normalised Gegenbauer (ultra-spherical) polynomial of degree $\ell$, and $P_{\ell}^{\alpha,\beta}$ is the Jacobi polynomial of degree $\ell$ .  The law of $H_{\ell}$ is invariant with respect to the rotations of the sphere, and in particular 
\begin{equation}
\label{eq:var==1 spher harm}
\var(H_{\ell}(x))=K_{\ell}(x,x) = G_{\alpha;\ell}(1) \equiv 1 
\end{equation}
due to the implemented normalisation.

\subsubsection{{\bf Statement of principal result I: Sign-balance for random spherical harmonics}}

Recall the definition of sign-imbalance in \eqref{eq:sign-imbalance def}. For every $\ell\ge 1$, $0< r < \pi$, the sign-imbalance of the random spherical harmonics $H_{\ell}$ at scale $r$ is the (real-valued) random variable 
\begin{equation}
\label{eq:sign-imbal spher harm}
\Bc_{\ell}(r)=\Bc_{\ell;d}(r) := \Bc(H_{d;\ell};r). 
\end{equation}
Our first principal result determines a scale $\bar{r}= \bar{r}_{\ell} \gg \frac{1}{\ell}$, such that $\Bc_{\ell}(r)$ vanishes in probability as $\ell\rightarrow \infty$, uniformly above $\bar{r}$. On the other hand, for a certain smaller scale $\frac{1}{\ell} \ll \ubar{r}_\ell \ll \bar{r}_\ell$, the random variables $\left\{\Bc_{\ell}(\ubar{r}_\ell)\right\}$ are shown to be bounded away from zero in the following strong sense:

\begin{definition}[Random variables bounded away from zero]
\label{def:bnd away zero}
A sequence of non-negative random variables $\left\{X_{n}\right\}$ 
is \textit{bounded away from zero in probability} if there exists a number $\epsilon_{0}>0$ such that 
\[ \lim\limits_{n\rightarrow \infty} \prob(X_{n} > \epsilon_0) = 1. \]
\end{definition}

\begin{theorem}
\label{thm:sign-bal spher harm}
Let $d\ge 2$. For $\ell\ge 1$ and $0<r<\pi$ define $\Bc_{\ell}(r)$ as in \eqref{eq:sign-imbal spher harm}, and set 
\begin{equation}
\label{eq:rl cric rad spher harm}
 \ubar{r}_{\ell} = \frac{(\log{\ell})^{\frac{1}{2(d-1)}}}{\ell} \quad \text{and} \quad \bar{r}_{\ell} = \frac{(\log{\ell})^{\frac{1}{d-1}}}{\ell}  .
\end{equation}
Then the following hold:

\begin{enumerate}[i.]
\item There exists a number $\mu>0$ sufficiently large, only depending on $d$, so that, as $\ell\rightarrow \infty$, 
\begin{equation}
\label{eq:sign-imbal->0}
\sup_{r \ge \mu \cdot \bar{r}_\ell}  \Bc_{\ell }(r) \xrightarrow{P} 0.
\end{equation}

\item There exists a number $\mu>0$ sufficiently small, only depending on $d$, so that the sequence 
of random variables $$\left\{\Bc_{\ell}(\mu\cdot \ubar{r}_{\ell})\right\}_{\ell \ge 1}$$ is bounded away from zero in probability.
\end{enumerate}
\end{theorem}

Theorem \ref{thm:sign-bal spher harm}(i.) asserts that the random spherical harmonics are sign-balanced, with almost full probability, above the scale $\bar{r}_{\ell}$ in \eqref{eq:rl cric rad spher harm}, and are {\em not} sign-balanced, with almost full probability, at and below the scale $\ubar{r}_{\ell}$; both scales are a logarithm power above the Planck scale. In fact, from a quantitative version of \eqref{eq:sign-imbal->0} obtained in the proof of Theorem \ref{thm:sign-bal spher harm}, and the Borel-Cantelli lemma, we can deduce that the same holds {\em almost surely} for sequences of random spherical harmonics, see Corollary \ref{cor:as} below for a (more general) result of this nature. As a concrete application, for $d=2$ the relevant scales are:
\[   \ubar{r}_{\ell} = \frac{\left(\log{\ell}\right)^{1/2}}{\ell} \quad \text{and} \quad \bar{r}_{\ell} = \frac{\log{\ell}}{\ell} . \] 
See a further discussion on the significance of this result in \S~\ref{sec:back sig} below. 

\subsection{Principal result II: Volume-balance for random waves on manifolds}

Our second principal result generalises Theorem \ref{thm:sign-bal spher harm} in two directions: to a wider class of Gaussian ensemble of random waves defined on general smooth manifolds, and by extending the notion of sign-balance to the {\em volume-balance} at arbitrary levels. Remarkably, for non-zero levels, we identify the {\em precise} scale, coinciding with $\bar{r}_{\ell}$ in \eqref{eq:rl cric rad spher harm}, at which volume-balance occurs with high probability. This is in contrast to the sign-balance, where there is a slight lacuna between the upper and the lower scales \eqref{eq:rl cric rad spher harm}.

\subsubsection{{\bf Random waves on smooth manifolds}}
\label{sec:rand waves}

Recall that $(\Mcc,g)$ is a closed Riemannian $d$-manifold, and the $\varphi_{j}$ are the Laplace eigenfunctions on $\Mcc$ corresponding to eigenvalues $\lambda_{j}$, satisfying the Helmholtz equation \eqref{eq:Helmholts}. Our results in Theorem \ref{thm:sign-bal spher harm} for the round sphere $\Mcc = \Sc^d$ benefited from the high spectral degeneracy of the Laplace operator on $\Sc^{d}$, which arises due to the rotational symmetry. Since this is not the case for generic manifolds $\Mcc$, we will instead superimpose Laplace eigenfunctions corresponding to different eigenvalues lying in an \textit{energy window}. If the energy window is suitably-chosen -- narrow enough to contain comparable frequencies, but large enough for the superposition to be sufficiently chaotic -- it is reasonable to expect the corresponding random linear combinations to represent `generic' Laplace eigenfunctions on~$\Mcc$.

\vspace{2mm}
Let $T$ be a (large) spectral parameter, and $\eta=\eta(T) \in (0, T]$ a band-width. The \textit{band-limited random Gaussian function (`random wave')} is the centred Gaussian random field
\begin{equation}
\label{eq:fTeta band-lim def}
f_{T}(x)=f_{T,\eta}(x) = \frac{\sqrt{|\Mcc|}}{\sqrt{N}}\sum\limits_{\lambda_{j}\in [T-\eta,T]} a_{j}\varphi_{j}(x), \quad x\in \Mcc,
\end{equation}
where the $a_{j}$ are standard Gaussian i.i.d. random variables, and 
\begin{equation}
\label{eq:N(T,eta) def}
N=N(T,\eta) = \#\{\lambda_{j}\in [T-\eta,T]\} = N(T)-N(T-\eta)
\end{equation}
is the number of energy levels inside the energy window $[T-\eta,T]$, with 
\begin{equation}
\label{eq:N(T) spec func}
N(T):= \#\{ \lambda_{j}\le T\}
\end{equation}
the spectral function of $\Mcc$. (We tacitly assume that $T-\eta$ is not an energy level of $\Mcc$.)
The normalising pre-factor in the definition \eqref{eq:fTeta band-lim def} is chosen so that 
\begin{equation}
\label{eq:var int Mcc}
\int\limits_{\Mcc}\var(f_{T}(x))dx = |\Mcc| .
\end{equation}
 In the regime $\eta \to \infty$, this further ensures that, as $T\rightarrow \infty$,
\[ \var(f_{T}(x))\sim 1 \]
 uniformly in $x \in \Mcc$ (see Proposition \ref{prop:asymp covar waves}). We stress that spectral degeneracies of $\Mcc$ are allowed, in which case the law of $f_{T}$ is invariant w.r.t.\ the choice of an orthonormal basis $\{\varphi_{j}\}_{j\ge 1}$ of $L^{2}(\Mcc)$. For the special case of the round sphere $\Mcc=\Sc^{d}$ we allow $\eta \equiv 1$, whence, by the conventions of \S~\ref{sec:covar kern}, for the choice
\[ T=\sqrt{\lambda_{\ell}}=\sqrt{\ell\cdot (\ell+d-1)}\sim \ell \]
 (otherwise the summation on the r.h.s.\ of \eqref{eq:fTeta band-lim def} is empty), the random waves in \eqref{eq:fTeta band-lim def} coincides with the random spherical harmonics \eqref{eq:Tl spher harm d>2 def}: $$f_{\ell,1}(x) \equiv  H_{\ell}(x).$$ More generally, for $\Mcc=\Sc^{d}$ the random waves \eqref{eq:fTeta band-lim def} are a superposition of the random spherical harmonics $H_{\ell'}$ with $\ell'$ in the given 
energy window (essentially $[\ell-\eta+1,\ell]$ with $\ell\approx T$) with weights prescribed by the normalising factors $\frac{1}{\sqrt{n_{\ell'}}}$ 
in \eqref{eq:Tl spher harm d>2 def}, see \eqref{eq:band-lim sphere} below.

\subsubsection{{\bf Volume-balance at arbitrary levels}}

\begin{definition}[Volume-balanced sequences of functions]
\label{def:sign-balance lev}

Let $\Mcc$ be a smooth manifold, $f:\Mcc\rightarrow\R$ a real-valued measurable function, $u \in \R$ a (fixed) level, and recall that $H(\cdot)$ is the sign function \eqref{eq:H Heaviside}.

\begin{enumerate}[i.] 

\item For $x\in \Mcc$ and $0 < r < \inj(\Mcc)$, the \textit{volume-bias} of $f$ at level $u$, restricted to $B_{r}(x)$, is 
\begin{equation}
\label{eq:defect def lev}
\Dc_{u}(x;r)=\Dc_{f;u}(x;r) := \frac{1}{|B_{r}(x)|}\int\limits_{B_{r}(x)} \left(H(f(y) - u) -  \tau(u)\right)  dy \in  [-2,2],
\end{equation}
with 
\begin{equation}
\label{eq:tau mean bias def}
\tau(u):=1 - 2 \Phi(u)= \E[H(Z-u)],
\end{equation}
where $Z$ a standard Gaussian random variable and $\Phi$ is the Gaussian cdf. We denote
\begin{equation}
\label{eq:uncentr vol bias def}
\widetilde{\Dc}_u(x;r)=\widetilde{\Dc}_{f;u}(x;r) := \frac{1}{|B_{r}(x)|}\int\limits_{B_{r}(x)} H(f(y) - u) dy \in  [-1,1],
\end{equation}
the uncentred variant of $\Dc_u(x;r)$.

\item Let $r_{j}\subseteq \R_{>0}$ be a sequence of positive numbers, and $f_{j}:\Mcc\rightarrow\R$ a sequence of measurable functions. The \textit{volume-imbalance} of $f_{j}$ at scale $r_{j}$ with respect to the level $u$ is
\begin{equation}
\label{eq:sign-imbalance def lev}
\Bc_u(f_{j};r_{j}):= \sup\limits_{\substack{x\in\Mcc}}|\Dc_u(x;r_{j})| \in [0,2].
\end{equation}
We say that the $f_{j}(\cdot)$ are \textit{volume-balanced} above the scale $r_{j}$ with respect to level $u$ if
\begin{equation}
\label{eq:sign-bal def lev}
\lim\limits_{j\rightarrow \infty} \sup_{r \ge r_j}  \Bc_u(f_{j};r) = 0.
\end{equation}
\end{enumerate}
\end{definition}

When applied to a centred Gaussian random field $f:\Mcc\rightarrow\R$, the term $\tau(u) $ in \eqref{eq:defect def lev} is the natural centering, so that $\E[\Dc_u(x;r)]\equiv 0$.  We observe that the volume-bias \eqref{eq:defect def lev} (resp.\ volume-imbalance \eqref{eq:sign-imbalance def lev}) at level $u=0$ coincides with the defect \eqref{eq:defect def} (resp.\ sign-imbalance \eqref{eq:sign-imbalance def}).

\subsubsection{{\bf Statement of principal result II: Volume-balance for random waves on manifolds}}

For $T>0$, $\eta \in (0,T]$, $0<r<\inj(\Mcc)$, and $u \in \R$, we define the \textit{volume-imbalance} of the random wave $f_{T,\eta}$ at level $u$ to be
\begin{equation}
\label{eq:sign-imbal def waves}
\Bc_{T,u}(r)=\Bc_{T,u;\eta}(r) := \Bc_u(f_{T,\eta};r),
\end{equation}
cf.\ \eqref{eq:sign-imbal spher harm}. As announced, our next result determines a scale $r =\bar{r}_T$ above which the volume-imbalance $\Bc_{T,u}(r)$ of $f_{T,\eta}$ at arbitrary level vanishes in probability as $T \to \infty$, and shows that the scale is optimal for non-zero levels $u \neq 0$.

\begin{theorem}
\label{thm:sign bal rand wav}
Let $(\Mcc,g)$ be a smooth Riemannian $d$-manifold, and $\eta=\eta(T)\in (0,T]$ satisfying either (a) $\eta(T)\rightarrow \infty$ or (b) $\Mcc=\Sc^{d}$. Let $u \in \R$, and define $\Bc_{T,u}(r)$ as in \eqref{eq:sign-imbal def waves}. Set 
\begin{equation}
\label{eq:rT crit rad rand wav}
\bar{r}_{T} =  \frac{\min\left\{ (\log{T})^{\frac{1}{d-1}} , \left(\frac{T\log{T}}{\eta}\right)^{\frac{1}{d}}  \right\}}{T} \quad \quad  \ubar{r}_{T} =  \frac{\min\left\{ (\log{T})^{\frac{1}{2(d-1)}} , \left(\frac{T\log{T}}{\eta}\right)^{\frac{1}{2d}}  \right\}}{T}  .
\end{equation}  
Then, as $T \to \infty$, the following hold:
\begin{enumerate}[i.]
\item There exists a number $\mu>0$ sufficiently large, only depending on $\Mcc$, so that
\begin{equation}
\label{eq:balance -> 0 prob}
 \sup_{r \ge \mu \cdot \bar{r}_{T} }  \Bc_{T,u}(r) \xrightarrow{P} 0.
\end{equation}
\item Assume that either (a') there exists a number $\delta_{0}>0$ so that $\eta(T)>T^{\delta_{0}}$ or (b') $\Mcc=\Sc^{d}$. Then there exists a number $\mu>0$ sufficiently small, only depending on $\Mcc$, so that:
\begin{itemize}
\item If $u \neq 0$, $\left\{\Bc_{T,u}(\mu\cdot \bar{r}_{T})\right\}_{T\ge 1}$ is bounded away from zero in probability.
\item If $u = 0$, $\left\{\Bc_{T,u}(\mu\cdot \ubar{r}_{T})\right\}_{T\ge 1}$ is bounded away from zero in probability.
\end{itemize}
\end{enumerate}
\end{theorem}

\vspace{2mm}

As mentioned in \S\ref{sec:rand waves}, Theorem \ref{thm:sign bal rand wav} allows for $\Mcc=\Sc^{d}$ and $\eta \equiv 1$, and so contains Theorem \ref{thm:sign-bal spher harm} as a particular case. See Corollary \ref{cor:as} below for an almost sure version of Theorem \ref{thm:sign bal rand wav} (i.).
\vspace{2mm}

If $\eta$ satisfies $\eta(T)=o(T)$, the random waves \eqref{eq:fTeta band-lim def} are called {\em monochromatic}, closest to `pure' eigenstates. Assuming in addition that $\eta(T) =o\big(T \cdot (\log T)^{-\frac{1}{d-1}}\big)$, the scales \eqref{eq:rT crit rad rand wav} in the monochromatic regime are 
\begin{equation}
\label{eq:rT crit monochrome}
\bar{r}_{T} = \frac{(\log{T})^{\frac{1}{d-1}}}{T} \quad \text{and}\quad \ubar{r}_{T} = \frac{(\log{T})^{\frac{1}{2(d-1)}}}{T},
\end{equation}
coinciding with the scales \eqref{eq:rl cric rad spher harm} for random spherical harmonics (with $T\sim \ell$). 

\vspace{2mm}
In general, the scales \eqref{eq:rT crit rad rand wav} {\em interpolate} between the monochromatic scales \eqref{eq:rT crit monochrome} and the scales 
\begin{equation*}
\bar{r}_{T}\asymp \frac{(\log{T})^{\frac{1}{d}}}{T} \quad \quad \ubar{r}_{T}\asymp \frac{(\log{T})^{\frac{1}{2d}}}{T}
\end{equation*} 
which are relevant in the regime of positively-banded waves satisfying $\eta(T) = c_{0}\cdot T$ with $0<c_{0}\le 1$ constant (alternatively $c_{0}=c_{0}(T)$ bounded away from $0$). The {\em crossover} between these two regimes occurs for {\em barely monochromatic} waves for which $\eta(T)  \asymp  T (\log{T})^{-\gamma}$ for some $0<\gamma < \frac{1}{d-1}$, where one has
\begin{equation*}
\bar{r}_{T} \asymp \frac{(\log{T})^{\frac{\gamma+1}{d}}}{T} \quad \quad \ubar{r}_{T} \asymp \frac{(\log{T})^{\frac{\gamma+1}{2d}}}{T}.
\end{equation*}
Equivalently to \eqref{eq:rT crit rad rand wav}, the scales $\bar{r}_{T}$ and $\ubar{r}_T$ may be defined implicitly as the respective solutions to the equations 
\begin{equation}
\label{eq:bar(r) implicit equation}
(\bar{r}_{T}T)^{d-1} \max\{1, \bar{r}_{T}\cdot \eta  \} = \log{T} \quad \quad  (\ubar{r}_{T}T)^{2(d-1)} \max\{1, \ubar{r}_{T}^2\cdot T \eta  \} = \log{T} ,
\end{equation} 
more natural in the proofs (see, e.g., \eqref{e:dev}). They are related by $\bar{r}_{T} \cdot T = (\ubar{r}_{T}\cdot T)^2$.

\subsection{Volume-bias concentration}

Theorem \ref{thm:sign bal rand wav} (and Theorem \ref{thm:sign-bal spher harm} as a particular case of Theorem \ref{thm:sign bal rand wav}) will follow from the concentration of the volume-bias of the restriction of $f_{T}(\cdot)$ 
to geodesic balls centred at a {\em fixed} point $x\in \Mcc$ of arbitrary radius. We believe these concentration results to be of independent interest. 

\vspace{2mm}
Recall that $(\Mcc,g)$ is a closed Riemannian $d$-manifold, and $\Dc_{T,u}(x;r)$ is the volume-bias defined in \eqref{eq:defect def lev} for the random wave $f_{T,\eta}(\cdot)$. In particular $\Dc_{T}(x;r)  = \Dc_{T,0}(x;r)$ is the defect. First, we present an upper bound for the defect concentration that will be used to infer Theorem \ref{thm:sign bal rand wav}(i.):

\begin{theorem}[Volume-bias concentration upper bound]
\label{t:cub}
There exists a number $C=C(\Mcc)>0$ sufficiently large, and for every $\eps >0$ a number $c=c(\Mcc,\eps) > 0$ sufficiently small, such that the following holds. For all $T \ge 1$ and $r \ge 1/T$ such that either (a) $\eta \in [C,T]$ or (b) $\Mcc = \Sc^{d}$ and $\eta \in [1,T]$, 
one has:
\begin{equation}
\label{eq:conc upper bnd}
\Pb( | \Dc_{T,u}(x;r) | > \eps ) <  e^{- c (r T)^{d-1} \max\{1, r \eta\}  } .
\end{equation}
uniformly w.r.t.\ $x\in\Mcc$ and $u \in \R$.
\end{theorem} 

To infer Theorem \ref{thm:sign bal rand wav}(ii.) we need the corresponding {\em lower} bound for the defect concentration. It will be more natural to work with the {\em uncentred} volume-bias \eqref{eq:uncentr vol bias def},
one benefit of which is that $u \mapsto \widetilde{\Dc}_{f,u}(x;r)$ is non-increasing. We infer Theorem \ref{thm:sign bal rand wav}(ii.) in the case of non-zero level $u \neq 0$ from the following bound:

\begin{theorem}[Volume-bias concentration lower bound I]
\label{t:clb}
For every $\eps > 0$ there exist a number $C = C(\Mcc,\eps) > 0$ sufficiently large such that, for all $T \ge 1$, $r \ge \frac{C}{T}$, such that either (a) $ \min\left\{ \eta, \frac{1}{r} \right\} > C \cdot (rT)^{(d-1)/2}$ or (b) $\Mcc = \Sc^d$ and $r < 1/C$, one has:
\[   \Pb(  \widetilde{\Dc}_{T,u}(x;r)  > -\eps   ) >   e^{- C (1+u^2) (r T)^{d-1}  \max\{1, r \eta\}  } . \]
 uniformly w.r.t.\ $x\in\Mcc$ and $u \in \R$.
 \end{theorem}

Observe that, for $u > 0$, $\E[\widetilde{\Dc}_{f;u}(x;r)]=\tau(u) > 0$. Hence, one could choose $\eps \in (0, \tau(u)) $ sufficiently small and apply Theorem \ref{t:clb} to obtain a lower bound on $\Pb( |\Dc_{T,u}(x;r)| > \eps )$ that matches the order of the upper bound in Theorem \ref{t:cub}. By contrast, for the defect ($u=0$), one has $\tau(u) = 0$, and thus Theorem \ref{t:clb} only asserts a lower bound on $\Pb( |\Dc_{T,u}(x;r)| < \eps )$, useless for the purpose of proving Theorem \ref{thm:sign bal rand wav}(ii.). To handle this regime we establish the following bound which is of smaller order and applies to a narrower range of parameters:

\begin{theorem}[Volume-bias concentration lower bound II]
\label{t:clb2}
There is an absolute constant $\eps > 0$ and a number $C = C(\Mcc) > 0$ so that, for all $T \ge 1$, $r \ge \frac{C}{T}$, such that either (a) $ \min\{ \eta, 1/  (r^2 T) \} > C (rT)^{d-1}$ or (b) $\Mcc = \Sc^d$ and $r < 1/ ( C \sqrt{T})$, one has:
\[   \Pb( \widetilde{\Dc}_{T,u}(x;r)   > \eps   ) >  e^{- C (1+u^2) (r T)^{2(d-1)} \max\{ 1, r^2 T \eta \} }   \]
 uniformly w.r.t.\ $x\in\Mcc$ and $u \in \R$. 
 \end{theorem}

Although Theorem \ref{t:clb2} is applied below for the purpose of studying the defect, the statement for {\em positive} levels $u > 0$ (chosen arbitrarily as $u=1$ within the proof of Theorem \ref{thm:sign bal rand wav}(ii.)) is a stronger result than the same with $u=0$. The stronger result is required, as the level $u$ will be slightly adjusted within the proof of Theorem \ref{thm:sign bal rand wav}(ii.) to exploit a certain `sprinkled decoupling' technique.

\vspace{2mm}
Note that theorems \ref{t:clb} and \ref{t:clb2} are only effective for general manifolds on {\em mesoscopic} scales: in the case of Theorem \ref{t:clb}, on scales smaller than 
$$r = T^{-(d-1)/(d+1)} = T^{2/(d+1)} / T = o(1) , $$ 
and for Theorem \ref{t:clb2}, scales smaller than
 $$ r = T^{-d/(d+1)} = T^{1/(d+1)} / T = o(1) . $$
For the sphere, Theorem \ref{t:clb} is effective on {\em all} mesoscopic scales $r = o(1)$, whereas Theorem \ref{t:clb2} is only effective on {\em mesoscopic} scales $r =  o(T^{1/2}/T)$. 

\vspace{2mm}
While we are particularly interested in the defect, for the volume-bias at non-zero levels $u \neq 0$ theorems \ref{t:cub} and \ref{t:clb} together show that, uniformly over geodesic balls at mesoscopic scales, the volume-bias has (upper) large deviations of order 
\begin{equation}
\label{e:dev}
 - \log \Pb( \Dc_{T,u}(x;r)  > \eps ) \asymp  (rT)^{d-1} \max\{1, r\eta\} .
 \end{equation}
On general manifolds we prove this only for sufficiently small mesoscopic scales, whereas on the sphere this holds at \textit{all} mesoscopic scales.

\smallskip
For example, in the case of unit energy band $\eta \approx 1$ (e.g.\ the random spherical harmonics), the deviations of the volume-bias ($u \neq 0$) are of minimal order 
\[ - \log \Pb( \Dc_{T,u}(x;r)  ) \asymp   (r T)^{d-1} , \]
whereas for the positively-banded case $\eta \asymp T$, the deviations are of maximal order 
\[ - \log \Pb( \Dc_{T,u}(x;r) > \eps ) \asymp   (r T)^d  . \]
In the intermediate monochromatic case $\eta \rightarrow \infty$ but $\eta= o(T)$, the deviations exhibit crossover for radii exceeding $1/\eta$.

\subsubsection{{\bf Almost sure sign-balance}}
\label{sec:almost sure sign bal}

The proof of Theorem \ref{thm:sign bal rand wav}(i.) yields a rate of convergence for $\Bc_{T}\left(\mu \cdot r_{T}\right)$ in \eqref{eq:balance -> 0 prob} that is at least polynomial, of degree that can be made arbitrarily large by taking $\mu$ sufficiently large. Therefore a straightforward application of the Borel-Cantelli lemma allows for upgrading the convergence mode to almost sure convergence, regardless of how the $f_{T}$ are drawn for different $T$ (i.e.\ independently or not):

\begin{corollary}[Almost sure sign-balance]
\label{cor:as}
Recall that for $u \in \R$, the level-imbalance $\Bc_{T,u}(r)$ of the random waves is defined in \eqref{eq:sign-imbal def waves}. Assume that either (a) $\eta(T)\rightarrow \infty$ or (b) $\Mcc=\Sc^{d}$, and let $\bar{r}_T$ be given by \eqref{eq:rT crit rad rand wav}. Then if $\mu > 0$ is sufficiently large, as $T \to \infty$, one has 
\begin{equation}
\label{eq:vol imbal->0 a.s.}
\sup\limits_{r\ge\mu\cdot \bar{r}_{T}}\Bc_{T,u}(r) \to  0   \qquad \text{almost surely.} 
\end{equation}
\end{corollary}

In fact, by taking a sufficiently tight net $\{u_{j}\}\subseteq \R$, and using the union bound, one could even take a supremum w.r.t.\ the level, i.e.\ the supremum on the l.h.s.\ of \eqref{eq:vol imbal->0 a.s.} could be replaced by 
$\sup\limits_{\substack{r\ge\mu\cdot \bar{r}_{T},\, u\in\R}}\Bc_{T,u}(r) $. The proof of this latter claim is left to the reader.

\subsection{Outline of the rest of the paper}
In \S~\ref{s:outline} we provide more background on our results, and discuss some elements of the proof. In \S~\ref{sec:covar kern} we begin the proof of our results by performing an asymptotic analysis of the covariance (reproducing) kernel of the random waves. In \S~\ref{sec:conc bnd proof 1} we prove the upper bound for the defect concentration in Theorem \ref{t:cub}. In \S~\ref{sec:conc bnd proof} we prove the corresponding lower bounds for the concentration of the volume-bias in theorems \ref{t:clb}-\ref{t:clb2}. In \S~\ref{sec:main res proofs} we complete the proof of the principal result of the manuscript, namely Theorem \ref{thm:sign bal rand wav} (containing Theorem \ref{thm:sign-bal spher harm} as a particular case), as well as prove Corollary \ref{cor:as}. The appendix contains the proof of two results: Proposition \ref{prop:no scale increase} and the auxiliary Lemma~\ref{lem:defect derivatives}.

\subsection*{Conventions}

\begin{enumerate}[\textbullet]

\item All random variables are defined on a common probability space $(\Omega,\Fcc,\prob)$. For random variables $\{X_{n}\}_{n \in \mathbb{N}}$ and $X$, $X_{n} \xrightarrow{P} X$ denotes convergence in probability with respect to $\prob$. Alhough the convergence in probability of Theorem \ref{thm:sign bal rand wav}(i.) is with respect to a continuous parameter $T\rightarrow \infty$, it is locally constant, and only jumps if either $T$ or $T-\eta$ is an energy level for $\Mcc$, hence also falls within the scope of a discrete sequence of random variables.

\item For a closed manifold $(\Mcc,g)$, $\inj(\Mcc)$ will stand for the injectivity radius of $\Mcc$. Given $x\in\Mcc$ and $0<r<\inj(\Mcc)$, we denote $B_{r}(x) := \{ y\in\Mcc:\: d_{g}(x,y)<r\}$ to be the geodesic (metric) ball, where $d_{g}(\cdot,\cdot)$ is the geodesic metric on $\Mcc$. We stress that the radius is always assumed to be smaller than $\inj(\Mcc)$. For a domain $\Dc\subseteq \Mcc$ (or $\Dc\subseteq \R^{d}$) the notation 
$|\Dc|$ will designate the $d$-volume measure of $\Dc$ (resp.\ the Lebesgue measure of $\Dc$).

\item We reserve $c_{d} > 0$ to designate a dimensional constant that, in general, varies throughout the text.

\item Given two positive expressions $A,B$ depending on some parameter (e.g.\ functions of $x$), $A=O(B)$ and $A\ll B$ both mean that there exists some constant $C>0$ so that $A\le C\cdot B$, and `$\asymp$' means that both $\ll$ and $\gg$ hold, i.e.\ $\frac{1}{C}\cdot B<A<C\cdot B$ for $C>0$ sufficiently large.
\end{enumerate}

\subsection*{Acknowledgements}
We are grateful to A. Logunov for sharing the results of the forthcoming work \cite{LoNa}, and to L. Polterovich, P. Sarnak, M. Sodin, and N. Yesha for many stimulating conversations. We deeply appreciate the interest they have all shown in this research. S.M. is supported by the Australian Research Council Future Fellowship FT240100396.

\medskip

\section{Discussion}
\label{s:outline}

\subsection{{\bf Background and significance of the results}}
\label{sec:back sig}

\subsubsection{Limit theory for the defect and volume-bias}
The study of the defect and volume-bias of random Laplace eigenfunctions has previously focused mainly on their limit theory; see \cite{mw11a,mw11b,mw14} for random spherical harmonics on $\Sc^2$, \cite{MaRo15} for generalisations to $d \ge 2$, \cite{kwy21} for random toral eigenfunctions, and \cite{bgs02} for Euclidean waves.
To illustrate what is known, the defect of the random spherical harmonics $H_{\ell}$ in \eqref{eq:Tl spher harm d>2 def} (here $\Mcc = \Sc^d$ with $d \ge 2$) on geodesic balls satisfies the following \cite{mw11a,mw11b,MaRo15}: 
for every fixed $u \in \R$, $x \in \Mcc$, and any sequence of radii $r = r_T$ above Planck scale (that is, satisfying $r \cdot T \to \infty$ as $T \to \infty$), it holds that, as $T \to \infty$
\begin{equation}
\label{e:var}
 \textrm{Var}( \Dc_{\ell,u}(x;r) ) \sim \begin{cases} 
 \frac{c_{d,u}}{(r\ell)^{d-1}} & u \neq 0 ,\\
 \frac{c'_{d,u}}{(r\ell)^{d} }  & u = 0 , \end{cases} 
 \end{equation}
for positive constants $c_{d,u}, c'_{d,u} > 0$, and moreover
\[ \frac{\Dc_{\ell,u}(x;r) }{\sqrt{\textrm{Var}( \Dc_{\ell,u}(x;r) )}} \stackrel{d}{\Longrightarrow} \mathcal{N}(0,1) . \]

One interesting feature of \eqref{e:var} is the phenomenon of \textit{variance reduction} at the nodal level ($u=0$), which can be attributed to the (anti-)symmetry of the sign function, and is analogous to the \textit{Berry cancellation} known to occur for other geometric functionals such as length of the nodal set~\cite{ber02,WigCMP}. From \eqref{e:var} one may infer that, with almost full probability, the volume-bias $\Dc_{\ell,u}(x;r)$ is small around {\em almost} all points $x \in \Mcc$ at {\em any} radius $r$ above Planck scale. More precisely, one can deduce that, for every $\eps>0$,
\begin{equation}
\label{eq:balance almost all}
\lim\limits_{R\rightarrow \infty}\limsup\limits_{\ell\rightarrow\infty} \prob(| \{x\in \Sc^{2}:\: |\Dc_{\ell,u}(x; R \cdot \ell^{-1})| > \eps   \}| > \eps ) = 0.
\end{equation}
However one cannot deduce from \eqref{e:var} that $\Dc_{\ell,u}(x;r) $ is \textit{uniformly} small over all $x \in \Mcc$ with high probability, as required by our definition of sign-balance. Indeed our results demonstrate that in order for the defect to be \textit{uniformly} small, the radius $r = r_\ell$ must be taken to be a logarithm power of the energy above Planck scale. The validity of \eqref{eq:balance almost all} only requires the decay of the covariance kernel, and, accordingly, could be extended to the random waves \eqref{eq:fTeta band-lim def} on smooth manifolds, under very mild assumptions on $\eta$.

\subsubsection{Uniform $L^2$-mass equidistribution}
In terms of \textit{uniform} properties of random Laplace eigenfunctions comparable to our results, the only statistic that has been addressed so far in the literature is the $L^2$-mass equidistribution, motivated by the quantum ergodicity (QE) premise.  Berry~\cite{Berry77} suggested that mass equidistribution should hold for deterministic Laplace eigenfunctions on small scales, i.e.\ every geodesic ball $B_{r}(x)\subseteq \Mcc$ and sequence of radii $r=r_{\lambda_{j}}$ satisfying  $r  \lambda_{j}\rightarrow \infty$ should satisfy
\[ \int\limits_{B_{r}(x)}\varphi_{j}(y)^{2}dy \rightarrow \frac{|B_{r}(x)|}{|\Mcc|} \]
uniformly w.r.t.\ $x\in\Mcc$. For \textit{random} Laplace eigenfunctions, uniform mass equidistribution has been proven to hold with high probability on scales a logarithm power above Planck scale (see \cite{di17} for random spherical harmonics, ~\cite{di21,ht20} for generalisations to manifolds), which is analogous to our results.

\vspace{2mm}

Compared to \cite{di17,di21,ht20}, the results presented in this manuscript are more general in terms of the permitted energy width $\eta$ on general manifolds, and they also establish lower bounds on the radius at which the balance occurs. We believe that an adaptation of our techniques could refine the said results, likely yielding an optimal lower bound, in line with our result for the volume-balance at non-zero levels in Theorem \ref{thm:sign bal rand wav}. We emphasise that the defect and level-balance have some extra challenging aspects compared to the $L^{2}$-mass distribution, not least since these functionals fail to be Lipschitz (in any appropriate function space) due to the discontinuity of the sign function at the origin.

\subsubsection{Deterministic sign-balance}

Our results lead to a natural conjecture on the sign-balance of  {\em deterministic} Laplace eigenfunctions $\varphi_{j}$ at least in the context of a generic scenario (cf.\ the results in ~\cite{NaPoSoAJM05,LoNa} explained in \S~\ref{sec:intro mot} above):

\begin{conjecture}
\label{conj:optimal balance}
Let $\Mcc$ be a chaotic smooth compact $d$-manifold, and $\{(\varphi_{j},\lambda_{j})\}_{j\ge 1}$ the corresponding sequence of Laplace eigenfunctions and eigenvalues. Then, along a density-$1$ subsequence, the $\varphi_{j}$ are sign-balanced  
above the scale $$\frac{(\log{\lambda_{j}})^{\frac{1}{d-1}+o(1)}}{\lambda_{j}}.$$
\end{conjecture}

\vspace{0.2cm}
Interestingly, it has been shown that `flat' toral eigenfunctions satisfy the strongest notion of sign-balance at optimal scales. These are the eigenfunctions on the standard flat $d$-torus $\Mcc=\R^{d}/\Z^{d}$ whose Fourier coefficients are bounded, or at most grow slowly, a setting which benefits from certain {\em number theoretic} methods. Lester-Rudnick ~\cite{LeRu} showed that the deterministic toral eigenfunctions (in $2d$) satisfy the small-scale QE ansatz, all the way down to the Planck scale (up to a sub-polynomial factor), subsequently refined in ~\cite{GrWi}. In~\cite[Theorem 1.1]{kwy21} it was proven that, for a density-$1$ sequence of energies $\lambda_j \in E$, deterministic `flat' eigenfunctions  have small defect around {\em almost} all points at {\em any} radius above Planck scale, in the sense that, for every $\eps > 0$,
\begin{equation*}
\lim\limits_{R\rightarrow \infty}\limsup\limits_{\lambda_j \in E, j \rightarrow \infty} | \{x\in \Sc^{2}:\: |\Dc_{\varphi_j}(x; R \cdot T^{-1})| > \eps   \}|  = 0.
\end{equation*}
Sartori~\cite{Sartori} showed how to modify the argument of ~\cite{kwy21} to upgrade this to \textit{all} points, i.e.\ showing that flat eigenfunctions are sign-balanced above the Planck scale. Since random toral eigenfunctions (`arithmetic random waves') are `flat' with almost full probability, this implies that, for `generic' sequences of energies, one can establish the analogue of Theorem \ref{thm:sign-bal spher harm} for the arithmetic random waves without the extra logarithmic factor (i.e.\ with $\bar{r}_T$ replaced by $R/\lambda_j$ with $R\rightarrow \infty$). 

\vspace{2mm}

Without assuming that $\Mcc$ is chaotic, the statement of Conjecture \ref{conj:optimal balance} fails e.g. for the torus at {\em macroscopic scales} ~\cite[Theorem 1.2]{kwy21}, 
and for the sphere at scale $\frac{1}{\sqrt{\lambda_{j}}}$, see Proposition \ref{prop:sign-bar exist man} below.

\subsection{{\bf Further directions}}

We discuss directions for further research motivated by our results:

\begin{enumerate}[(a)]

\item {\bf Subject 1: The phase transition for uniformity:} 
What is the true scale at which sign-balance occurs for random Laplace eigenfunctions? In other words, does there exist a scale $r_{T}$, so that, in the context of Theorem \ref{thm:sign bal rand wav} at level $u=0$,
\begin{itemize}
\item If $\mu$ is sufficiently large, $\Bc_{T}\left(\mu \cdot r_{T} \right)  \to  0$ in probability,
\item If $\mu$ is sufficiently small, $\Bc_{T}\left(\mu \cdot r_{T} \right)$ is bounded away from zero in probability?
\end{itemize}

\noindent Our results show that such a scale $r_{T}$ must satisfy $\ubar{r}_T \le r_T \le \bar{r}_T$ 
with $\ubar{r}_T ,\bar{r}_T$ given by \eqref{eq:rT crit rad rand wav}. We believe that $r_{T}=  o(\bar{r}_{T})$ is plausible, and also that a stronger concentration than proven in Theorem \ref{t:cub} holds for the defect compared to non-zero levels, by analogy with the variance reduction that is known to occur for the sign defect \eqref{e:var}. 

Provided that the scale $r_{T}$ is found, is there a {\em precise} phase transition at this scale, i.e.\ does there is a second scale $\omega_{T} = o(r_T)$ so that the above conclusions hold for the scale $\mu_{0}\cdot r_{T}+\gamma\cdot \omega_{T}$, with $\mu_{0}>0$ a fixed constant, and $\gamma$ sufficiently large or sufficiently small respectively?

\vspace{2mm}

\item {\bf Subject 2: Bridging the lacunas:}  There are various technical conditions in our results which may not be optimal, most notably the statement of Theorem \ref{thm:sign bal rand wav} does not apply in case $\eta$ stays bounded, and Theorem \ref{t:clb} imposes a stronger condition on $r$ than merely $r=o(1)$. It would be desirable to lift these, at least under some dynamical conditions on the manifold.
\vspace{2mm}

\item {\bf Subject 3: Concentration lower bounds for a larger deviation:} The concentration lower bound in Theorem \ref{t:clb2} assumes that the deviation $\eps>0$  is taken sufficiently small.  Does there exist a number $\eps_0 < 1$ such that the bounds in Theorem \ref{t:clb} no longer hold with $\eps > \eps_0$? If so, what is the correct order of deviations for such $\eps$?

\vspace{2mm}
\item {\bf Subject 4: Generalised notion of balance:} Our techniques are robust, and applicable to a wider class of functionals of random waves (or even general Gaussian random fields). One may consider a generalised notion of balance, given by some function $G:\R\rightarrow\R$ (smooth or not) in place of the sign function as a building block of the defect \eqref{eq:defect def}, and in turn, the corresponding notion of (im)balance \eqref{eq:sign-imbalance def}.  As a concrete example, one may take $G(t)=t^{2}$, associated to the $L^{2}$-mass equidistribution. Broadening slightly to functions $G: \R^k \rightarrow \R$, $k \ge 1$, one could also consider for instance the nodal volume of $f$ restricted to $B_{r}(x)$ (which can be expressed as an integral over a functional $G$ of $(f,\nabla f)$). For a general class of $G$, what is the critical scale above which the relevant ensembles of random fields are balanced, and below which they are not balanced? What properties of $G$ does this scale depend on?
\end{enumerate}

\subsection{On the proofs}
The proof of Theorem \ref{thm:sign bal rand wav} (including Theorem \ref{thm:sign-bal spher harm} as a special case) divides cleanly into an upper bound (Theorem \ref{thm:sign bal rand wav}(i.)) and a lower bound (Theorem \ref{thm:sign bal rand wav}(ii.)), which proceed via disjoint routes.
For the {\em upper bound} we apply the upper concentration estimate (Theorem \ref{t:cub}) to a well-chosen dense net of geodesic balls. We complete the proof by combining the union bound with a certain stability property of the volume-bias. 

The proof of Theorem \ref{t:cub} appeals to L\'{e}vy's concentration of measure principle and the Gaussian isoperimetric inequality. It is inspired by the proof of the exponential concentration of the nodal domain count for random spherical harmonics ~\cite{NaSoAJM}, but the present setting and the analysis are very different. 

Let us point out a curious aspect of our use of the union bound in deducing Theorem \ref{thm:sign bal rand wav}(i.) from Theorem \ref{t:cub}. One may be tempted to think that once the sign-balance of Theorem \ref{t:cub}(i.) has been established at some scale $\mu\cdot \bar{r}_{T}$, the same holds at all scales $r>\mu\cdot \bar{r}_{T}$, at least for $\mu$ sufficiently large. A natural way to argue would be to employ the integral-geometric sandwich approach of Nazarov-Sodin ~\cite{SoSPB}. However, such an argument would require an increase in the admissible scale, i.e.\ would only work for $r$ so that $\frac{r}{\bar{r}_{T}}\rightarrow \infty$. Interestingly, it turns out that there is an unexpected (at least, to the authors) obstruction to this line of argument: for every $r>1$ there exists a sequence of smooth functions on $\R^{2}$ which are sign-balanced at radius $1$ but {\em not} sign-balanced at radius $r$.

\begin{proposition}
\label{prop:no scale increase}
For every $r>1$ there exists a sequence of smooth functions $f_{j}:\R^{2}\rightarrow\R$ with the following properties as $j \to \infty$:
\begin{enumerate}[i.]

\item Uniformly over $x$ in compact subsets of $\R^2$,
\begin{equation}
\label{eq:sign balan rad=1}
\frac{1}{\pi}\int\limits_{B_{1}(x)} H(f_{j}(y))dy \rightarrow 0.
\end{equation}

\item There exists a number $z=z(r) > 0$ so that 
\begin{equation}
\label{eq:sign nonbalan rad>1}
\frac{1}{\pi r^{2}}\int\limits_{B_{r}(0)} H(f_{j}(y))dy \rightarrow z.
\end{equation}

\end{enumerate}
\end{proposition}

\noindent For completeness we give a proof of Proposition \ref{prop:no scale increase} in Appendix \ref{apx:no scale increase} (easy to generalise to arbitrary dimensions). While the proof constructs a sequence that satisfies the statement for {\em almost} every radius $r>1$ (i.e.\ satisfies \eqref{eq:sign nonbalan rad>1} with $r\in (1,\infty)\setminus S$, where $S$ is a set of isolated points), it is conceivable that one could construct a single sequence that works for {\em every} $r>1$, although we do not address this.

\smallskip
For the {\em lower bound} we apply the lower concentration estimates (theorems \ref{t:clb} and \ref{t:clb2}) to a well-separated collection of geodesic balls, and combine with a certain `sprinkled decoupling' technique that allows to establish the approximate independence of the defect (or volume-bias) on these balls. 
The proof of theorems \ref{t:clb} and \ref{t:clb2} proceed by constructing an exceptional event of sufficiently large probability. The construction in Theorem \ref{t:clb2} is by far the more demanding, since it requires the existence of a `sign-barrier': a sequence of functions in the reproducing kernel Hilbert space of the field that are {\em not} sign-balanced at the required scales. In fact, even constructing a {\em single} (deterministic) function that is not sign-balanced at scales above Planck's scale is a major challenge, since the natural candidate (the reproducing kernel $K_T(\cdot,x)$) turns out to be {\em perfectly} sign-balanced at the theoretical minimum scale, decisively failing to serve as a sign-barrier. 

Instead we construct a sign-barrier involving three superimposed copies of the reproducing kernel, associated to a certain modification of the field $f_T$ with restricted energy levels, centred at three distant points; this is inspired by a similar Euclidean construction in~\cite{kwy21}. This construction is by far the most subtle and difficult argument in the proof. The construction in Theorem \ref{t:clb} is relatively simple, involving only a single copy of the reproducing kernel $K_T(\cdot,x)$. This turns out to be sufficient, since the proof of Theorem \ref{t:clb} only requires the volume-imbalance of the reproducing kernel at some positive level $u > 0 $.

\smallskip
While the notion of sign-balance is most interesting for random waves, our results and methods actually apply to more general Gaussian ensembles than we consider in  Theorem \ref{thm:sign bal rand wav}. In fact, for many ensembles, namely those that have a non-trivial contribution from low eigenmodes, we could actually prove a stronger lower bound compared to Theorem \ref{thm:sign bal rand wav} that replaces the lower scale $\ubar{r}_T$ with the upper scale $\bar{r}_T$ even \textit{at} the nodal level $u=0$, matching our results for non-zero levels $u \neq 0$. Examples of such ensembles are the \textit{fully-banded} random waves (i.e.\ with $\eta = T$), or the Kostlan ensemble of random polynomials on the sphere. For these ensembles one can easily construct an optimal sign-barrier using the low eigenmodes of the field.

\medskip
\section{Asymptotics for the covariance kernel of random waves}
\label{sec:covar kern}

In this section we study the asymptotics of the covariance kernel 
\begin{equation}
\label{eq:K covar def expr}
K_{T}(x,y) = K_{T,\eta}(x,y) =\E\left[ f_{T,\eta}(x)\cdot f_{T,\eta}(y)  \right] = \frac{\left|\Mcc\right|}{N(T,\eta)}\sum\limits_{\lambda_{j}\in [T-\eta,T]} \varphi_{j}(x)\cdot \varphi_{j}(y)
\end{equation}
of the ensemble $\{f_{T,\eta}\}$ of band-limited functions \eqref{eq:fTeta band-lim def}, central to our analysis. The kernel $K_{T}$ coincides with the {\em spectral projector} in $L^{2}(\Mcc)$ onto the space spanned by $\{\varphi_{j}\}_{\lambda_{j}\in [T-\eta,T]}$, and its asymptotic behaviour in various regimes has been extensively studied 
in the microlocal analysis literature.

\subsection{Statement of the asymptotics}
\label{sec:asymp covar stat}

The results given in this section prescribe the leading asymptotics for $K_{T}(x,y)$ in some regimes, both on and off the diagonal. We first discuss the general case of smooth manifolds, before giving some finer results for the round sphere $\Mcc = \Sc^d$.

\subsubsection{General manifolds}

\begin{proposition}
\label{prop:asymp covar waves}
There exists a number $C = C(\Mcc) >0$ such that the following holds:
$\,$

\begin{enumerate}[i.]

\item {\bf Diagonal estimate:}
For every $T\ge 1$, $\eta\in [C,T]$, $x\in\Mcc$, one has
\begin{equation*}
 K_{T,\eta}(x,x) = 1 +  O\left( \frac{1}{\eta}  \right),
\end{equation*}
where the implicit constant in the `$O$'-notation depends only on $\Mcc$.

\item {\bf Off-diagonal estimate:} 
Define 
\begin{equation}
\label{eq:gammad offset def}
\gamma_{d}:= \frac{1}{4}(1-d)\pi \qquad \text{and} \qquad c_{d} = \sqrt{\frac{2}{\pi}}\cdot \frac{(2\pi)^{d/2}}{V_{d}} ,
\end{equation}
with $V_{d}$ the volume of the unit $d$-ball. For $x,y\in\Mcc$ denote $r:=d(x,y)$. Then for every $T\ge 1$, $\eta\in [C,T]$, $x,y\in\Mcc$, one has
\begin{equation}
\label{eq:covar rand wav asymp off}
K_{T,\eta}(x,y) = c_{d}\cdot (rT)^{-\frac{d-1}{2}} \Big(\cos\left( r\cdot T + \gamma_{d} \right) + O \Big(\eta r + \frac{1}{rT} + \frac{(rT)^{\frac{d-1}{2}}}{\eta}  \Big) \Big),
\end{equation}
where the implicit constant in the `$O$'-notation depends only on $\Mcc$.
\end{enumerate}
\end{proposition}

Proposition \ref{prop:asymp covar waves} asserts that, in suitable regimes, the properly rescaled random waves are well approximated by the $d$-dimensional monochromatic waves akin to \eqref{eq:Berry's RWM covar}. While this is well-known, we did not find a strong quantitative version of the type \eqref{eq:covar rand wav asymp off} in the literature (it is notably sharper than~\cite[\S~2.1]{SW}). We believe Proposition \ref{prop:asymp covar waves} to be of significant interest for the purpose of other applications on random waves on generic manifolds. 

\vspace{2mm}
Observe that the asymptotic formula \eqref{eq:covar rand wav asymp off} breaks down (i.e.\ the error term blows up), 
unless $rT\rightarrow \infty$, forcing $\eta\rightarrow \infty$, and a forteriori $r\rightarrow 0$. It also breaks down unless $\eta r \cdot \frac{1}{r T} = \frac{\eta}{T} \to 0$, so that $\eta = o(T)$. In other words, as might be expected, the asymptotics are only applicable for monochromatic $f_{T}$. The fact that the said asymptotics break down unless $\eta \to \infty$ is the underlying reason for this extra assumption in Theorem \ref{thm:sign bal rand wav} for generic manifolds. However note that Theorem \ref{thm:sign bal rand wav} does \textit{not} require $f_{T}$ to be monochromatic; this is since we shall apply Proposition \ref{prop:asymp covar waves} only after restricting the energies of $f_T$ to lie in a monochromatic band.

\subsubsection{Round sphere}
We next present stronger results for the round sphere $\Mcc = \Sc^d$. In this case, $K_{T,\eta}(x,y)$ depends only on the spherical distance $\theta=d(x,y)\in [0,\pi]$ (akin to $r$ in \eqref{eq:covar rand wav asymp off}), and has an explicit expression in terms of orthogonal polynomials.

\vspace{2mm}
In what follows it will be convenient to slightly abuse our notation. We assume that $\ell=T$ is a positive integer (in reality, $\ell(\ell+d-1)=T^{2}$ but we will neglect this discrepancy), and $\eta \ge 1$ is an integer, so that the energy window $\Wc:=[\ell-\eta+1,\ell] = [T-\eta+1,T]$ contains the `integer' energies $\ell-\eta+1,\ldots, \ell$ with the corresponding multiplicities $n_{\ell'}$ as in \eqref{eq:n dim spher harm def}, $\ell'\in \Wc$. We then re-define \eqref{eq:fTeta band-lim def} by writing
\begin{equation}
\label{eq:band-lim sphere}
f_{\ell,\eta}(x) = \frac{1}{\sqrt{N(\ell,\eta)}}\sum\limits_{\ell'=\ell-\eta+1}^{\ell}\sqrt{n_{\ell}}\cdot H_{\ell'}(x),
\end{equation}
where 
\begin{equation}
\label{eq:N(ell,eta) def spher}
N(\ell,\eta):=\sum\limits_{\ell'=\ell-\eta+1}^{\ell}n_{\ell'},
\end{equation}
with $n_{\ell'}$ as in \eqref{eq:n dim spher harm def}. Evidently, the random spherical harmonics $f_{\ell,1}(\cdot)\equiv H_{\ell}(\cdot)$ correspond to the `shortest' energy window. 

\vspace{2mm} The random field $f_{\ell,\eta}$ is invariant w.r.t.\ rotations of $\Sc^{d}$, hence the corresponding covariance kernel depends only on $\theta = d(x,y)$: 
\[ K_{\ell,\eta}(\theta)=K_{\ell,\eta}(x,y) = \E\left[f_{\ell,\eta}(x) \cdot f_{\ell,\eta}(y)\right]. \]
In particular, the identity 
\begin{equation}
\label{eq:var ==1 spher band}
\var\left(f_{\ell,\eta}(x)\right) =  K_{\ell,\eta}(0) \equiv   1
\end{equation}
holds by the normalisation 
\[ \E\left[\|f_{T,\eta}\|^{2}\right]\equiv |\Sc^{d}|, \]
 a by-product of the definition \eqref{eq:fTeta band-lim def} (cf.\ \eqref{eq:var==1 spher harm}).

\vspace{2mm}
One may express $K_{\ell,\eta}(\theta)$ in terms of orthogonal polynomials as follows. Denote 
\begin{equation}
\label{eq:N<=ell sphere sum enrg}
N(\ell):= N(\ell,\ell) = \sum\limits_{\ell'=1}^{\ell}n_{\ell'}
\end{equation}
(consistent to the new conventions)
and let 
\begin{equation}
\label{eq:covar full band sphere}
K_{\le\ell}(\theta)=K_{\le\ell}(x,y):= K_{\ell,\ell}(x,y)
\end{equation}
be the covariance kernel of the {\em fully banded} ensemble $f_{\ell,\ell}$, so that
\begin{equation}
\label{eq:covar rec full band spher}
N(\ell,\eta)K_{\ell,\eta}(\theta) = N(\ell)K_{\le\ell}(\theta)- N(\ell-\eta)K_{\le\ell-\eta}(\theta).
\end{equation}
Then one has the following exact formula, an instance of Christoffel-Darboux \cite[Equality (4.5.3)]{szego}:
\begin{equation}
\label{eq:Chris Darb}
K_{\le\ell}(\theta) = \frac{c_{d}}{N(\ell)}  \cdot \frac{  \Gamma(\ell+d) }{ \Gamma(\ell + d/2)   } P_{\ell}^{(d/2,(d-2)/2)}(\cos \theta) ,  
\end{equation}
where, as above, $P_\ell^{(\alpha,\beta)}$ is the Jacobi polynomial, and $c_{d}>0$ is a dimensional constant that, will be eventually recovered by invoking the unit variance constraint $K_{\le\ell}(0)=1$, cf.\ \eqref{eq:var==1 spher harm} or \eqref{eq:var int Mcc}.

\begin{proposition}
\label{prop:asymp covar spher}
For every $\ell\ge 1$ and $\eta\in [1,\ell]$ one has 
\begin{equation}
\label{eq:var spher==1}
K_{\ell,\eta}(0)\equiv 1.
\end{equation}
Further, for every $\ell\ge 1$, $\eta\in [1,\ell]$ and $\theta\in \left[0,\frac{\pi}{2}\right]$ one has
\begin{equation}
\label{eq:covar spher asymp off}
K_{\ell,\eta}(\theta)=c_{d}\cdot  \frac{1}{(\theta\ell)^{\frac{d-1}{2}}} \Big (\cos(\theta\ell+\gamma_{d}) + O\Big(\frac{1}{\ell \theta} + \eta \theta \Big)  \Big),
\end{equation}
where $\gamma_{d}$ is given by \eqref{eq:gammad offset def}, $$c_{d} = \frac{(d-1)!\cdot |\Sc^{d}|}{(2\pi)^{(d+1)/2}}$$ is a dimensional constant, and the constant implicit in the `$O$'-notation only depends on $d$.
\end{proposition}

A comparison between the respective error terms in \eqref{eq:covar rand wav asymp off} and \eqref{eq:covar spher asymp off} reveals that the latter is lacking a term analogous to $O(\eta^{-1} (rT)^{\frac{d-1}{2}})$. This is a manifestation of the fact that, thanks to the spectral degeneracies in the case $\Mcc = \Sc^d$, the width $\eta$ of the energy window does not need to grow for the corresponding random waves to exhibit universality.

\vspace{2mm}
Though we restrict $\theta \le \frac{\pi}{2}$, one may obtain approximate values of $K_{\ell,\eta}(\theta)$ for $\theta\in [\frac{\pi}{2},\pi]$ by using the natural symmetry of the Jacobi polynomials $P_{\ell}^{(d/2,(d-2)/2)}$ (which depends on the parity of $\ell$).

\subsubsection{Uniform off-diagonal decay}

We shall also need the following rougher estimate, providing uniform power-law decay of correlations:

\begin{corollary}
\label{cor:decay correlations}
Assume that either (a) there exists a number $\delta_{0}>0$ so that $\eta(T)>T^{\delta_{0}}$ or (b) $\Mcc=\Sc^{d}$.
There exist numbers $\delta_{1}>0$ sufficiently small and $C_{1}>0$ sufficiently large, so that one has
\begin{equation*}
|K_{T,\eta}(x,y)| \le \frac{C_{1}}{T^{\delta_{1}}},
\end{equation*}
uniformly for every $x,y\in\Mcc$ satisfying $d(x,y) > \frac{1}{T^{1-\delta_{1}}}$, restricted to $d(x,y)\in \left[0,\frac{\pi}{2}\right]$ in case (b).
\end{corollary}

\noindent In fact, it is possible to establish a stronger estimate of the form $|K_{T,\eta}(x,y)| = O( (rT)^{-\frac{d-1}{2}} )$ via a more technically demanding routine, but it will not be required.

\subsection{Auxiliary lemmas}
Towards the proofs of Propositions \ref{prop:asymp covar waves}-\ref{prop:asymp covar spher} we present two auxiliary lemmas. The first one deals with the asymptotics of the covariance kernel $K_{\le T}$ corresponding to the fully banded regime, separately for the random waves on generic smooth manifolds and the spheres. It is easy to express the covariance kernel $K_{T,\eta}$ in terms of $K_{\le T}$ as follows:
\begin{equation}
\label{eq:covar rec full band gen}
N(T,\eta)K_{T,\eta}(x,y) = N(T)K_{\le T}(x,y)-N(T-\eta)K_{\le T-\eta}(x,y), 
\end{equation}
cf. \eqref{eq:covar rec full band spher}.

\begin{lemma} 
\label{lem:full band kern asymp}
$\, $

\begin{enumerate}[i.] 
 
\item
For every $T\ge 1$ and $x,y\in\Mcc$, one has
\begin{equation*}
K_{\le T}(x,y) =   c_{d}\cdot \frac{T^{d}}{ N(T)}\cdot \frac{J_{d/2}(rT)}{(rT)^{d/2}}  + O\left(T^{-1}\right) = 
\frac{(2\pi)^{d/2}}{V_{d}}\cdot \frac{J_{d/2}(rT)}{(rT)^{d/2}}  + O\left(T^{-1}\right),
\end{equation*}
where $$c_{d} = \frac{|\Mcc|}{(2\pi)^{d/2}}$$
is a dimensional constant, and the constant implicit in the `$O$'-notation only depends on $\Mcc$.

\item Let $\Mcc=\Sc^{d}$ be the round sphere, and recall the notation in \eqref{eq:N<=ell sphere sum enrg} and \eqref{eq:covar full band sphere}. For every $\ell\ge 1$ and $x,y\in\Sc^{d}$ so that $\theta:= d(x,y)\in \left[0,\frac{\pi}{2}\right]$, one has
\begin{equation}
\label{eq:K full band Bessel}
K_{\le\ell}(\theta) = c_{d}\cdot \frac{\ell^{d}}{N(\ell)} \Big( \mu(\ell)\kappa(\theta)J_{d/2}\Big(\Big(\ell+\frac{d}{2} \Big)\theta \Big) + O\Big(\frac{1}{\theta^{(d-1)/2}\ell^{(d+3)/2}} \Big)   \Big),
\end{equation}
where $$c_{d}=\frac{|\Sc^{d}|}{2^{d}\pi^{d/2}}$$ is a dimensional constant,
\begin{equation}
\label{eq:kappa def asymp}
\kappa(\theta) =  ( \sin(\theta/2))^{-d/2} (\cos(\theta/2))^{-(d-2)/2} \Big(\frac{\theta}{\sin{\theta}}\Big)^{1/2}  = 2^{d/2}\cdot \theta^{-d/2}+O_{\theta\rightarrow 0}\left(\theta^{-d/2+1}\right),
\end{equation}
\begin{equation}
\label{eq:mu(ell) def}
\mu(\ell) = \frac{ (\ell + d/2)^{-d/2+1}  (\ell+d-1)! }{ \ell^d \ell !  } =\ell^{-d/2}\cdot \Big(1+ O \Big( \frac{1}{\ell} \Big)\Big),
\end{equation}
and
\begin{equation*}
\frac{\ell^{d}}{N(\ell)} = \frac{(2\pi)^{d}}{|\Sc^{d}|\cdot V_{d}}\cdot \Big(1+ O \Big( \frac{1}{\ell} \Big) \Big),
\end{equation*}
and the constant implicit in the `$O$'-notation only depends on $d$.
\end{enumerate}

\end{lemma}

\begin{proof}
First we prove Lemma \ref{lem:full band kern asymp}(i.). Recall Weyl's law (with quantitative error term) for the spectral function \eqref{eq:N(T) spec func}:
\begin{equation}
\label{eq:Weyl quant}
N(T) = \frac{V_{d}\cdot|\Mcc|}{(2\pi)^{d}}\cdot T^{d} + O\left(T^{d-1}\right),
\end{equation}
and its {\em local} version as stated\footnote{The factor of $(2\pi)^{-d/2}$ is missing in \cite{SW}.} ~\cite[Section 2.1]{SW}: uniformly for $T\ge 1$ and $x,y\in\Mcc$, one has
\begin{equation}
\label{eq:loc Wey SW}
\sum\limits_{\lambda_{j}\le T}\varphi_{j}(x)\cdot \varphi_{j}(y) = \frac{T^{d}}{(2\pi)^{d}}\cdot A_{d}(T\cdot d(x,y)) + O(T^{d-1}), 
\end{equation}
where for $t>0$,
\begin{equation}
\label{eq:Ad Fourier Bessel}
A_{d}(t) =  (2\pi)^{d/2}\frac{J_{d/2}(t)}{t^{d/2}}
\end{equation}
is the (radial) Fourier transform of the indicator of the unit $d$-ball, which is bounded (see Lemma \ref{lem:aux simple Bessel}(ii.)). 
Then both estimates of Lemma \ref{lem:full band kern asymp}(i.) follow upon substituting \eqref{eq:Ad Fourier Bessel} into the asymptotics \eqref{eq:loc Wey SW}, dividing the error term by \eqref{eq:Weyl quant}, 
and recalling the definition \eqref{eq:K covar def expr} of the covariance kernel.

\vspace{2mm}

Now we turn to proving Lemma \ref{lem:full band kern asymp}(ii.). Recall the exact expression \eqref{eq:Chris Darb} for $K_{\le\ell}(\theta)$ in terms of the Jacobi polynomial $P_\ell^{(\alpha,\beta)}$, and
denote 
\begin{equation}
\label{eq:M=ell+d/2}
M:=\ell+\frac{d}{2}.
\end{equation}
An application of the general `Hilb'-type asymptotics \cite[Formula (8.21.17)]{szego} for $P_\ell^{(\alpha,\beta)}$ with $\alpha=\frac{d}{2}$ and $\beta = \frac{d}{2}-1$ yields the uniform asymptotics
\begin{equation}
\label{eq:Jacobi pol approx}
\begin{split}
P^{(d/2,(d-2/2))}(\cos{\theta}) &= \sin\left(\frac{\theta}{2}\right)^{-\frac{d}{2}}\cos\left(\frac{\theta}{2}\right)^{-\frac{d-2}{2}}M^{-\frac{d}{2}}\frac{\Gamma(M+1)}{\ell!}
\left(\frac{\theta}{\sin(\theta)}\right)^{1/2}J_{d/2}\left(M\theta\right) + \epsilon_{\ell}(\theta)
\\&= \sin\left(\frac{\theta}{2}\right)^{-\frac{d}{2}}\cos\left(\frac{\theta}{2}\right)^{-\frac{d-2}{2}}M^{-\frac{d}{2}+1}\frac{\Gamma(M)}{\ell!}
\left(\frac{\theta}{\sin(\theta)}\right)^{1/2}J_{d/2}\left(M\theta\right) + \epsilon_{\ell}(\theta),
\end{split}
\end{equation}
where 
\begin{equation*}
\epsilon_{\ell}(\theta) =\begin{cases}
O\left(\frac{1}{\theta^{(d-1)/2}\ell^{3/2}}\right) &\theta>1/\ell , \\
O\left(\theta^{2}\ell^{d/2}\right) &\theta\le 1/\ell.
\end{cases}
\end{equation*}
We observe that, for $\theta \le \frac{1}{\ell}$, the inequality $$\epsilon_{\ell}(\theta) =  O\Big(\frac{1}{\theta^{(d-1)/2}\ell^{3/2}}\Big)$$ holds trivially, since in this range 
$$\theta^{2}\ell^{d/2} \le \ell^{d/2-2}\le \frac{1}{\theta^{(d-1)/2}\ell^{3/2}}.$$ Hence the error term $\epsilon_{\ell}(\theta)$ in \eqref{eq:Jacobi pol approx} may be replaced by $O(\frac{1}{\theta^{(d-1)/2}\ell^{3/2}})$.

Now, since, thanks again to Weyl's law \eqref{eq:Weyl quant}, $$\frac{\Gamma(\ell+d)}{N(\ell)\cdot  \Gamma(\ell+d/2)} \asymp \frac{1}{\ell^{d/2}},$$
we may obtain Lemma \ref{lem:full band kern asymp}(ii.) by 
substituting \eqref{eq:Jacobi pol approx} into \eqref{eq:Chris Darb}. Indeed, the main term of \eqref{eq:K full band Bessel} coincides up to a dimensional constant (on recalling \eqref{eq:M=ell+d/2}), and the error term 
\[ \frac{1}{\ell^{d/2}}\epsilon_{\ell}(\theta) = O\Big(\frac{1}{\theta^{(d-1)/2}\ell^{(d+3)/2}}\Big) \]
 is as claimed. Finally, we may recover the dimensional constant in \eqref{eq:K full band Bessel} via the unit variance constraint $K_{\le \ell}(0)=1$, and Lemma \ref{lem:aux simple Bessel}(ii.) below.
\end{proof}

The next lemma deals with some standard asymptotic expressions for the usual Bessel $J$ functions and their relation to the relevant Jacobi polynomials. This will allow us to pass the asymptotics in Lemma \ref{lem:full band kern asymp} from the fully banded regime to arbitrary bands.

\begin{lemma}
\label{lem:aux simple Bessel}
Suppose $d \ge 0$.
\begin{enumerate}[i.]

\item 
As $t\rightarrow \infty$, one has (with $\gamma_{d}$ as in \eqref{eq:gammad offset def})
 $$J_{d/2}(t) = \sqrt{\frac{2}{\pi}}\cdot \frac{\cos(t+\gamma_{d})}{t^{1/2}} + O\left( t^{-3/2} \right).$$

\item As $t \rightarrow 0$, 
\begin{equation*}
J_{d/2}(t) = c_{d}t^{d/2}+O\left(t^{d/2+1}\right), \qquad  c_{d}=  \frac{V_{d}}{(2\pi)^{d/2}} .
\end{equation*}

\item One has $$J_{d/2}(t) -(1-\delta)^{d/2} J_{d/2}((1-\delta)t) = \delta t\cdot J_{(d-2)/2}(t) + O \big( \delta^{2}t^{3/2}  \big),$$ uniformly for $0\le \delta\le 1$, $t>0$.
\end{enumerate}

\end{lemma}

\begin{proof}
Lemma \ref{lem:aux simple Bessel}(i.) follows from keeping the leading term in the classical expansion at infinity for the Bessel $J$ functions, see e.g.\ \cite[(1.71.1)]{szego}. Lemma \ref{lem:aux simple Bessel}(ii.) is standard. 
Lemma \ref{lem:aux simple Bessel}(iii.) is derived by writing $$(1-\delta)^{d/2}= 1- \frac{d}{2}\delta +O\left(\delta^{2}\right),$$ Taylor expanding the function $J_{d/2}(\cdot)$ around $t$
$$J_{d/2}((1-\delta)t)  = J_{d/2}(t)-\delta t\cdot J_{d/2}'(t) + O \Big(\delta^{2}t^{2} \sup\limits_{\xi\in [(1-\delta)t,t]}J_{d/2}''(\xi)\Big),$$
and bearing in mind the identity 
\begin{equation}
\label{eq:Bessel der recurr}
J'_{d/2}(t) =J_{(d-2)/2}(t)  - \frac{d}{2t}  J_{d/2}(t) . \qedhere
\end{equation}
\end{proof}

\subsection{Proofs of Propositions \ref{prop:asymp covar waves}-\ref{prop:asymp covar spher} and Corollary \ref{cor:decay correlations}}

\begin{proof}[Proof of Proposition \ref{prop:asymp covar waves}]

First, for the diagonal estimate of Proposition \ref{prop:asymp covar waves}(i.), we employ \eqref{eq:covar rec full band gen}, and substitute the asymptotics of Lemma \ref{lem:full band kern asymp}(i.) with $r=0$,
while bearing in mind Lemma \ref{lem:aux simple Bessel}(ii.). Since for $\eta>0$ one has $$T^{d}-(T-\eta)^d\asymp \eta T,$$ it follows that: 
\begin{equation*}
K_{T,\eta}(x,x) = \frac{T^{d}-(T-\eta)^{d} + O\left(T^{d-1}\right)}{T^{d}-(T-\eta)^{d} + O\left(T^{d-1}\right)} = \frac{\left(T^{d}-(T-\eta)^{d}\right)\left(1+O\left(\frac{1}{\eta}\right)\right)}{\left(T^{d}-(T-\eta)^{d}\right)\left(1+O\left(\frac{1}{\eta}\right)\right)}
=1+O\left(\frac{1}{\eta}\right),
\end{equation*}
provided that we chose $\eta>C$ with $C$ sufficiently large so that the denominator does not blow up.

For the off-diagonal estimate of Proposition \ref{prop:asymp covar waves}(ii.) we recall the tacit assumption, with no loss of generality, that $T-\eta$ is not an eigenvalue. We denote $$\alpha:=\frac{d-2}{2}.$$
Then by \eqref{eq:covar rec full band gen} and Lemma \ref{lem:aux simple Bessel}(iii.), we have
\begin{align*}
N(T,\eta) K_{T,\eta}(x,y) &=  N(T) K_{\le T}(x,y) - N(T-\eta)K_{\le T-\eta}(x,y)   \\
& = \frac{|\Mcc|}{(2\pi)^{d/2}}\cdot T^{d/2} r^{-d/2} \left( J_{d/2}(rT) - (1 - \eta/T)^{d/2} J_{d/2}(( 1 - \eta/T) rT) \right) + O(T^{d-1})
\\&= \frac{|\Mcc|}{(2\pi)^{d/2}}\cdot \eta T^{d/2} r^{-(d-2)/2} J_{\alpha}(rT) + O( T^{(d-1)/2}\eta^{2}r^{-(d-3)/2}  + T^{d-1})  .
\end{align*}
Now we invoke Lemma \ref{lem:aux simple Bessel}(i.) to yield
\begin{equation}
\label{eq:asynp NKTeta}
\begin{split}
&N(T,\eta) K_{T,\eta}(x,y) =\sqrt{\frac{2}{\pi}} \frac{|\Mcc|}{(2\pi)^{d/2}}\cdot \eta T^{(d-1)/2} r^{-(d-1)/2} \Big(\cos\left(rT+\gamma_{d}\right) + O\Big( \frac{1}{rT}  \Big) \Big)
\\& \qquad \qquad \qquad \qquad \qquad + O ( T^{(d-1)/2}\eta^{2}r^{-(d-3)/2}+ T^{d-1}  ) 
\\&= \sqrt{\frac{2}{\pi}} \frac{|\Mcc|}{(2\pi)^{d/2}}\cdot \eta T^{(d-1)/2} r^{-(d-1)/2} \Big(\cos\left(rT+\gamma_{d}\right) + O \Big( \frac{1}{rT} + \eta r + \frac{(rT)^{\frac{d-1}{2}}}{\eta}  \Big) \Big).
\end{split}
\end{equation}

To obtain the statement of Proposition \ref{prop:asymp covar waves}(ii.) it remains to divide \eqref{eq:asynp NKTeta} by $N(T,\eta)$ (or, rather, its asymptotics). By invoking Weyl's law \eqref{eq:Weyl quant}, it is easy to infer
that $$N(T,\eta)= d\eta T^{d-1} +O\left(T^{d-1}\right) \sim d\eta T^{d-1},$$ provided that $\eta\rightarrow \infty$ and 
$\eta = o(T),$ that is not assumed at this stage. Instead, we write
\begin{equation*}
(2\pi)^{d}\frac{N(T,\eta)}{V_{d}|\Mcc|}= d\cdot\eta T^{d-1}\cdot \frac{1}{d} \frac{T}{\eta} \Big(1-\Big(1-\frac{\eta}{T}\Big)^{d}\Big) +O(T^{d-1})   = d\cdot \eta T^{d-1}\cdot h\Big(\frac{\eta}{T}\Big) + O(T^{d-1}), 
\end{equation*}
$$h(\delta):= \frac{1}{d}\cdot \frac{1-(1-\delta)^{d}}{\delta}.$$ We observe that the function $h(\cdot)$ is smooth on $[0,1]$, $h(0)=1$, and it is bounded between two strictly positive constants. It follows that 
$$\frac{1}{h(\delta)} = 1+O(\delta),$$ and, on recalling the assumption $\eta(T)>C$ where we have the freedom to choose $C$ arbitrarily large depending on $d$,
\begin{equation}
\label{eq:1/N(T,eta) asymp}
\frac{V_{d}|\Mcc|}{(2\pi)^{d} N(T,\eta)} = \frac{1}{d\cdot\eta T^{d-1} h(\delta)\big(1+O \big(\frac{1}{\eta} \big)  \big)} = \frac{1}{d\cdot\eta T^{d-1}}\cdot \Big(1+O \Big(\frac{\eta}{T}+\frac{1}{\eta} \Big)  \Big),
\end{equation}
with the constant involved in the `$O$'-notation only depending on $d$. 

We may infer from Proposition \ref{prop:asymp covar waves}(i.) that $K_{T,\eta}(x,y)$ is uniformly bounded by a constant only depending on $d$, via Cauchy-Schwarz. 
Hence, provided that $r\ge \frac{1}{T}$, we have $$\frac{\eta}{T}+\frac{1}{\eta} \le r\eta + \frac{(rT)^{\frac{d-1}{2}}}{\eta},$$
and if otherwise $r \le \frac{1}{T}$, $$K_{T,\eta}(x,y) \ll 1\le \frac{1}{rT}.$$ 
The statement of Proposition \ref{prop:asymp covar waves}(ii.) now follows upon dividing both sides of \eqref{eq:asynp NKTeta} by 
$N(T,\eta)$, and using the asymptotic expression \eqref{eq:1/N(T,eta) asymp} 
for $\frac{1}{N(T,\eta)}$.
\end{proof}

\begin{proof}[Proof of Proposition \ref{prop:asymp covar spher}]

First, \eqref{eq:var spher==1} is a re-iteration of \eqref{eq:var ==1 spher band}. Towards the proof of \eqref{eq:covar spher asymp off},  
recall that one may recover $K_{\ell,\eta}(x,y)$ from $K_{\ell}(x,y)$ via \eqref{eq:covar rec full band spher}, and that $N(\ell,\eta)$ is the number of energy levels \eqref{eq:N(ell,eta) def spher} lying in $(\ell-\eta,\ell]$
(under the same abuse of notation of \S\ref{sec:asymp covar stat}). We invoke Lemma \ref{lem:full band kern asymp}(ii.) to write
\begin{equation}
\label{eq:Nell,eta Kell,eta approx}
\begin{split}
&N(\ell,\eta)\cdot K_{\ell,\eta}(\theta) = c_d \kappa(\theta) \left( \ell^d \mu(\ell) J_{d/2}( M \theta) - (\ell-\eta)^d \mu( \ell - \eta) J_{d/2}( ( M - \eta ) \theta) \right) + O\Big(\frac{\ell^{(d-3)/2}}{\theta^{(d-1)/2}}\Big),   
\end{split}
\end{equation}
with $$c_{d}= \frac{|\Sc^{d}|}{2^{d}\pi^{d/2}},\;\;\; M:=\ell+\frac{d}{2}.$$
Abbreviating $\delta =  \frac{\eta}{M}$, we can rewrite \eqref{eq:Nell,eta Kell,eta approx} as
\begin{equation}
\label{eq:recover band full sphere}
N(\ell,\eta)\cdot K_{\ell,\eta}(\theta) = c_d \kappa(\theta)\ell^{d} \cdot \left( \mu(\ell) J_{d/2}( M \theta) - (1-\delta)^d \mu( \ell - \eta) J_{d/2}( M(1-\delta)\theta) \right) + O \Big(\frac{\ell^{(d-3)/2}}{\theta^{(d-1)/2}}\Big),
\end{equation}
and, further
\begin{equation}
\label{eq:mu J approx}
\begin{split}
&\mu(\ell) J_{d/2}( M \theta) - (1-\delta)^d \mu( \ell - \eta) J_{d/2}( M(1-\delta)\theta) \\
& \qquad =  \mu(\ell)\cdot \left( J_{d/2}( M \theta)-(1-\delta)^{d/2} \cdot J_{d/2}( M(1-\delta)\theta)\right) \\
&  \qquad \qquad + \left(\mu(\ell)(1-\delta)^{d/2}- (1-\delta)^d \mu( \ell - \eta)\right)\cdot J_{d/2}( M(1-\delta)\theta).
\end{split}
\end{equation}

Next, we approximate the former of the two terms on the r.h.s. of \eqref{eq:mu J approx}, and bound above the latter of these terms. Indeed, we denote $\alpha:= \frac{d-2}{2}$, and use Lemma \ref{lem:aux simple Bessel}(iii.),
and, further, Lemma \ref{lem:aux simple Bessel}(i.) with \eqref{eq:mu(ell) def} for approximating $\mu(\ell)\approx \ell^{-d/2}$, to yield
\begin{equation}
\label{eq:first term sphere}
\begin{split}
&\mu(\ell)\cdot \left( J_{d/2}( M \theta)-(1-\delta)^{d/2} \cdot J_{d/2}( M(1-\delta)\theta)\right) = \mu(\ell)\cdot  \Big(\eta \theta \cdot J_{\alpha}(M\theta) +  O \Big( \frac{\eta^{2}\theta^{3/2}}{\ell^{1/2}}   \Big)    \Big)
\\&= \sqrt{\frac{2}{\pi}}\cdot\frac{\eta\theta^{1/2}}{\ell^{(d+1)/2}}\Big(\cos\left(M\theta+\gamma_{d}\right)    +O\Big(\frac{1}{\ell \theta} + \eta\theta   \Big) \Big),
\end{split}
\end{equation} 
for the former term of \eqref{eq:mu J approx} (bearing in mind the obvious inequality $\frac{1}{\ell}\ll \frac{1}{\ell \theta}$). 

For the latter term of \eqref{eq:mu J approx}, we write
\begin{equation}
\label{eq:secondary term simp}
\begin{split}
&\left(\mu(\ell)(1-\delta)^{d/2}- (1-\delta)^d \mu( \ell - \eta)\right)\cdot J_{d/2}( M(1-\delta)\theta) 
\\& \qquad = (1-\delta)^{d/2}J_{d/2}( M(1-\delta)\theta) \cdot \Big(\mu(\ell)- (1-\delta)^{d/2} \mu \Big( M \Big(\frac{\ell}{M}- \delta \Big)\Big)\Big), 
\end{split}
\end{equation}  
and observe that, by interpreting the expression $\frac{(\ell+d-1)!}{\ell!}$ in \eqref{eq:mu(ell) def} as a $(d-1)$-degree polynomial, the function $\mu(\cdot)$ naturally extends to an elementary function of real variable. 
Taylor expanding 
the function $$\delta\mapsto \mu \Big( M \Big(\frac{\ell}{M}- \delta \Big) \Big)$$ around $\delta=0$ we obtain that 
\begin{equation}
\label{eq:Taylor exp mu}
\Big| |\mu(\ell)- (1-\delta)^{d/2} \mu \Big( M \Big(\frac{\ell}{M}- \delta \Big)\Big)\Big| \ll \frac{\delta M}{\ell^{d/2+1}} = \frac{\eta}{\ell^{d/2+1}},
\end{equation}
valid for $\delta<\frac{1}{2}$ (say), and trivial for $\delta>\frac{1}{2}$. We substitute \eqref{eq:Taylor exp mu} into \eqref{eq:secondary term simp}, and use the decay at infinity of the Bessel function of Lemma \ref{lem:aux simple Bessel}(i.) 
to obtain the bound
\begin{equation}
\label{eq:bound 2nd term}
\begin{split}
&\left|\left(\mu(\ell)(1-\delta)^{d/2}- (1-\delta)^d \mu( \ell - \eta)\right)\cdot J_{d/2}( M(1-\delta)\theta)\right| \\& \qquad \ll \frac{(1-\delta)^{d/2}}{(\ell (1-\delta)\theta)^{1/2}}\cdot \frac{\eta}{\ell^{d/2+1}} \ll \frac{\eta}{\ell^{(d+3)/2}\theta^{1/2}}
\end{split}
\end{equation}
for the second term of \eqref{eq:mu J approx}. 

\vspace{2mm}

We consolidate the estimate \eqref{eq:first term sphere} and the bound \eqref{eq:bound 2nd term} (that, for $\theta>\frac{1}{\ell}$ is majorised by the error term in \eqref{eq:first term sphere}), and substitute into \eqref{eq:mu J approx} to yield the estimate
\begin{equation}
\label{eq:intern expr covar}
\begin{split}
&\mu(\ell) J_{d/2}( M \theta) - (1-\delta)^d \mu( \ell - \eta) J_{d/2}( M(1-\delta)\theta) \\& \qquad \ll 
\sqrt{\frac{2}{\pi}}\cdot\frac{\eta\theta^{1/2}}{\ell^{(d+1)/2}}\Big(\cos\left(M\theta+\gamma_{d}\right)    +O\Big(\frac{1}{\ell \theta} + \eta \theta   \Big)\Big).
\end{split}
\end{equation}
It then remains to insert the estimate \eqref{eq:intern expr covar} into \eqref{eq:recover band full sphere}, and multiply it by $\frac{c_{d}\kappa(\theta)\ell^{d}}{N(\ell,\eta)},$ with $\kappa(\theta)$ as in \eqref{eq:kappa def asymp},
and divide the error term in \eqref{eq:recover band full sphere} by $N(\ell,\eta)$. 
To this end, we mind that, using the explicit spectral multiplicities \eqref{eq:n dim spher harm def} of the sphere, we find the asymptotic expression
\begin{equation*}
N(\ell,\eta) = \frac{2}{d!} \left(\ell^{d}-(\ell-\eta)^{d}\right) + O\left(\eta^{2} \ell^{d-2}\right) = \frac{2}{(d-1)!}\eta \ell^{d-1}\left(1+O\left( \frac{\eta}{\ell} \right)\right),
\end{equation*}
while $N(\ell,\eta)\asymp \eta\ell^{d-1}$, cf. \eqref{eq:1/N(T,eta) asymp}. Hence its reciprocal is
\begin{equation*}
\frac{1}{N(\ell,\eta)} = \frac{(d
-1)!}{2\eta \ell^{d-1}} \left(1+O\left( \frac{\eta}{\ell} \right)\right). 
\end{equation*}
Consolidating these estimates, we finally obtain:
\begin{equation*}
\begin{split}
&K_{\ell,\eta}(\theta) = c_d 2^{d/2-1} (d-1)! \sqrt{\frac{2}{\pi}}\cdot\frac{1}{(\theta\ell)^{\frac{d-1}{2}}}\Big(\cos\left(M\theta+\gamma_{d}\right)    +O\Big(\frac{1}{\ell \theta} + \eta \theta+\frac{\eta}{\ell}   \Big)\Big) \\& \qquad \qquad \qquad +  
O\Big(\frac{1}{\eta \ell^{\frac{d+1}{2}}\theta^{\frac{d-1}{2}}}\Big)
\\&= c_d 2^{d/2-1} (d-1)! \sqrt{\frac{2}{\pi}}\cdot\frac{1}{(\theta\ell)^{\frac{d-1}{2}}}\cdot \Big(\cos\left(M\theta+\gamma_{d}\right)    +O\Big(\frac{1}{\ell \theta} + \eta \theta+\frac{\eta}{\ell}  +  \frac{1}{\ell\eta} \Big)\Big).
\end{split}
\end{equation*}
Note that $\frac{1}{\ell\eta} \ll \frac{\eta}{\ell}$, and for $\theta \ge \frac{1}{\ell}$, $\frac{\eta}{\ell} \le \eta\theta$, whereas for $\theta\le \frac{1}{\ell}$, $\frac{1}{\ell \theta}\ge 1$, so that the estimate \eqref{eq:covar spher asymp off} is a tautology, due to the uniform boundedness of $K_{\ell,\eta}$, by Proposition \ref{prop:asymp covar spher}(i.) and Cauchy-Schwarz. 
\end{proof}

\begin{proof}[Proof of Corollary \ref{cor:decay correlations}]

First, we assume case (a), i.e. $\Mcc$ is an arbitrary smooth manifold, $\eta(T)>T^{\delta_{0}}$. We apply Lemma \ref{lem:full band kern asymp}(i.), together with \eqref{eq:covar rec full band gen} and \eqref{eq:1/N(T,eta) asymp}, 
and using the standard decay 
\begin{equation}
\label{eq:Bessel decay 1/sqrt(t)}
|J_{d/2}(t)| \ll \frac{1}{\sqrt{t}}
\end{equation}
of the Bessel function at infinity, to obtain
\begin{equation}
\label{eq:K as Bessel diff}
c_{d}'K_{T,\eta}(x,y) = \frac{\frac{J_{d/2}(rT)}{(rT)^{d/2}}\cdot T^{d} - \frac{J_{d/2}(r(T-\eta))}{(r(T-\eta))^{d/2}}\cdot (T-\eta)^{d}}{\eta\cdot T^{d-1}} + O \Big( \frac{1}{\eta}  \Big),
\end{equation}
with $r=d(x,y)$, for some dimensional constant $c_{d}'>0$. We rewrite \eqref{eq:K as Bessel diff} as
\begin{equation}
\label{eq:K as diff g}
c_{d}' K_{T,\eta}(x,y) = \frac{g(T)-g(T-\eta)}{\eta\cdot T^{d-1}}+O\Big( \frac{1}{\eta}  \Big),
\end{equation}
with 
\begin{equation}
g(t)=g_{r}(t) = \frac{J_{d/2}(rt)}{(rt)^{d/2}}\cdot t^{d/2} =r^{-d/2}\cdot t^{d/2}J_{d/2}(rt)  .
\end{equation}
Next, we bound the function $g(T)-g(T-\eta)$ in the denominator of \eqref{eq:K as diff g}. We differentiate:
\begin{equation*}
g'(t) = r^{-d/2}\cdot \Big(\frac{d}{2}t^{d/2-1}J_{d/2}(rt) +t^{d/2}J'_{d/2}(rt)\cdot r   \Big),
\end{equation*}
which, using the recurrence relation \eqref{eq:Bessel der recurr}, simplifies to the neat expression
\begin{equation*}
g'(t) = r^{1-\frac{d}{2}}t^{d/2}J_{(d-2)/2}(rt).
\end{equation*}
We now apply the Mean Value theorem to the main term of \eqref{eq:K as diff g}, to obtain the estimate
\begin{equation}
\label{eq:main term K MV}
\frac{g(T)-g(T-\eta)}{\eta\cdot T^{d-1}} = \frac{\eta\cdot r^{1-\frac{d}{2}}S^{d/2}J_{(d-2)/2}(rS)}{\eta T^{d-1}} = \frac{r^{1-\frac{d}{2}}S^{d/2}J_{(d-2)/2}(rS)}{T^{d-1}},
\end{equation} 
with some $S\in [T-\eta,T]$. We bound \eqref{eq:main term K MV} from above, on reusing \eqref{eq:Bessel decay 1/sqrt(t)}:
\begin{equation}
\label{eq:g(T)-g(T-eta) est}
\begin{split}
\frac{|g(T)-g(T-\eta)|}{\eta\cdot T^{d-1}} &= O \Big( \frac{r^{1-\frac{d}{2}}S^{d/2}}{(rS)^{1/2}T^{d-1}} \Big) = O \Big( \frac{r^{1-\frac{d}{2}}S^{d/2}}{(rS)^{1/2}T^{d-1}} \Big)
= O \Big( \frac{S^{\frac{d-1}{2}}}{r^{\frac{d-1}{2}}T^{d-1}} \Big) \\&= O\big( (rT)^{-\frac{d-1}{2}} \big) =O \big( T^{-\delta_{1}}\big),
\end{split}
\end{equation}
by the assumption $r=d(x,y) > \frac{1}{T^{1-\delta_{1}}}$, and since $S<T$. 

\vspace{2mm}

The statement of Corollary \ref{cor:decay correlations} (case (a)), with arbitrary $\delta_{1}\le \delta_{0}$, follows upon substituting the estimate \eqref{eq:g(T)-g(T-eta) est} into \eqref{eq:K as diff g}, and recalling that $\eta(T)>T^{\delta_{0}}$ by assumption. The proof for case (b) follows along similar lines, except appealing to Lemma \ref{lem:full band kern asymp}(ii.) instead of Lemma \ref{lem:full band kern asymp}(i.).
\end{proof}

\medskip

\section{Upper bound for defect concentration: Proof of Theorem \ref{t:cub}}
\label{sec:conc bnd proof 1}

\subsection{General upper bound for defect concentration}

In this section we establish an upper bound for the defect concentration that is applicable to general random fields. It will be applied to the random waves in \S~\ref{sec:appl gen bnd upper conc} below. Ahead of stating the general result, we explain the abstract setting.

\vspace{2mm}

In this section we assume that $\Xc\subseteq \Mcc$ is a compact subdomain of $\Mcc$ of volume $|\Xc|$, 
and $(f(x))_{x\in \Xc}$ is a Gaussian random field on $\Xc$, a.s.\ continuous but not necessarily centred. Let $\Hc$ be the reproducing kernel Hilbert space (RKHS) associated to $f(\cdot)$, and $i:\Hc\rightarrow L^{2}(\Xc)$ the canonical inclusion map. We denote the usual operator norm $\|i\|_{\Hc\rightarrow L^{2}}$ of $i$: 
\begin{equation}
\label{eq:I def norm inclusion}
I := \|i\|_{\Hc\rightarrow L^{2}} :=\sup\left\{\|h\|_{L^{2}(\Xc)}:\: h\in \Hc,\, \|h\|_{\Hc}=1 \right\} .
\end{equation}
Recall that $H(\cdot)$ is the sign function \eqref{eq:H Heaviside}, and define the volume-bias of $f(\cdot)$ on $\Xc$:
\begin{equation}
\label{eq:Fc def int H}
\Fc = \Fc(f) :=\frac{1}{|\Xc|} \int\limits_{\Xc} H(f(x) ) dx  .  
\end{equation}
We are now ready to state the abstract upper bound for the defect concentration:

\begin{proposition}
\label{prop:conc upper bnd gen}
Let $\Xc$, $f(\cdot)$, $\Fc$, and $I$ be as above, and further assume that $f(\cdot)$ satisfies $$\var(f(x))\ge 1$$ for every $x\in\Xc$.  
Then, for every $\varepsilon>0$, there exists a number $\delta>0$, only depending on $\varepsilon$, so that 
\begin{equation}
\label{eq:gen upper conc ineq}
\prob(\Fc> \E[\Fc] + \varepsilon) < 3\cdot e^{-\delta |\Xc|/I^{2}}.
\end{equation}

\end{proposition}

Though in what follows we will only require the {\em existence} of a number $\delta=\delta(\varepsilon)$ with properties as in Proposition \ref{prop:conc upper bnd gen}, the proof of Proposition \ref{prop:conc upper bnd gen} yields 
the quantitative dependency $\delta= c_{0}\cdot \varepsilon^{3}$, with some absolute $c_{0}>0$. By a suitable approximation, it would also be possible to prove the analogue of Proposition \ref{prop:conc upper bnd gen} with an arbitrary bounded measurable functional $\Xi(\cdot)$ in place of the sign function $H(\cdot )$ in the definition \eqref{eq:Fc def int H} of $\Fc$.

\begin{proof}

First, we assume that 
\begin{equation}
\label{eq:eps>=I^2/3/X^1/3}
\varepsilon \ge c_{0} \cdot \frac{I^{2/3}}{|\Xc|^{1/3}}
\end{equation}
with an arbitrary fixed number $c_{0}>0$. Indeed, if, otherwise, $\varepsilon < c_{0} \cdot \frac{I^{2/3}}{|\Xc|^{1/3}}$, then the inequality 
\eqref{eq:gen upper conc ineq} is a tautology with 
$\delta = \frac{\varepsilon^{3}}{c_{0}^{3}}$, since, in this case, the r.h.s. of \eqref{eq:gen upper conc ineq} is $>3\cdot e^{-1}>1$. 

\vspace{2mm}

Now we recall the Gaussian isoperimetric inequality (see e.g. ~\cite[Chapter 1]{Ledoux}). Let $(\Hc,\Fcc, \prob)$ be the probability space associated to $f(\cdot)$, and $\Phi$ be the standard Gaussian cdf. Then, for every event $A\in\Fcc$ and $s>0$, one has
\begin{equation*}
\prob(A^{+s}) \ge \Phi\left(\Phi^{-1}(\prob(A))+s\right), 
\end{equation*}
where $A^{+s}$ is the $s$-neighbourhood in $\Hc$-norm (and $\prob(A):=\prob(f\in A)$). In particular, if for some $s>0$ one has $\prob(A^{+s})\le \frac{1}{2}$, then we may infer that 
\begin{equation}
\label{eq:isop ineq upper}
\prob(A) \le \Phi(-s) \le e^{-s^{2}/2}.
\end{equation}
Similarly, one may show the `converse' statement: if $\prob(A)\ge \frac{1}{4}$, then
\begin{equation}
\label{eq:isop ineq upper perturb}
\prob\left(A^{+s}\right) \ge 1- e^{-s^{2}/4}.
\end{equation}

\vspace{2mm}

Let $B$ be the `unstable' event 
\begin{equation*}
B:= \left\{\vol\left(\left\{x\in\Xc:\: |f(x)|<\frac{\varepsilon}{4}\right\}\right)> 2\varepsilon\cdot |\Xc|\right\}.
\end{equation*}
We aim to bound the probability of $B$ from above, by applying the inequality \eqref{eq:isop ineq upper}. We observe that, by the definition \eqref{eq:I def norm inclusion} 
of $I$ as the operator norm of $i$, one has for every $h\in\Hc$,
\begin{equation*}
\|h\|_{L^{2}(\Xc)}^{2} \le I^{2}\cdot \|h\|^{2}_{\Hc},
\end{equation*}
and therefore,
\begin{equation}
\label{eq:vol h large norm}
\vol\left(\left\{x\in\Xc:\: |h(x)|>\frac{\varepsilon}{4}\right\}\right) \le \frac{\|h\|_{L^{2}(\Xc)}^{2}}{(\varepsilon/4)^{2}} \le  \frac{I^{2}\cdot \|h\|^{2}_{\Hc}}{(\varepsilon/4)^{2}},
\end{equation}
by Markov's inequality. Therefore, if, for some $f\in B$ and $h\in\Hc$ one has $\|h\|_{\Hc}<s$ with $s$ given by 
\begin{equation}
\label{eq:s perturb def}
\frac{I^{2}\cdot s^{2}}{(\varepsilon/4)^{2}}=\varepsilon\cdot |\Xc|,
\end{equation}
then it follows that 
\begin{equation}
\label{eq:C def s-neigh}
f+h \in C:=  \left\{\vol\left(\left\{x\in\Xc:\: |f(x)|<\frac{\varepsilon}{2}\right\}\right) >\varepsilon\cdot |\Xc|\right\}.
\end{equation}
To put it differently, with the choice of $s$ as in \eqref{eq:s perturb def}, one has 
\begin{equation}
\label{eq:B+s <= C}
B^{+s}\subseteq C.
\end{equation}
However, since, by the assumptions of Proposition \ref{prop:conc upper bnd gen}, for all $x\in\Xc$ we have $\var\left(f(x)\right) \ge 1$, the expectation of the volume of the set within the definition of
$C$ satisfies
\begin{equation*}
\E\left[  \vol\left(\left\{x\in\Xc:\: |f(x)|<\frac{\varepsilon}{2}\right\}\right)\right] \le \sup\limits_{x\in\Xc} \prob\left(|f(x)| < \frac{\varepsilon}{2}   \right) \cdot |\Xc| \le \frac{\varepsilon}{\sqrt{2\pi}}\cdot |\Xc| < \frac{\varepsilon\cdot |\Xc|}{2},
\end{equation*}
since the pdf of a centred Gaussian attains its maximum at the origin. Hence, by \eqref{eq:B+s <= C} and Markov's inequality, we may bound the probability of $B^{+s}$ (for the particular choice \eqref{eq:s perturb def} for $s$) by
\begin{equation}
\label{eq:prob(B+s)<=1/2}
\prob(B^{+s})\le \prob(C) \le \frac{1}{2}.
\end{equation}
Invoking \eqref{eq:isop ineq upper} with \eqref{eq:prob(B+s)<=1/2} yields
\begin{equation}
\label{eq:prob(B)<=e^-eps^{3}}
\prob(B) \le e^{-s^{2}/2} = e^{-\frac{1}{32}\varepsilon^{3}|\Xc|/I^{2}}< e^{-\frac{1}{64}\varepsilon^{3}|\Xc|/I^{2}},
\end{equation}
on substituting the value of $s$ as in \eqref{eq:s perturb def}.

\vspace{2mm}
Now let $\med(\Fc)$ be the defect median, and define the event $$D:=\{\Fc \le \med(\Fc)\}\setminus B.$$
Recall the assumption \eqref{eq:eps>=I^2/3/X^1/3} (where we have the freedom to choose the constant $c_{0}>0$), 
so that \eqref{eq:prob(B)<=e^-eps^{3}} reads 
$\prob(B)\le \frac{1}{4} ,$ provided that $c_{0}$ is sufficiently large. Therefore, $$\prob(D)\ge \prob\left(\{\Fc \le \med(\Fc)\}\right) - \prob(B) \ge \frac{1}{2}-\frac{1}{4}=\frac{1}{4},$$ meaning that
the underlying assumptions of the inequality \eqref{eq:isop ineq upper perturb}, which we aim to apply with $A:=D$, are satisfied. An application of \eqref{eq:isop ineq upper perturb} then yields the inequality 
\begin{equation}
\label{eq:Ds prob large}
\prob(D^{+s}) \ge 1-e^{-s^{2}/4}.
\end{equation}

Let $s$ be given by \eqref{eq:s perturb def}, and assume that $f\in D$, and for some $h\in \Hc$ one has $\|h\|_{\Hc}<s$, then, on reusing \eqref{eq:vol h large norm} with the given $s$, the defect of $f+h$ on $\Xc$ is
\begin{equation*}
\begin{split}
\Fc({f+h})&=\frac{1}{|\Xc|}\int\limits_{\Xc}H(f(x)+h(x))dx \\&\le \Fc + \frac{\vol\left(\left\{x\in \Xc:\: |f(x)|<\frac{\varepsilon}{4} \right\}\right)}{|\Xc|} + 
\frac{\vol\left(\left\{x\in \Xc:\: |h(x)|>\frac{\varepsilon}{4} \right\}\right)}{|\Xc|}\\&\le \med(\Fc) + 2\varepsilon + \varepsilon = \med(\Fc)+3\varepsilon .
\end{split}
\end{equation*}
Hence, in this case, $$f+h\notin E:=\{\Fc > \med(\Fc)+3\varepsilon\}.$$ 
That is, $D^{+s}\cap E=\varnothing$, i.e. $D^{+s}$ does not intersect $E$. 
Thus, 
\begin{equation}
\label{eq:def>med+3eps}
\prob\left(  \{\Fc > \med(\Fc)+3\varepsilon\}\right) = \prob(E) \le 1-\prob(D^{+s}) \le e^{-s^{2}/4} < e^{-\frac{1}{128}\varepsilon^{3}|\Xc|/I^{2}},
\end{equation}
by \eqref{eq:Ds prob large} and \eqref{eq:s perturb def} (cf.\ \eqref{eq:prob(B)<=e^-eps^{3}}). Since \eqref{eq:def>med+3eps} holds true for all $\varepsilon>c_{0} \cdot \frac{I^{2/3}}{|\Xc|^{1/3}} $ in accordance to \eqref{eq:eps>=I^2/3/X^1/3}, 
choosing $c_{0}$ sufficiently large and integrating w.r.t. $\varepsilon$ shows that 
\begin{equation}
\label{eq:F-med(F) small}
\E[\Fc] \le \med(\Fc) + 4\varepsilon.
\end{equation}
At last, we consolidate the estimates \eqref{eq:def>med+3eps} and \eqref{eq:F-med(F) small}, to obtain
\begin{equation*}
\prob(\Fc\ge 7\varepsilon) \le \prob\left(  \{\Fc > \med(\Fc)+3\varepsilon\}\right) <e^{-\frac{1}{128}\varepsilon^{3}|\Xc|/I^{2}} ,
\end{equation*}
and the statement of Proposition \ref{prop:conc upper bnd gen} finally follows upon replacing $\varepsilon$ by $\frac{\varepsilon}{7}$.
\end{proof}

\subsection{Application of the abstract upper bound to random waves}
\label{sec:appl gen bnd upper conc}

Let $f_{T,\eta}:\Mcc\rightarrow\R$ be the random waves \eqref{eq:fTeta band-lim def}, $K_{T,\eta}:\Mcc\times\Mcc\rightarrow\R$ the reproducing (covariance) kernel \eqref{eq:K covar def expr} 
of $f_{T,\eta}$ and $\Hc=\Hc(T,\eta)$ the corresponding reproducing kernel Hilbert space. One has that
\begin{equation*}
\left\|f_{T,\eta}\right\|_{\Hc} = \|{\bf a}\|_{2},
\end{equation*}
where ${\bf a}\in \R^{N(T,\eta)}$ is the vector ${\bf a} = (a_{j})_{\lambda_{j}\in [T-\eta,T]}$ with the $a_{j}$ as in \eqref{eq:fTeta band-lim def}, $N(T,\eta)$ is the number \eqref{eq:N(T,eta) def} of energy levels in $[T-\eta,T]$, and 
$\|{\bf a}\|_{2}$ is the Euclidean norm of ${\bf a}$. 

\begin{definition}[Restriction of $\Hc$ and the operator norm]
Let $B\subseteq \Mcc$ be a closed subdomain of $\Mcc$.

\begin{enumerate}[i.]
\item Denote the restriction of $\Hc$ to $B$, i.e. an element of $\Hc_{B}$ is a restriction
$h|_{B}$ of some $h\in\Hc$, and, by definition, $$\left\|h|_{B}\right\|_{\Hc_{B}}:= \|h\|_{\Hc}.$$

\item Denote $I_{B}=I_{B}(T,\eta)$ to be the operator norm \eqref{eq:I def norm inclusion} of the natural inclusion $i_{B}:\Hc_{B}\rightarrow L^{2}(B)$.

\end{enumerate}
\end{definition}

Note the slight ambiguity of the norm $\|p\|_{\Hc_{B}}$ for $p\in \Hc_{B}$, in case $p=h|_{B}=\widetilde{h}|_{B}$ for some distinct elements $h,\widetilde{h}\in\Hc$ of the ambient RKHS. However, it is easy to see that this problem cannot occur in our settings, as a (finite) linear combination of Laplace eigenfunctions on $\Mcc$ vanishing on some domain $B$ forces all the coefficients to vanish, a by-product of the Aronszajn unique continuation principle, hence we will ignore this possibility. 

\vspace{2mm}

We will need to determine the dependence of $I_{B}(T,\eta)$ on the parameters. The following result is contained, in essence, in ~\cite{sogge16,ht20}, which we explicate for the purpose of giving it a form suited to the subsequent applications.

\begin{proposition}[Cf. {~\cite[Lemma 4.2]{ht20}}]
\label{prop:IB dep B}
There exists a constant $c=c(\Mcc)$, depending only on $\Mcc$, such that, for every $T\ge 1$, $\eta\in [1,T]$, $r \ge \frac{1}{T}$, and a geodesic ball $B=B_{r}(x)$, one has
\begin{equation}
\label{eq:IB depend}
I_{B}^{2} \le \zeta(T,\eta;r):=  c\cdot\frac{\min\{1, r \eta\} } {N(T,\eta)} .
 \end{equation}

\end{proposition} 

\begin{proof}
The operator norm $I_{B}$ is given explicitly by
\begin{equation}
\label{eq:IB expl}
I_{B}^{2} = \sup\limits_{\|{\bf a}\|_{2}=1} \left\| \sqrt{\frac{|\Mcc|}{N(T,\eta)}}\sum\limits_{\lambda_{j}\in [T-\eta,T]}a_{j}\varphi_{j}  \right\|^{2}_{L^{2}(B)}.
\end{equation}
On the other hand, a straightforward application of \cite[Lemma 4.2]{ht20} (appealing to ~\cite[Eq. (4.1)]{sogge16} and a Cauchy-Schwarz argument) on $$u=\sum\limits_{\lambda_{j}\in [T-\eta,T]}a_{j}\varphi_{j},$$
of norm $\|u\|_{L^{2}(\Mcc)} = \|{\bf a}\|_{2}$, reads 
\begin{equation*}
\|u\|_{L^{2}(B)}^{2} \le \upsilon(T,\eta;r) \cdot \|u\|_{L^{2}(\Mcc)}^{2} =  \upsilon(T,\eta;r) \cdot \|{\bf a}\|_{2}^{2} ,
\end{equation*}
with $$\upsilon(T,\eta;r) := \begin{cases} c\cdot r\cdot\eta &\frac{1}{T}\le r\le \frac{1}{\eta} ,\\
1 &\frac{1}{\eta}<r < \inj(\Mcc).
\end{cases}$$
Hence,
\begin{equation*}
\left\|\sqrt{\frac{|\Mcc|}{N(T,\eta)}}u   \right\|_{L^{2}(B)}^{2} = \frac{|\Mcc|}{N(T,\eta)} \left\| u\right\|_{L^{2}(B)}^{2} \le \frac{|\Mcc|}{N(T,\eta)} \upsilon(T,\eta;r)\cdot \|{\bf a}\|_{2},
\end{equation*}
which, in light of \eqref{eq:IB expl}, is \eqref{eq:IB depend}, with 
\begin{equation*}
\zeta(T,\eta;r) = \frac{|\Mcc|}{N(T,\eta)} \cdot\upsilon(T,\eta;r). \qedhere
\end{equation*}
\end{proof}

We are now in a position to give a proof for Theorem \ref{t:cub}:

\begin{proof}[Proof of Theorem \ref{t:cub}]
We first address case (a) of Theorem \ref{t:cub}, and assume that $u\in\R$ is fixed. By the symmetry of $f_{T,\eta}(\cdot)$ with respect to negation, one has that
\begin{equation*}
\prob( \Dc_{T,u}(x;r) <  -\eps) = \prob( \Dc_{T,-u}(x;r) > \eps).
\end{equation*}
Therefore, it is sufficient to prove that $\prob( \Dc_{T,u}(x;r) >  \eps)$ is bounded by the r.h.s.\ of \eqref{eq:conc upper bnd}, and decrease the constant $c>0$ to absorb the incurred multiplicative factor of $2$.

\vspace{2mm}

We aim to apply Proposition \ref{prop:conc upper bnd gen} to (a rescaled version of) $f_{T,\eta}$ restricted to the ball $B_x(r)$. We may invoke Proposition \ref{prop:asymp covar waves}(i.) to infer that $$K_{T,\eta}(x,x)>\frac{1}{2}$$ for all $x\in\Mcc$, provided that $\eta>C$ with $C>0$ sufficiently large, and $T$ is sufficiently large. Then the random field 
\begin{equation}
\label{eq:g=rescaled f}
g : B_x(r) \to \R , \qquad g(x)=g_{u;T,\eta}(x):=  \sqrt{2}\cdot (f_{T,\eta}(x) - u)
\end{equation}
satisfies the assumption postulated by Proposition \ref{prop:conc upper bnd gen}, 
i.e.\ that $\var(g(x))\ge 1$. Hence, Proposition \ref{prop:conc upper bnd gen} yields the bound 
\begin{equation}
\label{eq:prob Def>eps<exp}
\prob( \Dc_{T,u}(x;r) >  \eps) < e^{- c_1\delta |B|/I_{B}^{2}},
\end{equation}
where  $c_1 > 0$ is an absolute constant (incurred when passing from $f$ to $g$), $I_{B}$ is the operator norm of the inclusion $i:\Hc_{B}\rightarrow L^{2}(B)$ as in \eqref{eq:I def norm inclusion}, and $\delta=\delta(\eps)$ is as prescribed by Proposition~\ref{prop:conc upper bnd gen}.

We now evaluate the exponent $|B|/I_{B}^{2}$ that appears on the r.h.s.\ of \eqref{eq:prob Def>eps<exp}. First, by Weyl's law \eqref{eq:Weyl quant}, we have that $$N(T,\eta)\asymp \eta\cdot T^{d-1},$$ with the involved constants depending only on $\Mcc$. Further, for $x\in\Mcc$, $0<r<\inj(\Mcc)$, one has $$|B_{r}(x)| \ge c_{1}r^{d},$$ for some $c_{1} > 0$ only depending on $\Mcc$. Hence, on taking into account Proposition \ref{prop:IB dep B},
\begin{equation*}
\frac{|B|}{I^{2}} \ge \frac{r^{d}}{\zeta(T,\eta;r)} \ge c_{2}\cdot \begin{cases}
(rT)^{d-1} &\frac{1}{T}\le r\le \frac{1}{\eta}\\ r^{d}\eta T^{d-1} , &  \frac{1}{\eta}<r<\inj(\Mcc) .
\end{cases}
\end{equation*}
for some $c_{2}>0$ only depending on $\Mcc$. 

\vspace{2mm}

This concludes the proof in the general case (a) of Theorem \ref{t:cub} (arbitrary manifold, $\eta \ge C$), with $c= c_1 \cdot \delta\cdot c_{2}$. Case (b) (round sphere, with arbitrary $\eta\ge 1$) follows along the same lines, this time invoking Proposition \ref{prop:asymp covar spher}(i.) in place of Proposition \ref{prop:asymp covar waves}(i.), except that, while passing from $f_{T,\eta}$ to $g$ as in \eqref{eq:g=rescaled f}, 
there is no need to scale it up by $\sqrt{2}$ to increase the variance to $\ge 1$. 
\end{proof}

\medskip

\section{Lower bound for volume-bias concentration: Proof of theorems \ref{t:clb}-\ref{t:clb2}}
\label{sec:conc bnd proof}

In this section we prove the lower bounds for the volume-bias concentration in theorems \ref{t:clb} and \ref{t:clb2}. The proofs of both results involve the construction of a certain event on which $f_{T}(\cdot)=f_{T,\eta}(\cdot)$ is not sign-balanced or volume-balanced on the geodesic ball $B_x(r) \subseteq \Mcc$. In turn, this involves constructing a deterministic `barrier' function, belonging to the support of the probability measure corresponding to $f_{T}(\cdot)$ that has a large defect/volume-bias on $B_x(r)$, akin to the Nazarov-Sodin's {\em barrier method} ~\cite{NaSoAJM}, introduced for the purpose of the study of the nodal domain count. Constraining $f_{T}(\cdot)$ to be `close' to the barrier provides us with the desired event.

\smallskip
The proofs of theorems \ref{t:clb} and \ref{t:clb2} make use of distinct barriers. For the former, tailored to the volume-bias at non-zero levels, the barrier is relatively simple: we take it to be a (suitably rescaled) reproducing kernel $K_T(x,\cdot)$ of $f_T$ centred at the given point $x \in \Mcc$. On the other hand, for the latter, tailored to the defect, the construction involves a superposition of \textit{three} reproducing kernels associated to \textit{modifications} of $f_T$ with restricted energy levels; the details of this take up the following two subsections. We conclude the proofs of theorems \ref{t:clb}-\ref{t:clb2} in \S~\ref{sec:lower defect-concentration proof}.

\subsection{Existence of Euclidean sign-barrier}
\label{sec:Euclid barrier}
We now begin our construction of a suitable sign-barrier function to be used within the proof of Theorem \ref{t:clb2}, inspired by a the construction of a sign-imbalanced {\em toral} eigenfunction ~\cite[\S 5]{kwy21}.  At a first stage we will exhibit a {\em Euclidean} sign-barrier, a {\em periodic} (rather, invariant) function on $\R^{d}$ that has a large defect.

\vspace{2mm}
For $d\ge 2$ we will find a sign-barrier function $w:\R^{d}\rightarrow \R$ within the $1$-parameter family
\begin{equation}
\label{eq:wt barrier def}
w_{t}(x)=w_{t}^{d}(x) := w_{0}(x)+t\cdot p(x),  \qquad t \ge 0 
\end{equation}
where $w_{0}$ and $p$ are defined as follows. Let $v_{1}:=1$, $v_{2}:=\omega=e^{2\pi i/3}$, and $v_{3}:=\omega^{2}$ be the three roots of unity in $\C\cong \R^{2}$. We view these as vectors in $\R^{2}$, and then embed them in $\R^{d}$ by adding zero components: $v_{1}=(1,0,\ldots, 0) \in \R^{d}$, and $v_{2,3} = (-\frac{1}{2},\pm\frac{\sqrt{3}}{2},0,\ldots, 0 )\in\R^{d}$. We set
\begin{equation}
\label{eq:w0 def}
w_{0}(x):=-\cos\left(2\pi \left\langle x,v_{1}\right\rangle\right)=-\cos(2\pi x_{1}),
\end{equation}
and 
\begin{equation}
\label{eq:p def}
p(x):=-\cos\left(2\pi \left\langle x,v_{2}\right\rangle\right)-\cos\left(2\pi \left\langle x,v_{3}\right\rangle\right).
\end{equation}
The negative sign of $w_{0}$ and $p$ was added so that, as we will show, the defect of $w_{t}$ will be {\em positive}.

\vspace{2mm}
We recall that, given a lattice $\Lambda$ in $\R^{d}$, its dual lattice is defined as $$\Lambda^{*}:=\left\{y\in \R^{d}:\: \forall x\in\Lambda.\: \langle x,\, y\rangle \in \Z\right\}.$$ The following lemma, whose proof is omitted, 
collects some well-known facts on the hexagonal `honeycomb' lattice $\Lambda\subseteq\C\cong\R^{2}$ generated by the two complex numbers $v_{1}=1,v_{2}=\omega$.

\begin{lemma}
\label{lem:Lambda lattice}
Let $\Lambda$ be the lattice $$\Lambda = \left\langle 1, \omega \right\rangle \subseteq\C$$ generated by the two complex numbers $1,\omega$.

\begin{enumerate}[i.]

\item The lattice $\Lambda$ gives rise to a tiling of $\R^{2}$ by equilateral triangles. It acts by translations on $\R^{2}$, and its action has a fundamental domains that is a regular hexagon of area $\frac{\sqrt{3}}{2}$, that tessellates~$\R^{2}$. 

\item The dual lattice to $\Lambda$ is the hexagonal lattice $\Lambda^{*}= \langle 2, 1-\frac{1}{\sqrt{3}}i\rangle$, and one has $2\Lambda\subsetneq \Lambda^{*}$.

\item For every $t>0$, the function $w_{t}$ is invariant w.r.t.\ the action of $\Lambda^{*}$ on $\C$ by translations. 

\item The group action of $\Lambda^{*}$ has a fundamental domain $\Pi$ that is a regular hexagon of side length $\frac{2}{3}$ and area $\frac{2}{\sqrt{3}}$, illustrated in Figure \ref{f}.

\end{enumerate}

\end{lemma}

\begin{figure}[htp]

\centering
\includegraphics[width=.33\textwidth]{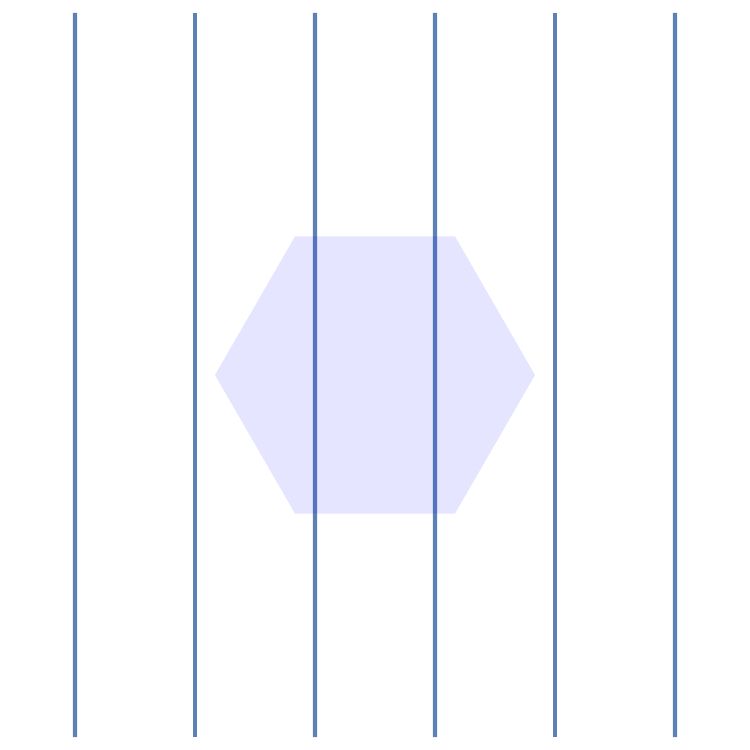}\hfill
\includegraphics[width=.33\textwidth]{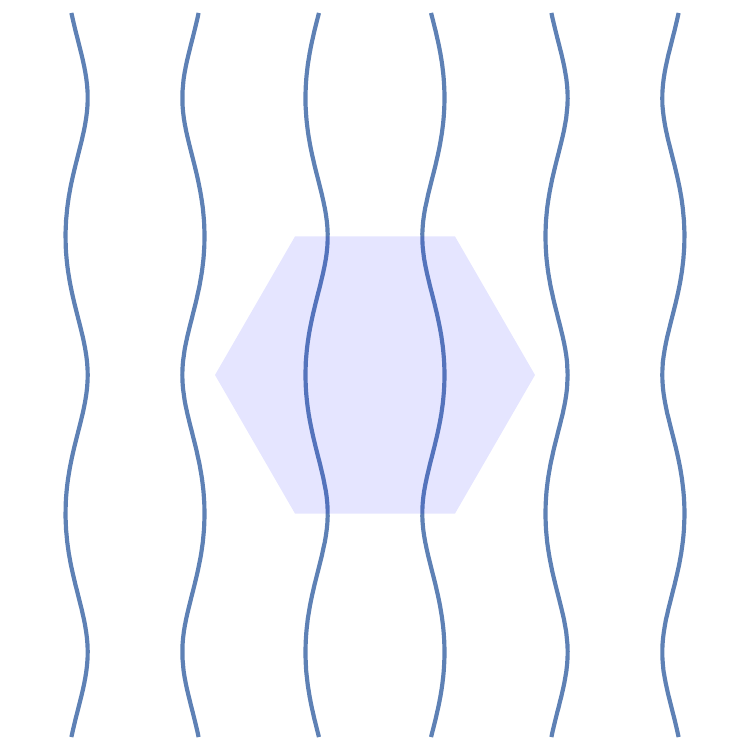}\hfill
\includegraphics[width=.33\textwidth]{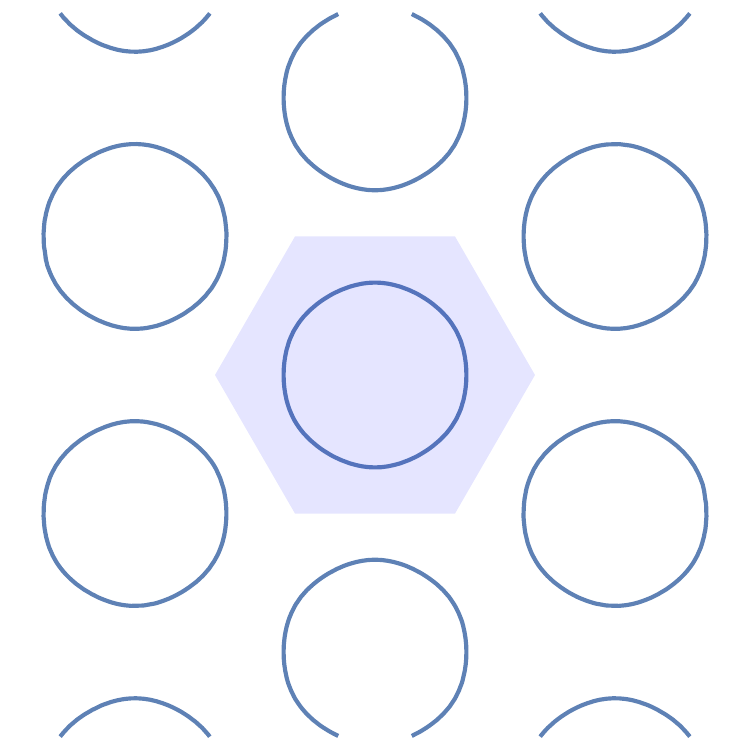}

\caption{The nodal set of $w_t$ in $d=2$, for $t = 0$ (left), $t=0.2$ (centre), $t=1$ (right). The fundamental domain is shaded. Proposition \ref{prop:Euclid sign-bar} asserts that $w_t$ 
has a positive defect on its fundamental domain for sufficiently small $t > 0$.}
\label{f}
\end{figure}

\begin{proposition}[Euclidean sign-barrier]
\label{prop:Euclid sign-bar}
Let $d\ge 2$, and for $t>0$ let $w_{t}:\R^{d}\rightarrow \R$ be given by \eqref{eq:wt barrier def}. There exist absolute constants $\varepsilon=\varepsilon_{0}>0$, $t=t_{0}\in (0,1)$, and $r_{0}=r_{0}(d)$ with the following property: for every $r>r_{0}$, one has
\begin{equation*}
\widetilde{\Dc}_{w_t,\varepsilon}(0;r) > \varepsilon,
\end{equation*}
recalling that $\widetilde{\Dc}_{f,u}(x;r)$ is the uncentred volume-bias defined in \eqref{eq:uncentr vol bias def}.
\end{proposition}

Towards the proof of Proposition \ref{prop:Euclid sign-bar} we will require the following general lemma that facilitates ruling out that the defect of a $1$-parameter family of smooth functions vanishes identically on the fundamental domain $\Pi$, so that a sign-barrier will be found inside the said family of functions. Although only stated for $d=2$ (as used below), it is easy to generalise it to arbitrary dimensions.

\begin{lemma}
\label{lem:defect derivatives}
Let $\phi_{0},\psi:\Gamma\rightarrow \R$ be two smooth functions defined on an open domain $\Gamma\subseteq\R^{2}$, and $\Pi\subseteq \Gamma$ a compact domain of area $|\Pi|=\area(\Pi)\in (0,\infty)$, 
so that $\phi_{0}$ has no critical zeros on $\Pi$. For $t\in \R$ define 
 \[ \phi_{t}(x):=\phi_{0}(x)+t\cdot \psi(x) \qquad \text{and} \qquad D(t):= \frac{1}{|\Pi|} \int\limits_{\Pi} H(\phi_{t}(x)) dx. \]
 Then
\[  D'(0) = \frac{2}{|\Pi|}\int\limits_{\phi_{0}^{-1}(0)\cap \Pi} \frac{\psi(x)}{\|\nabla  \phi_{0}(x)\|} dx \]
and
\begin{equation}
\label{eq:D''(0) expr int}
\begin{split}
D''(0) &= \frac{2}{|\Pi|}\int\limits_{\phi_{0}^{-1}(0)\cap \Pi} 
\bigg(-2\frac{\psi(x)\cdot\left\langle \nabla\phi_{0}(x),\nabla\psi(x)\right\rangle}{\|\nabla\phi_{0}(x)\|^{3}}dx + \frac{\psi(x)^{2}}{\|\nabla\phi_{0}(x)\|^{4}}\cdot \partial_{\nabla\phi_{0}(x)}\left[\left\|\nabla\phi_{0}(x)\right\|\right] 
\\& \qquad \qquad \qquad \qquad \qquad \qquad + \frac{\psi(x)^{2}}{\|\nabla  \phi_{0}(x)\|^{2}} \cdot \kappa(x)\bigg)dx,
\end{split}
\end{equation}
where $\partial_v$ denotes the directional derivative and $\kappa(x)$ is the curvature of a point $x \in \phi^{-1}(0)$, signed in accordance to the orientation defined by the unit normal $\vec{N}(x) = \frac{\nabla \phi_{0}(x)}{\|\nabla \phi_{0}(x)\|}$.
\end{lemma}

The proof of Lemma \ref{lem:defect derivatives} will be given in Appendix \ref{sec:defect derivatives}. We are now ready to give a proof for Proposition \ref{prop:Euclid sign-bar}.

\begin{proof}[Proof of Proposition \ref{prop:Euclid sign-bar}]

Recall that $\Pi$ is a hexagonal fundamental domain of the group action of $\Lambda^{*}$ on $\R^{2}$, preserving $w_{t}$ for every $t$, in the $d=2$ context (see Lemma \ref{lem:Lambda lattice}). First, we claim that it is sufficient to show that the defect 
\begin{equation*}
D(t) := \frac{1}{|\Pi|}\int\limits_{\Pi} H(\omega_{t}(x)) dx
\end{equation*}
of $\omega_{t}$ on $\Pi\subseteq \R^{2}$ does not vanish identically as a function of $t>0$, numerically verified ~\cite[\S~4]{kwy21} for the choice $t=1$. 

Indeed, suppose this were true. Then since $w_{t}$ is independent of the last $(d-2)$ coordinates, we may assume that $d=2$, and by the continuity of the function 
$$s\mapsto \widetilde{D}_{s}(t) := \frac{1}{|\Pi|} \int\limits_{\Pi} H(\omega_{t}(x) - s) dx,$$ it follows that 
if, for some $t>0$, one has $D(t)>0$, then that would imply that $\widetilde{D}_{\varepsilon}(t)  > \frac{1}{2}D(t) > 0$ with $\varepsilon>0$ sufficiently small. Now choose $r_{0}$ sufficiently large, and tile a large ball of radius $r>r_{0}$ by translates of $\Pi$, save for a small corridor around the boundary. Then, for $w=w_{t_{0}}$, as $r \to \infty$
$$\widetilde{\Dc}_{w,\varepsilon}(0,r) = \widetilde{D}_{\varepsilon}(t_{0}) + o(1),$$ 
so that for $r_{0}$ sufficiently large and $r>r_{0}$, $\widetilde{\Dc}_{w,\varepsilon}(0;r)$ is positive and bounded away from $0$. Thus, further decreasing $\varepsilon$ will complete the proof.

\vspace{2mm}
It remains to prove that $D(t)$ does not vanish identically. For $t=0$, we have $w_{0}(x)=-\cos(2\pi x_{1})$, and by the elementary geometry involved, the nodal set
\begin{equation}
\label{eq:nodal line w0 vectical}
\Pi\cap w_{0}^{-1}(-\infty,0) = \Pi\cap \Big\{x\in \R^{2}:\: x_{1}\in \Big(-\frac{1}{4},\frac{1}{4} \Big) \Big\} =\Big\{ \pm \frac{1}{4}  \Big\}\times \Big[-\frac{1}{\sqrt{3}},\frac{1}{\sqrt{3}}  \Big]
\end{equation}
is a rectangle whose side are of length $\frac{1}{2}$ (horizontal), and $2\cdot \frac{2}{3}\sin(\pi/3)= \frac{2}{\sqrt{3}}$ (vertical), 
of area $$\area\left( \left\{x\in\Pi :\: w_{0}(x)<0\right\}\right) =\frac{1}{\sqrt{3}}=\frac{1}{2}\cdot \area(\Pi).$$ Therefore, 
$$D(0) = \frac{1}{\area(\Pi)}\cdot \left(\area\left( \left\{x\in\Pi :\: w_{0}(x)>0\right\}\right)-\area \left(\left\{x\in\Pi :\: w_{0}(x)<0\right\}\right)\right) = 0.$$

\vspace{2mm}

Next, we evaluate the first two derivatives of $D(t)$ at $t=0$, with the help of Lemma \ref{lem:defect derivatives} with $\phi_{0}(x):=w_{0}(x)$ as in \eqref{eq:w0 def}, and $\psi(x):=p(x)$ as in \eqref{eq:p def} (and $\Pi$ as in Lemma \ref{lem:Lambda lattice}(iv.)). We will see that $D'(0)=0$, but $D''(0)>0$, implying that the assertion of Proposition \ref{prop:Euclid sign-bar} follows for $t_{0}>0$ sufficiently small.

First, we apply Lemma \ref{lem:defect derivatives}(i.) to evaluate $D'(0)$. The nodal line $\phi_{0}^{-1}(0) = w_{0}^{-1}(0)$ consists of two straight vertical segments \eqref{eq:nodal line w0 vectical}.  
Then, on the nodal line, we substitute $x_{1}=\pm \frac{1}{4}$ into \eqref{eq:p def} to yield
\begin{equation}
\label{eq:p(x) nodal line}
p(x) = -\sqrt{2}\cos\big(\sqrt{3}\pi x_{2}\big).
\end{equation}
It is also straightforward to compute $\|\nabla w_{0}(x)\| \equiv 2\pi$ for $x \in w_{0}^{-1}(0)$. Hence, on $x\in w_{0}^{-1}(0)$, 
$$\frac{p(x)}{\|\nabla w_{0}(x)\|} = \frac{1}{\sqrt{2}\pi}\cos\big(\sqrt{3}\pi x_{2}\big), $$ which is a periodic function of $x_{2}$ of period $\frac{2}{\sqrt{3}}$. Since the length of each of the vertical lines of 
\eqref{eq:nodal line w0 vectical} is also $\frac{2}{\sqrt{3}}$, the expression $\frac{p(x)}{\|\nabla w_{0}(x)\|}$ integrates to $0$, confirming that $D'(0)=0$, by Lemma \ref{lem:defect derivatives}(i.).

\vspace{2mm}
Next, we apply Lemma \ref{lem:defect derivatives}(ii.) to evaluate $D''(0)$. First, since the curvature of the vertical lines \eqref{eq:nodal line w0 vectical} vanishes, we may ignore the $3$rd term on the r.h.s.\ of \eqref{eq:D''(0) expr int}. It is straightforward to evaluate the gradient of $w_{0}(x)$ to be $\nabla w_{0}(x) = \left(2\pi \sin(2\pi x), 0 \right)$, hence $$\|\nabla w_{0}(x)\| = 2\pi |\sin(2\pi x)| = \pm 2\pi \sin(2\pi x).$$
Therefore, on the nodal line, $$\partial_{\nabla w_{0}(x)} \left[\|\nabla w_{0}(x)\|\right] = \pm \frac{\partial}{\partial x} \left[\|\nabla w_{0}(x)\|\right] = \pm 4\pi^{2}\cos(2\pi x)\big|_{x=\pm 1/4} = 0,$$
showing that the $2$nd term on the r.h.s. of \eqref{eq:D''(0) expr int} vanishes too. 

It then remains to evaluate the $1$st term on the r.h.s. of \eqref{eq:D''(0) expr int}. On $x\in w_{0}^{-1}(0)$ we have
\begin{equation*}
\begin{split}
\left\langle  \nabla\phi_{0}(x),\nabla\psi(x)   \right\rangle &= 2\pi^{2}\big\langle (1,0),  \big(\sqrt{2}\cos\big(\sqrt{3}\pi x_{2}\big),\sqrt{6}\cos\big(\sqrt{3}\pi x_{2}\big) \big)  \big\rangle
\\&=2\sqrt{2}\pi^{2}\cos\big(\sqrt{3}\pi x_{2}\big),
\end{split}
\end{equation*}
and thus, on recalling \eqref{eq:p(x) nodal line}, here $$\frac{\psi(x)\cdot\left\langle \nabla\phi_{0}(x),\nabla\psi(x)\right\rangle}{\|\nabla\phi_{0}(x)\|^{3}} = -\frac{4\pi^{2} \cos\left(\sqrt{3}\pi x_{2}\right)^{2}}{8\pi^{3}} = 
-\frac{1}{2\pi} \cos\big(\sqrt{3}\pi x_{2}\big)^{2} .$$ Integrating that expression along the nodal line \eqref{eq:nodal line w0 vectical}, and substituting into \eqref{eq:D''(0) expr int} finally yields
\begin{equation*}
D''(0) = \frac{4}{|\Pi|}\cdot \frac{1}{2\pi \cdot 2}\cdot 2\cdot\frac{2}{\sqrt{3}} = \frac{4}{\sqrt{3}\pi |\Pi|} = \frac{2}{\pi} >0 \qedhere
\end{equation*} 
\end{proof}

\subsection{Existence of sign-barrier on $\Mcc$}

Recall that $\Hc=\Hc(T,\eta)$ is the RKHS corresponding to $f_{T,\eta}$ of \eqref{eq:fTeta band-lim def}, as in \S~\ref{sec:appl gen bnd upper conc}. The following proposition asserts the existence of a sign-barrier, with properties analogous to the Euclidean sign-barrier of Proposition \ref{prop:Euclid sign-bar}, around every reference point $x\in\Mcc$. 

\begin{proposition}[Sign-barrier]
\label{prop:sign-bar exist man}
There exists a constant $\eps = \varepsilon_{1}>0$, and constants $C>0$, $T_{0}>0$ sufficiently large, only depending on $\Mcc$, with the following property. Let $T>T_{0}$, $\eta\in [1,T]$, $r>0$ satisfy either (a) $\frac{C}{T}<r<\frac{1}{C} \cdot \frac{1}{T^{\frac{d}{d+1}}}$ and $\eta> C(rT)^{d-1}$, or (b) $\Mcc=\Sc^{d}$, and $\frac{C}{T}<r < \frac{1}{C\sqrt{T}}$. Then, for every $x\in\Mcc$, there exists an element $h=h_{x,r}\in\Hc$ of the RKHS of norm
\begin{equation}
\label{eq:sign bar norm bnd}
\|h\|^2_{\Hc} < C\cdot \max\{ 1,r^2 T \cdot\eta  \}\cdot (rT)^{2(d-1)},
\end{equation}
satisfying
\begin{equation}
\label{eq:sign bar pos bias}
\widetilde{\Dc}_{h;\eps}(x;r) > \varepsilon.
\end{equation}
\end{proposition}

Towards the proof of Proposition \ref{prop:sign-bar exist man} we will require the following {\em geometric} lemma.

\begin{lemma}
\label{lem:distortions plane waves}

\begin{enumerate}[i.]

\item Let $r,R >0$ be two numbers so that $r<R$, and $\xi\in \R^{d}$ so that $\|\xi\|=R$. Then, uniformly for $y\in B_{r}(0)\subseteq\R^{d}$, one has
\begin{equation}
\label{eq:dist distortion R2}
d_{2}(\xi,y) = R-\frac{1}{R}\langle \xi,y\rangle + O\Big( \frac{r^{2}}{R}  \Big),
\end{equation}
where $d_{2}(\cdot,\cdot)$ is the standard Euclidean distance in $\R^{d}$.

\item Let $x\in\Mcc$, and $r,\widetilde{r}\in (0,\inj(\Mcc))$ two radii below the injectivity radius of $\Mcc$, so that $\widetilde{r}>2r$. Let
$v\in \Sc^{1}\subseteq \R^{d}$ be a unit vector, and $\xi=\exp_{x}(\widetilde{r}\cdot v)$. Then, uniformly on the ball $w\in B_{r}(0)\subseteq T_{x}(\Mcc)$ in the tangent space, one has
\begin{equation}
\label{eq:distortion Riemann}
d(\xi,\exp_{x}(w)) = \widetilde{r}-\langle w,v\rangle + O\Big(\frac{r^{2}}{\widetilde{r}}  \Big).
\end{equation} 

\end{enumerate}

\end{lemma}

\begin{proof}[Proof of Lemma \ref{lem:distortions plane waves}]

First we prove Lemma \ref{lem:distortions plane waves}(i.). By rotating $\R^{d}$ if necessary, we may assume that $\xi=(R,0,\ldots,0)$, whence \eqref{eq:dist distortion R2} reads
\begin{equation}
\label{eq:dist distortion R2 horizontal}
d_{2}(\xi,y) = R-y_{1} + O\left( \frac{r^{2}}{R}  \right),
\end{equation}
with $y=(y_{1},\ldots, y_{n})\in B_{r}(0)$. It is easy to derive the estimate \eqref{eq:dist distortion R2 horizontal} via a straightforward application of Pythagoras's theorem.

\vspace{2mm}

Now we turn to proving Lemma \ref{lem:distortions plane waves}(ii.). By Lemma \ref{lem:distortions plane waves}(i.) we have that
\begin{equation}
\label{eq:d2 distortion proj}
d_{2}(\widetilde{r}\cdot v,w) = \widetilde{r} - \langle w,v\rangle + O\Big( \frac{r^{2}}{\widetilde{r}}  \Big),
\end{equation}
on $w\in B_{r}(0)$. We then claim that
\begin{equation}
\label{eq:d2-d bound small ball}
\left|d_{2}(\widetilde{r}\cdot v,w) - d(\xi,\exp_{x}(w))\right| = O\Big( \frac{r^{2}}{\widetilde{r}}  \Big),
\end{equation}
which, together with \eqref{eq:d2 distortion proj} implies the statement \eqref{eq:distortion Riemann} of Lemma \ref{lem:distortions plane waves}(ii.), via the triangle inequality. 

Indeed, let us consider the function $$w\mapsto F(w) =  d_{2}(\widetilde{r}\cdot v,w) - d(\xi,\exp_{x}(w)).$$ Evidently, for $w=0$, 
$$d_{2}(\widetilde{r}\cdot v,0) =  \widetilde{r} = d(\xi,\exp_{x}(0)),$$ thus $F(0)=0$, and, in addtion, it is easy to check that, 
$$\nabla_{w} d_{2}(\widetilde{r}\cdot v,w)\big|_{w=0} = \nabla_{w} d(\xi,\exp_{x}(w))\big|_{w=0} = -v,$$ so $\nabla F(0) =0$. 
Further, an explicit computation shows that the Hessian of $d_{2}(\widetilde{r}\cdot v,\cdot)$ is bounded (entry-wise) by $O(\frac{1}{\widetilde{r}-r}) = O(\frac{1}{\widetilde{r}})$ on $B_{r}(0)$, and a forteriori, so is the other Hessian $\nabla^{2} d(\xi,\exp_{x}(\cdot))$, via the Hessian comparison theorem (it is possible to express the leading term for the latter in terms of the curvature tensor of $\Mcc$), hence \eqref{eq:d2-d bound small ball} follows. 
\end{proof}

\begin{proof}[Proof of Proposition \ref{prop:sign-bar exist man}]

Let $C>0$ be a sufficiently large constant to be chosen later, $$v_{1}=1,v_{2}=\omega,v_{3}=\omega^{3}\in\Sc^{d}\times \{0\}^{d-2}\subseteq \R^{d}$$ be the three roots of unity of 
degree $3$ of \S~\ref{sec:Euclid barrier}, augmented by $(d-2)$ zeros, $\varepsilon_{0}>0$ and
$t_{0}\in (0,1)$ the absolute constants prescribed in Proposition \ref{prop:Euclid sign-bar}, $\gamma_{d}$ as in \eqref{eq:gammad offset def} of Proposition \ref{prop:asymp covar waves}, and $\eta'\in [1,\eta]$ to be chosen later. 
For every $x\in \Mcc$ let 
\begin{equation}
\label{eq:iota ident Rd tangent}
\iota_{x}:\R^{d}\xrightarrow{\sim}T_{x}(\Mcc)
\end{equation}
be an identification of $\R^{d}\cong T_{x}(\Mcc)$, and $\exp_{x}(\cdot)$ the exponential map in $\Mcc$ based at $x$. 
(There is no continuous choice of $\iota_{x}$ as a function of $x$, but a measurable one will work.)

Given $x\in \Mcc$ and $r>0$, we set 
\begin{equation}
\label{eq:r tild def}
\widetilde{r}:= \min\left\{r'\ge C\cdot r^{2}T:\: r'\cdot T+\gamma_{d}\in 2\pi\Z\right\},
\end{equation} 
with a sufficiently large constant $C>0$ to be determined later, and where $\gamma_{d}$ is as in \eqref{eq:gammad offset def}. Let
\begin{equation}
\label{eq:s param def}
s:=\frac{2}{c_{d}\varepsilon_{0}}\cdot \left(\widetilde{r}\cdot T\right)^{\frac{d-1}{2}}
\end{equation}
with $c_{d}$ as in \eqref{eq:covar rand wav asymp off} of Proposition \ref{prop:asymp covar waves}(ii.), and $\alpha_{1}:=1$ and $\alpha_{2}=\alpha_{3}=t_{0}$. With this notation we define 
\begin{equation}
\label{eq:h sign-bar def gen}
\begin{split}
h(\cdot)&=h_{x,r}(\cdot) = -s\cdot\sum\limits_{j=1}^{3} \alpha_{j}K_{T,\eta'}\left(\exp_{x}\left(\iota_{x}\left(\widetilde{r}\cdot v_{j}\right)\right),\cdot\right) \\&= -\frac{2}{c_{d}\varepsilon_{0}}\left(\widetilde{r}\cdot T\right)^{\frac{d-1}{2}} \sum\limits_{j=1}^{3} \alpha_{j}K_{T,\eta'}\left(\exp_{x}\left(\iota_{x}\left(\widetilde{r}\cdot v_{j}\right)\right),\cdot\right) \in\Hc,
\end{split}
\end{equation}
a superposition of three waves (approximating plane waves) emanating from points on $\Mcc$ at distance $\widetilde{r}$ from the given point $x$. The function $h(\cdot)$ is illustrated in Figure \ref{fig:sign-barrier}.

\vspace{2mm}

In what follows, we claim that, for an appropriate choice of the parameters ($C$ sufficiently large, and $\eta'$ chosen below), $h$ is an element of $\Hc$, whose properties, \eqref{eq:sign bar norm bnd} and \eqref{eq:sign bar pos bias}, are asserted in Proposition \ref{prop:sign-bar exist man}. Namely, we will show that, inside the small ball $B_{x}(r)$ around $x$, properly scaled $h$ approximates the Euclidean sign-barrier, constructed in Proposition \ref{prop:Euclid sign-bar}.
First, to bound the norm of $h\in \Hc$, we use the definition \eqref{eq:K covar def expr} of $K_{T,\eta}(\cdot,\cdot)$ with $\eta'$ instead of $\eta$, together with the definition of the norm in $\Hc$, to yield that
\begin{equation}
\label{eq:kernel centred norm bnd}
\left\|K_{T,\eta'}\left(\exp_{x}\left(\iota_{x}\left(\widetilde{r}\cdot v_{j}\right)\right),\cdot\right)\right\|_{\Hc} \le \sqrt{\frac{N_{T,\eta}}{N_{T,\eta'}}}. 
\end{equation}
Bearing in mind that $t_{0}\in [0,1]$ (hence $\alpha_{1}=1$, and $\alpha_{j}\in [0,1]$, $j=2,3$), an application of the triangle inequality implies
\begin{equation}
\label{eq:hxr bnd gen}
\|h_{x,r}\|_{\Hc} \le 3s\cdot \sqrt{\frac{N_{T,\eta}}{N_{T,\eta'}}} = \frac{6}{c_{d}\varepsilon_{0}}\cdot \left(\widetilde{r}\cdot T\right)^{\frac{d-1}{2}}\cdot \sqrt{\frac{N_{T,\eta}}{N_{T,\eta'}}},
\end{equation} 
by the choice \eqref{eq:s param def} of the parameter $s$.

\begin{figure}[htp]

\centering
\includegraphics[width=.35\textwidth]{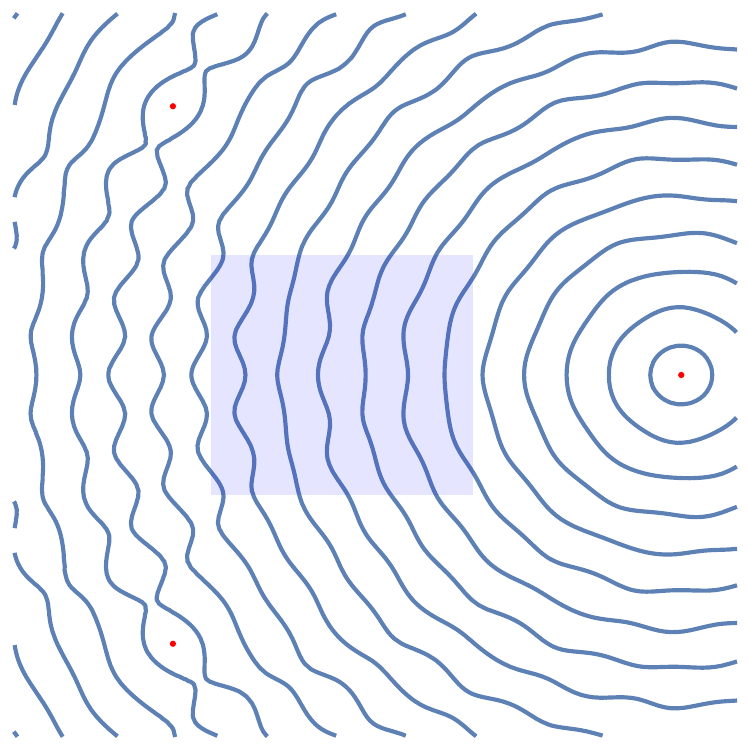}\hspace{1.5cm}
\includegraphics[width=.35\textwidth]{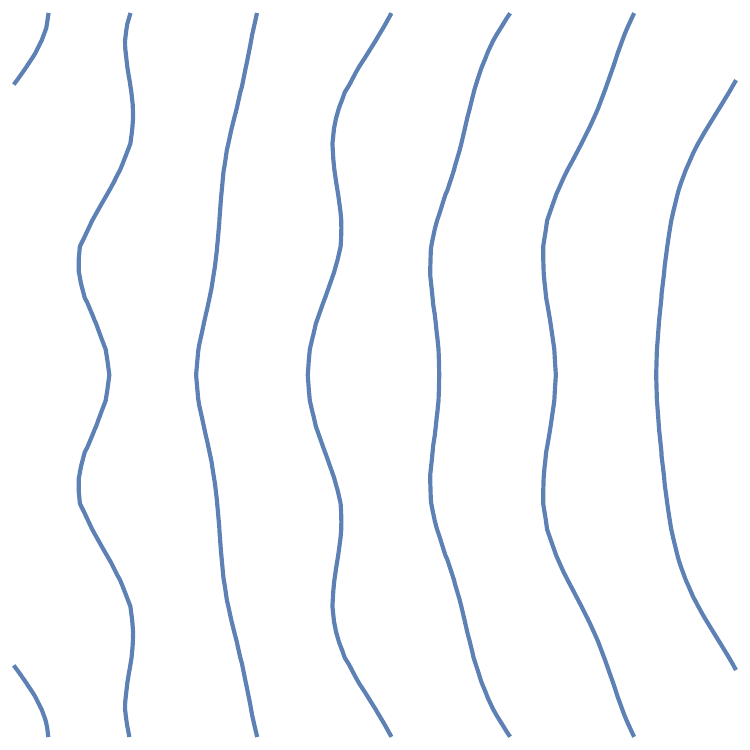}
\caption{The nodal set of the sign-barrier $h$, constructed within the proof of Proposition \ref{prop:sign-bar exist man} projected onto $\R^2$, and a close-up view of the shaded square. The function is a weighted superposition of three monochromatic kernels centred at the red dots; on the square it approximates the function $w_t$ of Proposition \ref{prop:Euclid sign-bar}.}
\label{fig:sign-barrier}

\end{figure}

\vspace{2mm}

Now, we assume $\Mcc$ is a `generic' manifold, i.e.\ we work under the settings of case (a) of Proposition \ref{prop:sign-bar exist man}, with the stronger constraints on $r$ and $\eta$. We set 
\begin{equation}
\label{eq:eta' def}
\eta' = \frac{1}{C}\cdot \min\Big\{\eta,\frac{1}{\tilde{r}}\Big\}
\end{equation}
and invoke Proposition \ref{prop:asymp covar waves}(ii.) to approximate the function $h(\cdot)$ in 
\eqref{eq:h sign-bar def gen}. Recalling the Euclidean sign-barrier $w_{t}$ of \eqref{eq:wt barrier def}, with properties prescribed in Proposition \ref{prop:Euclid sign-bar} for the particular value $t=t_{0}$, 
and comparing it to the asymptotic expansion of \eqref{eq:h sign-bar def gen}, when approximating $K_{T,\eta'}\left(\exp_{x}\left(\iota_{x}\left(\widetilde{r}\cdot v_{j}\right)\right),\cdot\right)$ using \eqref{eq:covar rand wav asymp off} 
we will find below that, in the relevant regime, $h(\cdot)$ may be approximated by $\frac{2}{\varepsilon_{0}}w_{t}$,
under the usual identification \eqref{eq:iota ident Rd tangent} of $\R^{d}\cong T_{x}(\Mcc)$ given by $\iota_{x}$, and the local bijection $\exp_{x}:T_{x} \rightarrow \Mcc$, as follows.

For $j=1,2,3$, let $u_{j}$ be the unit tangent vectors $u_{j}:=\iota_{x}\left( v_{j}\right) \in T_{x}(\Mcc)$ at $x$. Then, for 
$u\in B_{r}(0)\subseteq T_{x}(\Mcc)$ one has that
\begin{equation*}
d\left(  \exp_{x}\left(\iota_{x}\left(\widetilde{r}\cdot v_{j}\right)\right),\, \exp_{x}(u)\right) = \widetilde{r} - \langle u,u_{j}\rangle + O \Big( \frac{r^{2}}{\widetilde{r}}  \Big)= \widetilde{r}\cdot \Big(1+O\Big(\frac{1}{C}\Big)\Big),
\end{equation*}
by Lemma \ref{lem:distortions plane waves}(ii.). Hence, an application of Proposition \ref{prop:asymp covar waves}(ii.) yields the estimate on $u\in B_{r}(0)\subseteq T_{x}(\Mcc)$:
\begin{equation}
\label{eq:norm kern cosine}
\begin{split}
&\frac{1}{c_{d}}\left(\widetilde{r}\cdot T\right)^{\frac{d-1}{2}} \cdot  K_{T,\eta'}\left(\exp_{x}\left(\iota_{x}\left(\widetilde{r}\cdot v_{j}\right)\right), \exp_{x}(u)  \right)
\\& \qquad = \Big(1+O\Big(\frac{1}{C}\Big)\Big) \cdot \Big( \cos\left(\widetilde{r}T+\gamma_{d}- T\langle u,u_{j}\rangle    \right)  + O \Big( \frac{1}{C}  \Big) \Big)
\\& \qquad =  \cos\left( T\langle u,u_{j}\rangle    \right)  + O \Big( \frac{1}{C}  \Big),
\end{split}
\end{equation}
by the assumptions of Proposition \ref{prop:sign-bar exist man} on $r,T$ and $\eta$ (case (a)), our definitions \eqref{eq:r tild def} of $\widetilde{r}$ and \eqref{eq:eta' def} and $\eta'$, and the $2\pi$-periodicity of the cosine. 

\vspace{2mm}

Upon summing up the estimates \eqref{eq:norm kern cosine} multiplied by $\alpha_{j}$ for
$1\le j\le 3$, we finally obtain the announced comparison of the (Riemannian) sign-balance $h$ in \eqref{eq:h sign-bar def gen} to the Euclidean sign-balance $w_{t}$ in Proposition \ref{prop:Euclid sign-bar}:
\begin{equation}
\label{eq:h pullback approx barrier}
h(\exp_{x}(z))= \frac{2}{\varepsilon_{0}}w_{t}\left(\iota^{-1}(T\cdot z)\right) + O \Big( \frac{1}{C}  \Big)
\end{equation}
on $z\in B_{r}(0)\subseteq T_{x}(\Mcc)$. Let $\widetilde{h}:B_{r}(0) \rightarrow \R$ be the pullback function $\widetilde{h}(z):= h(\exp_{x}(z))$.
Then, Proposition \ref{prop:Euclid sign-bar}, whose conclusion is applicable to a scaling by $T$ of $w_{t}$, and a forteriori to its pushforward on $T_{x}(\Mcc)$, together with \eqref{eq:h pullback approx barrier} imply that
\begin{equation*}
\widetilde{\Dc}_{\widetilde{h};\eps_{2}}(x;r)>\eps_{2},
\end{equation*}
by further decreasing $\eps_{0}$ to a sufficiently small number $\eps_{2}>0$ depending only on $\Mcc$, provided that $C$ was chosen sufficiently large. 
We now assume that, given $\delta>0$, the radius $r$ is sufficiently small so that for every $\Ac\subseteq B_{r}(0)$ one has $$|A| = |\exp_{x}(A)|\cdot (1+O(\delta)).$$
Since $$\widetilde{\Dc}_{h;\eps}(x;r) = \frac{1}{|B_{r}(x)|} \left(|h^{-1}(\eps,\infty)| - |h^{-1}(-\infty,\eps)|   \right),$$ it follows that, by further reducing $\eps_{2}$ to $\eps_{1}$, by a factor depending only on the local geometry of $\Mcc$, 
\begin{equation*}
\widetilde{\Dc}_{h;\eps_{1}}(x;r)>\eps_{1},
\end{equation*}
that is \eqref{eq:sign bar pos bias}.

\vspace{2mm}

Now we show the bound \eqref{eq:sign bar norm bnd} for the norm of $h$, still within the context of case (a) of Proposition \ref{prop:sign-bar exist man} pertaining to general smooth manifold, bearing in mind the general bound \eqref{eq:hxr bnd gen}, 
and the particular choice \eqref{eq:eta' def} of the parameter $\eta'$. In fact, \eqref{eq:sign bar norm bnd} easily follows from substituting \eqref{eq:eta' def} into the consequence \eqref{eq:1/N(T,eta) asymp} of Weyl's law, 
and finally into \eqref{eq:hxr bnd gen}, upon recalling the asymptotics $\widetilde{r}\sim C r^{2}T$ that is a straightforward conclusion from \eqref{eq:r tild def}. That concludes the treatment of case (a) of Proposition \ref{prop:sign-bar exist man}. The proof of case (b) Proposition \ref{prop:sign-bar exist man} pertaining to the round sphere follows literally along the same lines as above, except that it appeals to Proposition \ref{prop:asymp covar spher} in place of Proposition \ref{prop:asymp covar waves}, and, as a consequence, is less restrictive in terms of the constraints imposed on $r,\eta$. 
\end{proof}

\subsection{Volume-bias lower concentration II: Proof of Theorem \ref{t:clb2}}
\label{sec:lower defect-concentration proof}

Using the construction in Proposition \ref{prop:sign-bar exist man} we may conclude the proof of Theorem \ref{t:clb2}. In the proof we make use, in essential way, 
of the fact that $(f,u) \mapsto \widetilde{\Dc}_{f,u}(x;r)$ is non-decreasing in $f$ and non-increasing in $u$. 

\begin{proof}[Proof of Theorem \ref{t:clb2}]
Let $x \in \Mcc$ and $u \in \R$ be given; by the monotonicity of $u \mapsto \widetilde{\Dc}_{T,u}(x;r)$, it is sufficient to consider the case $u \ge 0$. 
By Proposition \ref{prop:asymp covar waves}(i.) and our assumption that either $\eta(T) \ge C$ (a consequence of $\eta>C(rT)^{d-1}$ in case (a) of Theorem \ref{t:clb2}) 
or $M = \mathbb{S}^d$ (case (b) of Theorem \ref{t:clb2}), one has that 
\begin{equation}
\label{eq:var(fTeta) bounded}
\var\left(f_{T,\eta}(x)\right)  \le C_1
\end{equation}
with some constant $C_1 > 0$ depending only on $\Mcc$.
Let $\eps = \eps_0 > 0$ be the absolute constant in Proposition \ref{prop:sign-bar exist man}, and let $h \in \Hc$ be such that $\widetilde{\Dc}_{h;\eps}(x;r) > \eps$ and $\|h\|_{\Hc}^2 < C \cdot \max\{1,r\eta\} \cdot (rT)^{2(d-1)}$, as prescribed by Proposition \ref{prop:sign-bar exist man}. Denote the normalised element $\widehat{h}:= \frac{h}{\|h\|_{\Hc}}$,  so that the random wave $f_T = f_{T,\eta}$ has the representation as 
\begin{equation}
\label{eq:f=Zh+tildf}
f_T(x) \stackrel{d}{=}  Z\cdot \widehat{h}(x) + \tilde{f}(x)   ,
\end{equation}
where $Z$ is a standard Gaussian, and $\tilde{f}$ is a Gaussian random field on $\Mcc$, independent of $Z$. 
In particular, one has that 
\begin{equation}
\label{eq:var(ftild) bounded}
\var \big( \tilde{f}(x) \big) \le \var\left( f_{T,\eta}(x)\right)\le C_{1}
\end{equation}
for every $x\in\Mcc$, by \eqref{eq:var(fTeta) bounded}. 

\vspace{2mm}

Now choose a constant $C_2 > 0$ sufficiently large so that 
\begin{equation}
\label{eq:Phi(-eps c2/sqrt(c1))<eps/4}
\Phi(- \eps C_2/\sqrt{C_1}) \le \eps / 8,
\end{equation}
where $\Phi(\cdot)$ is the Gaussian cdf. We set $v := \max\{u,C_2\}$, and consider the event $\Ecc:=A \cap B$, where
\[ A := \left\{ Z \ge 2 v \cdot \|h\|_{\Hc}  \right\} \quad \text{and} \quad B :=  \bigg\{   \frac{ | \widetilde{f}^{-1}(-\infty,- \eps v)  \cap B_{r}(y) | }{|B_r(y)|  }   <  \frac{\eps}{4} \bigg\} ,\]
with $Z$ and $\widetilde{f}$ as in \eqref{eq:f=Zh+tildf}. 
The {\em independent} events $A$ and $B$ imply respectively that  
\[ \widetilde{\Dc}_{ Z \cdot \widehat{h} ; \, 2 \eps v }(x;r)  > \eps \quad \text{and}  \quad \widetilde{\Dc}_{ T ;\,   \eps v}(x;r)  -  \widetilde{\Dc}_{ Z \cdot\widehat{h};\,  2 \eps v }(x;r)  \ge - \frac{\eps}{2} ,\]
so that, on $\Ecc$, we have 
\begin{equation}
\label{eq:D>=eps/2 on E}
\widetilde{\Dc}_{T, \eps u}(x;r) \ge \widetilde{\Dc}_{T, \eps v}(x;r) >  \frac{\eps}{2}  ,
\end{equation}
since $u\le v$.

To estimate 
\begin{equation}
\label{eq:pb(E)=pb(A)*pb(B)}
\Pb(\Ecc)=\Pb(A \cap B) = \Pb(A) \cdot \Pb(B)
\end{equation}
we first observe that
\begin{equation}
\label{eq:pb(A) lower bnd}
\Pb(A) \ge e^{-c_{0} v^2 \|h\|_H^2}  \ge e^{- c_{0} v^2 C \max\{1,r^2 T \eta\} (rT)^{2(d-1)} },
\end{equation}
with an absolute constant $c_{0} > 0$.
Moreover, by Markov's inequality, \eqref{eq:var(ftild) bounded}, and the choice of $C_2 > 0$ satisfying \eqref{eq:Phi(-eps c2/sqrt(c1))<eps/4},
\begin{equation}
\label{eq:pb(B) lower bnd}
\begin{split}
\Pb(B) &\ge 1 -  \frac{ \E\big[  |   \tilde{f}^{-1}( -\infty,  - \eps v)  \cap B_y(r) | \big]  }{  (\eps/4)\cdot |B_y(r)|   } = 1 -  \frac{ \Phi(-\eps v/\sqrt{C_1})}{\eps/4 }  \\&\ge 1 -  \frac{ \Phi(-\eps C_2/\sqrt{C_1})}{\eps/2 } \ge 1/2.
\end{split}
\end{equation}
The statement of Theorem \ref{t:clb2} now follows upon multiplying the inequalities \eqref{eq:pb(A) lower bnd} and \eqref{eq:pb(B) lower bnd}, substituting the outcome into \eqref{eq:pb(E)=pb(A)*pb(B)}, on
bearing in mind that $\Ecc$ implies \eqref{eq:D>=eps/2 on E}, and adjusting the various constants accordingly.
\end{proof}

\subsection{Volume-bias lower concentration I: Proof of Theorem \ref{t:clb}}
The proof of Theorem \ref{t:clb} is similar to that of Theorem \ref{t:clb2}, except that a significantly simpler analogue of the sign-barrier (termed {\em level-barrier}) is constructed, only requiring that the volume-bias at some positive level $\eps > 0$ is larger than $-\eps$ (rather than $>\eps$ as in Proposition \ref{prop:sign-bar exist man}). Indeed, instead of superimposing three monochromatic waves as in the previous construction, here we settle for a perfectly sign-balanced realisation of a single monochromatic wave.

\begin{proposition}[Level-barrier]
\label{prop:lev-bar exist man}
There exist numbers $C>0$, $T_{0}>0$ sufficiently large, only depending on $\Mcc$, with the following property. Let $T>T_{0}$, 
$\eta \in [1,T]$, $r > 0$, $u > 0$, satisfy either (a) $\frac{C}{T}<r<\frac{1}{C} \cdot \frac{1}{T^{\frac{d-1}{d+1}}}$ and $\eta> C\cdot (rT)^{\frac{d-1}{2}}$, or (b) $\Mcc=\Sc^{d}$, and $\frac{C}{T}<r<\frac{1}{C}$. 
Then, for every $x\in\Mcc$, and $\eps>0$ there exists a number $w=w(\eps)>0$ sufficiently small, and an element $h=h_{x,r;\eps}\in\Hc$ of the RKHS of norm
\begin{equation}
\label{eq:lev bar norm bnd}
\|h\|^2_{\Hc} \le C \max\{ 1,r\cdot\eta  \}\cdot (rT)^{d-1},
\end{equation}
satisfying
\begin{equation}
\label{eq:lev bar pos bias}
\widetilde{\Dc}_{h; w}(x;r) > - \eps.
\end{equation}
\end{proposition}

The proof of Proposition \ref{prop:lev-bar exist man} follows along similar (but simpler) lines to that of Proposition \ref{prop:sign-bar exist man} above, hence we will only sketch the proof to avoid repetition.
\begin{proof}

We will assume throughout the proof that $\Mcc$ is a general manifold, i.e. work under scenario (a) of Proposition \ref{prop:lev-bar exist man}, the other case being simpler.
Let $\eta'$ be the parameter $\eta' = \min\{\eta,1/r\}$, and choose the element $h\in\Hc$, defined as the function
\begin{equation*}
h(\cdot) = h_{x,r}(\cdot) = c_d^{-1} (rT)^{\frac{d-1}{2}} K_{T,\eta'}(x, \cdot).
\end{equation*}
We claim that $h$ satisfies the properties asserted in Proposition \ref{prop:lev-bar exist man}. To this end we invoke Proposition \ref{prop:asymp covar waves}(ii.), implying that, on $y\in B_{r}(x)$, one has
\begin{equation}
\label{eq:h barrier = cosine 1 wave}
h(y) = \cos( d(x,y)\cdot T + \gamma_d) + O\Big(\frac{1}{C}\Big),
\end{equation}
for $C>0$ sufficiently large number.

\vspace{2mm}

We observe that for $u\in (-1,1)$, the inverse image $p^{-1}((u,\infty))$ of the excursion set at level $u$ of the function 
$a\mapsto p(a)=\cos\left(\|a\|+\gamma_{d}\right)$ on $\R^{d}$ consumes proportion $\varphi(u)=\frac{\arccos(u)}{\pi}$ of the entire space $\R^{d}$. That is, as $R\rightarrow  \infty$,
\begin{equation*}
\Dc_{p,u}(0;R) \rightarrow \phi(u):=2\varphi(u)-1,
\end{equation*}
{\em uniformly} w.r.t. $u$. In particular, $p$ is balanced at level $0$, i.e. $\phi(0)=0$.  
Hence, for every $0<\eps<1$, $\phi^{-1}(-\eps)>0$, so that one may find a number $w=w(\eps) \in (0,\phi^{-1}(-\eps))$, and, evidently, for $R>R_{0}$ sufficiently large (absolute),  
\begin{equation*}
\Dc_{p,w}(0;R) > -\eps.
\end{equation*}

\vspace{2mm}

Now let $\widetilde{h}:T_{x}(\Mcc)\rightarrow\R$ be the pullback $\widetilde{h}(z) = h(\exp_{x}(z))$ of $h$. Then the above, together with the asymptotics \eqref{eq:h barrier = cosine 1 wave} (on taking into account the scaling by $T$) 
shows that one has
\begin{equation*}
\Dc_{\widetilde{h},w}(0;r) > -\eps,
\end{equation*}
with a slightly decreased $w>0$, as a result of the arising error term in \eqref{eq:h barrier = cosine 1 wave}. One may then pushforward this result to $h$ by further decreasing $w$, i.e.\ assert \eqref{eq:lev bar pos bias}, in a manner similar to the way done at the end of the proof of Proposition \ref{prop:sign-bar exist man} above, on reusing the {\em small} distortions in volume measure around $x$ in $\Mcc$, 
as a by-product of the curvature. The norm claim \eqref{eq:lev bar norm bnd} on $h$ also follows along the same lines as in Proposition \ref{prop:sign-bar exist man}, on reusing \eqref{eq:kernel centred norm bnd},
and Weyl's law \eqref{eq:1/N(T,eta) asymp} (rather, its consequence).
\end{proof}

We are now in a position to give a proof for Theorem \ref{t:clb}. Again it follows along similar lines to the proof of Theorem \ref{t:clb2} above, so we will only sketch the proof.

\begin{proof}[Proof of Theorem \ref{t:clb}]
Let $x \in \Mcc$ and $u \ge 0$ be given. For a given $\eps>0$, we aim for a bound of the form 
\[ \prob \Big(  \widetilde{\Dc}_{T, u}(x;r) > -\eps \Big) \le \alpha,   \]
with $\alpha=\alpha(u)$.  We now apply Proposition \ref{prop:lev-bar exist man} with the given $\eps$ to obtain a number $w=w(\eps)>0$ and an element $h \in \Hc$ so that 
$\widetilde{\Dc}_{h;w}(x;r) > -\eps$ and $\|h\|_{\Hc}^2 < C \max\{1,r\eta\} \cdot (rT)^{d-1}$. 
Then, in accord with \eqref{eq:f=Zh+tildf}, on putting $\widehat{h}:= \frac{h}{\|h\|_{\Hc}}$, the random field $f_T = f_{T,\eta}$ may be represented as 
\[ f_T(x) \stackrel{d}{=}  Z\cdot \widehat{h}(x) + \tilde{f}(x),   \]
where $Z$ is a standard Gaussian, and $\tilde{f}$ is a Gaussian random field on $\Mcc$, independent of $Z$, 
such that $\textrm{Var}[ \tilde{f}(x) ] \le \textrm{Var}[f_{T,\eta}(x)]   \le C_1$ for a constant $C_1 > 0$ depending only on $\Mcc$. 

\vspace{2mm}

Choose a constant $C_2 > 0$ large enough so that $\Phi(- \eps C_2/\sqrt{C_1}) \le \eps / 2$, where $\Phi(\cdot)$ is the Gaussian cdf, and let $v = \max\{u,C_2\}$. Now let $\Ecc$ be the event $\Ecc:=A \cap B$, where
\[ A = \big\{ Z \ge 2 v \|h\|_{\Hc}  \big\} \quad \text{and} \quad B = \bigg\{   \frac{ | \tilde{f}^{-1}(-\infty,- w v)  \cap B_y(r) | }{|B_y(r)|  }   <  \eps  \bigg\} . \]
The events $A$ and $B$ respectively imply that 
\[ \widetilde{\Dc}_{ Z\cdot \widehat{h} ;  \, 2 w\cdot  v }(x;r)  > - \eps \quad \text{and}  \quad \widetilde{\Dc}_{T ;   w v}(x;r)  -  \widetilde{\Dc}_{ Z\cdot \widehat{h} ; \, 2 w v }(x;r)  \ge - 2\eps   ,\]
so that, on $\Ecc$, we have 
\[ \widetilde{\Dc}_{T, \eps u}(x;r) \ge \widetilde{\Dc}_{T, \eps v}(x;r) > -3 \eps  . \]
Moreover since
\[ \Pb(A) \ge e^{-c_{0} v^2 \|h\|_{\Hc}^2}  \ge e^{- c_{0} v^2 C \max\{1,r\eta\} (rT)^{d-1} } ,\]
with some absolute $c_{0}>0$, and by Markov's inequality and the above choice of $C_2 > 0$,
\[ \Pb(B) \ge 1 -  \frac{ \E\big[  |   \tilde{f}^{-1}( -\infty,  - \eps v)  \cap B_y(r) \}| \big]  }{  \eps |B_y(r)|   }  \ge 1 -  \frac{ \Phi(-\eps C_2/\sqrt{c_1})}{\eps } \ge 1/2 \]
by adjusting constants we obtain the desired conclusion.
 \end{proof}

\medskip

\section{Sign-balance of random waves: Proof of Theorem \ref{thm:sign bal rand wav} and Corollary \ref{cor:as}}
\label{sec:main res proofs}

\subsection{Upper bound for sign-balance: Proof of Theorem \ref{thm:sign bal rand wav}(i.) and Corollary \ref{cor:as}}
\label{sec:proof upper bnd sign-balance}

For a fixed $u$ and $(x,r)\in \Mcc\times \R_{>0}$, Theorem \ref{t:cub} implies that $|\Dc_{T,u}(x;r)|<\eps$ occurs with high probability provided that $r$ is sufficiently large. To assert Theorem \ref{thm:sign bal rand wav}(i.) we will need to prove the same for the {\em supremum} of $|\Dc_{T,u}(x;r)|$ w.r.t.\ both $x$ and $r \ge r_0$. For this purpose we shall rely on the following \textit{stability} property of the volume bias w.r.t.\ perturbations of $x$ and $r$, which will allow us to apply Theorem \ref{t:cub} to a suitably chosen `dense' net. 

\begin{lemma}
\label{l:stability}
For every $\eps > 0$ there exists a number $\delta  = \delta(\Mcc,\eps) > 0$ with the following property. For all smooth functions $f:\Mcc\rightarrow\R$, $u\in\R$, $x\in\Mcc$, $r>0$, $y\in B_{\delta\cdot r}(x)$ and 
$r' \in [r, (1+\delta)r]$ so that $r'<\inj(\Mcc)$, one has
\begin{equation*}
\left|\Dc_{f;u}(y;r')-\Dc_{f;u}(x;r)\right|<\eps,
\end{equation*}
where $\Dc_{\cdot,\cdot}(\cdot,\cdot)$ is as in \eqref{eq:defect def lev}.
\end{lemma}

\begin{proof}
The definition \eqref{eq:defect def lev} of volume-bias clearly implies that
\begin{equation}
\label{eq:bias-bias perturb}
\left|\Dc_{f;u}(y;r')-\Dc_{f;u}(x;r)\right|\le  4\cdot \frac{ | B_{r}(x) \triangle B_{r'}(y) | }{  \min\{ |B_{r}(x)|, |B_{r'}(y)| \} } .
\end{equation}
For $y \in B_{\delta\cdot r}(x)$ and $r' \in [r, (1+\delta)r]$, we can further bound \eqref{eq:bias-bias perturb} as
\begin{equation}
\label{eq:bias-bias<=geom}
\begin{split}
&\left|\Dc_{f;u}(y;r')-\Dc_{f;u}(x;r)\right|\le  \max\limits_{x \in \Mcc} \frac{ \max\limits_{y \in B_{\delta\cdot r}(x) }  | B_{r}(x) \triangle B_{r}(y) | }{  \min\limits_{y \in B_{\delta\cdot r}(x)}   |B_{r}(y)|   } + 
\max\limits_{x \in \Mcc} \frac{  | B_{r}(x) \triangle B_{(1+\delta)r}(x) |}{  |B_{r}(x)|   } 
\\&\le  \max\limits_{x \in \Mcc} \frac{ \max\limits_{y \in B_{\delta\cdot r}(x) }  | B_{r}(x) \triangle B_{r}(y) | }{  \min\limits_{y \in B_{\delta\cdot r}(x)}   |B_{r}(y)|   } + 
\max\limits_{x \in \Mcc} \frac{  | B_{(1+\delta)r}(x)\setminus B_{r}(x) |}{  |B_{r}(x)|   } .
\end{split}
\end{equation}

We now claim that the two terms on the r.h.s.\ of \eqref{eq:bias-bias<=geom} can be made arbitrary small by choosing $\delta>0$ sufficiently small (depending on $\Mcc$). Indeed, while for the round sphere, it is easy to derive the said conclusion via an explicit computation, in general it follows from a standard argument appealing to the Bishop-Gromov comparison inequality, the compactness of $\Mcc$, and Cantor's theorem.
\end{proof}

We are now in a position to conclude the proof of Theorem \ref{thm:sign bal rand wav}(i.).

\begin{proof}[Proof of Theorem \ref{thm:sign bal rand wav}(i.)]

Let $\bar{r}_T$ be as in \eqref{eq:rl cric rad spher harm}. We aim to prove that, for $\mu > 0$, one has for every $u \in \R$ and $\varepsilon > 0$, as $T \to \infty$
\begin{equation}
\label{e:claim}
\Pb \Big( \sup_{r \ge \mu \bar{r}_T}  \mathcal{B}_{T,u}(r) >  \eps \Big) =O\left( T^{d-c  \mu^{d-1} }\cdot \log{T} \right),
\end{equation}
with some constant $c>0$ depending only on $\Mcc$.
This is stronger than the mere vanishing \eqref{eq:balance -> 0 prob} of the probability on the l.h.s. of \eqref{e:claim}, asserted by Theorem \ref{thm:sign bal rand wav}(i.)
for $\mu>0$ sufficiently large, for which it is sufficient to take the constant $\mu$ sufficiently large so that $c  \mu^{d-1} > d$. We prove \eqref{e:claim} in two steps: first we apply Theorem \ref{t:cub} to obtain an upper bound on $ \Pb( \Dc_{T,u}(x;r) > \varepsilon) $ for a fixed $x \in \Mcc$ and $r \ge  \bar{r}_T$ for some $s \ge 1$, and then we use the union bound and the stability of the volume-bias to complete the proof. In what follows, $c_i > 0$ designates constants that depend only on $\Mcc$ and $\eps$. 

\vspace{2mm}

Now, let $\eps>0$ be given, $x\in\Mcc$ and $r\ge \bar{r}_T$, $B_{x}(r)$ be the corresponding geodesic ball, and let $s:=\frac{r}{\bar{r}_T}\ge 1$. Recalling that $\bar{r}_T$ in \eqref{eq:rT crit rad rand wav} is defined so as to satisfy the equality  
\begin{equation}
\label{e:barr}
(\bar{r}_T T)^{d-1} \max\{1,  \bar{r}_T \cdot\eta \} = \log T 
\end{equation}
(see \eqref{eq:bar(r) implicit equation}), in either case $r>\frac{1}{\eta}$ or $r\le \frac{1}{\eta}$ an application of Theorem \ref{t:cub} yields
\begin{equation}
\label{e:bound}
\Pb( |\Dc_{T,u}(x;r)| > \varepsilon)  < e^{ - c( s \bar{r}_T T)^{d-1} \max\{1 , s \cdot \bar{r}_T\cdot  \eta \} } \le e^{-c s^{d-1} ( \bar{r}_T\cdot  T)^{d-1} \max\{1,  \bar{r}_T\cdot \eta \}}  =   T^{-c \cdot s^{d-1} }  .
\end{equation}

\vspace{2mm}

Let us conclude the proof. The assertion of Lemma \ref{l:stability} applied to $f_{T,\eta}(\cdot)$ reads $$|\Dc_{T,u}(x;r) - \Dc_{T,u}(y;r') |<\eps$$ with some $\delta=\delta(\Mcc,\eps) > 0$, 
uniformly for $x\in\Mcc$, $r>0$, $y\in B_{x}(\delta\cdot r)$ and 
$r' \in [r, (1+\delta)r]$. Define $r_i = \mu \cdot (1 + \delta)^i \cdot \bar{r}_T$ for $0\le i\le I$, where $I=I_{T}$ is the maximum integer such that $r_I < \text{Inj}(M)$; since $\bar{r}_T \ge 1/T$, we have $I \le c_1 \log T$. 
Then for every $0 \le i \le I$, let $\{x^i_j\}_{j \le N_i}$ be a `$\delta r_i$-net', i.e. a collection of points on $\Mcc$ satisfying:
\begin{itemize}
\item $\max\limits_{x \in \Mcc} \min\limits_{j \le N_i} d(x, x^i_j) \le \delta r_i$, i.e.\ the radius $\delta r_{i}$ geodesic balls centred at $\{x^i_j\}_{j \le N_i}$ cover $\Mcc$;
\item $N_i \le c_2 r_i^{-d} \le  c_2 r_0^{-d}  \le c_3 T^d$.
\end{itemize}
We observe that for $0\le i\le I$, we have $$s_{i}:=\frac{r}{\bar{r}_T} = \mu\cdot (1+\delta)^{i} > \mu.$$ Hence, 
\eqref{e:bound} yields, for every $$(i,j)\in (i,j)\in \Jc:=\left\{\Z_{\ge 0}^{2}:\: 0 \le i \le I,\, 1 \le j \le N_i\right\},$$ the bound
 \[ \Pb( |\Dc_{T,u}(x^i_j ; r_i )| > \varepsilon)  < T^{-c  \mu^{d-1} }.  \]

Further, since $\#\Jc \le (I+1)\cdot (N_{0}+1)\le  c_4 T^d (\log T)$, the union bound implies
 \[ \Pb \Big(  \bigcup\limits_{ (i,j) \in J }  \left\{\Dc_{T,u}(x^i_j;r_i) > \varepsilon\right\} \Big) \le  c_{4} T^{d-c  \mu^{d-1} }\log{T} . \] 
Since, for every $i\le I$, the $\{x_{j}^{i}\}_{j\le N_{i}}$ is a $\delta r_{i}$-net on $\Mcc$, that means that for every $x \in \Mcc$ and $r \ge \mu \bar{r}_T$ there exists some $(i,j) \in \Jc$, such that 
$r \in [r_i, (1+\delta)r_i]$ and $y \in B_{x^i_j, \delta r_i}$. Therefore, in this context, Lemma \ref{l:stability} implies that
\begin{equation*}
\bigcup\limits_{  \substack{x \in \Mcc\\ r \ge \mu \cdot \bar{r}_T} } \left\{ \Dc_{T,u}(y;r) > 2 \varepsilon\right\} \subseteq \bigcup\limits_{ (i,j) \in J }  \left\{\Dc_{T,u}(x^i_j;r_i) > \varepsilon\right\},
\end{equation*}
and therefore, 
\[ \Pb \Big(  \bigcup\limits_{  \substack{x \in \Mcc\\ r \ge \mu \cdot \bar{r}_T}} \left\{ \Dc_{T,u}(y;r) > 2 \varepsilon\right\} \Big)\le  c_{4} T^{d-c  \mu^{d-1} }\log{T}. \] 
Replacing $\eps$ by $\frac{\eps}{2}$ and adjusting the constants accordingly, this proves \eqref{e:claim}.\end{proof}

\begin{proof}[Proof of Corollary \ref{cor:as}]
The proof of Corollary \ref{cor:as} is contained in essence within the proof of Theorem \ref{thm:sign bal rand wav}(i.) above. Fix $\mu > 0$ such that $c \mu^{d-1} > 3d$, and input it into \eqref{e:claim}. The upshot is that 
\begin{equation}
\label{eq:bound eps single T}
\Pb \Big( \sup\limits_{r \ge \mu \bar{r}_T}  \mathcal{B}_{T,u}(r)  > \varepsilon \Big) =  O\left(T^{-2d-\xi} \log{T}\right) ,
\end{equation}
with $\xi:=c \mu^{d-1}-3d>0$. 

On the other hand, the number of energy levels of $\Mcc$ contained in an energy window $[S,S+1]$ is $O(S^{2d-1} )$. Therefore, on taking into account the fact that the law of $f_{T,\eta}$ depends on the energy levels lying in $[T-\eta,T]$,
the number of genuinely different $f_{T,\eta}(\cdot)$ (with different law) corresponding to $T\in [S,S+1]$ is $O(S^{2d-1} )$, on taking into account that we do not have control on how $\eta$ varies as a function of $T$. Hence, on using \eqref{eq:bound eps single T} with the union bound, 
\begin{equation*}
\Pb \Big( \bigcup\limits_{T \in [S,S+1] } \sup\limits_{r \ge \mu \bar{r}_T}  \mathcal{B}_{T,u}(r)  > \varepsilon \Big) =O\left( S^{-1-\xi}\cdot\log{S}\right).
\end{equation*}
The almost sure convergence \eqref{eq:vol imbal->0 a.s.} now follows from the Borel-Cantelli lemma.
\end{proof}

\subsection{Lower bound for sign-balance: Proof of Theorem \ref{thm:sign bal rand wav}(ii.)}
First we assume that $u \neq 0$, whence, by symmetry we may assume that $u > 0$. Recall that $\bar{r}_T$ is given by \eqref{eq:rT crit rad rand wav} and $\tau(u)>0$ is given by \eqref{eq:tau mean bias def}, and observe that, in this case, 
we may restate Theorem \ref{thm:sign bal rand wav}(ii.) as
\begin{equation}
\label{eq:prob(exists bad x)->1}
\Pb\big(  \exists x \in \Mcc :  \Dc_{T,u}(x; \mu \bar{r}_T)  >  \varepsilon \big)  \to 1,
\end{equation}
with some $\eps > 0$ and $\mu \in (0,1)$ sufficiently small. 
The proof of \eqref{eq:prob(exists bad x)->1} is in two steps. First, we apply Theorem \ref{t:cub} to obtain a lower bound on $ \Pb( \widetilde{\Dc}_{T,u+1}(x; \mu \bar{r}_T) > - \eps') $ for some fixed $\eps' \in (0, \tau(u))$ and $x \in \Mcc$. Then we use a `sprinkled' decoupling technique to show that, for a suitably chosen collection of points $x_i \in \Mcc$, with high probability at least one of them satisfies 
$\widetilde{\Dc}_{T,u}(x_i; \mu \bar{r}_T) > - \eps'$. Since, by definition, 
\[   \Dc_{T,u}(x; r)   =   \widetilde{\Dc}_{T,u}(x; r)  + \tau(u) , \]
this will conclude the proof of \eqref{eq:prob(exists bad x)->1} (and thus of Theorem \ref{thm:sign bal rand wav}(ii.) for case $u\ne 0$).

\vspace{2mm}

For the first step, note that case (a') of Theorem \ref{thm:sign bal rand wav}(ii.) makes the extra assumption that $\eta(T) > T^{\delta_0}$ with some $\delta_{0}>0$, so that in this case 
\[ \min\left\{\eta,\frac{1}{\mu \bar{r}_T} \right\}  > C (\mu \bar{r}_T)^{(d-1)/2} \]
for fixed $\mu > 0$ and $T$ sufficiently large. Therefore, in either case (a') or (b') of Theorem \ref{thm:sign bal rand wav}(ii.), the radius $r=\bar{r}_{T}$ satisfies the assumptions of Theorem \ref{t:clb}, 
and application of which, on recalling \eqref{e:barr}, yields
\begin{equation}
\label{e:lbfix}
 \Pb\left(  \widetilde{\Dc}_{T,u+1}(x; \mu \bar{r}_T) > - \eps'\right)  > e^{-C  (\mu \bar{r}_T T)^{d-1} \max\{1,  \mu \bar{r}_T \eta \} }    \ge     e^{-C \mu^{d} ( \bar{r}_T T)^{d-1} \max\{1,  \bar{r}_T \eta \}}  =     T^{-C  \mu^d } ,
 \end{equation}
for $c > 0$ depending only on $\Mcc$, $u$, and $\eps'$.

\vspace{2mm}

We now turn to the decoupling step. We shall make use of the following \textrm{sprinkled decoupling inequality} for Gaussian random fields \cite{m23}. Let $f$ be a continuous Gaussian random field with covariance kernel $K(\cdot,\cdot)$, 
decreasing events $\Fc,\Gc$, and $v > 0$. Then one has the inequality 
\begin{equation}
\label{eq:sprinkled decoupling}
\Pb( f \in \Fc , f \in \Gc ) \le \Pb( f - v  \in \Fc)\cdot  \Pb(f - v \in \Gc)  + \frac{ c_0 k}{  v^2} 
\end{equation}
where $c_0 > 0$ is an absolute constant, and $$k := \sup\limits_{\substack{x \in \textrm{supp}(\Fc)\\ y \in \textrm{supp}(\Gc) }} |K(x,y)| ,$$ where $ \textrm{supp}(E)$ is the support of $E$, i.e.\ the smallest closed set $D$ such that $E$ is determined by $f|_D$. 

We apply the sprinkled decoupling inequality \eqref{eq:sprinkled decoupling} to the events $\{   \widetilde{\Dc}_{T,u}(x_i; r ) \le -\eps  \}$, where $r=\mu \bar{r}_T$, and $\left\{x_i\right\}_{1 \le i \le I}$ is some collection of well-separated points, as follows. For $1\le i_{0} < I$, define
$$\Fc(i_{0};u):= \bigcap\limits_{1 \le i \le i_0}  \left\{    \widetilde{\Dc}_{T,u}(x_i; r ) \le -\eps  \right\},$$  
$$\Gc(i_{0};u):= \bigcap\limits_{i_{0}+1\le i \le I}  \left\{    \widetilde{\Dc}_{T,u}(x_i; r ) \le -\eps  \right\}.$$  Then \eqref{eq:sprinkled decoupling} reads
\begin{align}
\label{e:sdi}
& \Pb \Big( \bigcap\limits_{1 \le i \le I}  \left\{    \widetilde{\Dc}_{T,u}(x_i; r ) \le -\eps  \right\}  \Big) = \Pb\left(\Fc(i_{0};u)\cap \Gc(i_{0};u)\right)\\
\nonumber & \qquad \qquad \qquad \le \Pb \left(  f_T  - v \in \Fc(i_{0};u)   \right) \cdot 
\Pb\left(  f_T - v \in  \Gc(i_{0};u)\right)+  E  \\
\nonumber &\qquad \qquad \qquad =\Pb \left(  f_T   \in \Fc(i_{0};u+v)    \right)\cdot  
\Pb \left(  f_T   \in \Gc(i_{0};u+v)\right)+  E  ,
\end{align}
where  
\begin{equation}
\label{eq:E=c sup/v^2}
E =  c_0\cdot \frac{  \sup\limits_{x \in D_{\Fc},\, y\in D_{\Gc} } | K_{T}(x,y) | } { v^2} ,
\end{equation}
with $D_{\Fc} = \textrm{supp}(\Fc) = \bigcup\limits_{1 \le i \le i_0} B_{x_i}(r)$, and $D_{\Gc} = \textrm{supp}(B) = \bigcup\limits_{i_0 + 1 \le i \le I} B_{x_i}(r)$.

\vspace{2mm}

Now we define an appropriate collection of points $\{x_i\}_{i\le I}$. Observe that, since the extra assumption $\eta(T)>T^{\delta_{0}}$ with some $\delta_{0}$ was made in case (a') of Theorem \ref{thm:sign bal rand wav}(ii), 
the hypotheses of Corollary \ref{cor:decay correlations} are satisfied. Let $\delta_{1}>0$ be as in Corollary \ref{cor:decay correlations}. 
Choose a collection of points $(x_i)_{i \le I}$, with 
\begin{equation}
\label{eq:I well-separated points}
I = I(T)= \left\lfloor T^{\delta_{1}/4} \right\rfloor,
\end{equation}
so that, in addition,  
\begin{equation}
\label{eq:dist > T^1-del1}
\text{for every } i \neq j \text{, one has       }\;\;\;\; d(B_{\mu \bar{r}_T}({x_i}) ,B_{\mu \bar{r}_T}({x_j}) ) > T^{-(1-\delta_{1})}.
\end{equation}
(For the round sphere, we constraint these balls to lie in a single hemisphere to avoid the possible degeneracies that might occur for the pure spherical harmonics ($\eta=1$)).

Let
\begin{equation*}
\bar k := \sup\limits_{\substack{x\in B_{r}(x_{i}) ,\, y\in B_{r}(x_{j})\\ i\ne j }} | K_{T}(x,y) |.
\end{equation*}
Then, on one hand, Corollary \ref{cor:decay correlations} implies that 
\begin{equation}
\label{eq:bar(k) bound decay}
\bar {k} \le \frac{C_{1}}{T^{\delta_{1}}}.
\end{equation}
On the other hand, we may apply the inequality \eqref{e:sdi} to the collection $\{x_{i}\}$, with $v=\frac{1}{I}$, first with $i_{0}=1$, and then, by induction, to the probability 
$\Pb \left(  f_T   \in \Gc(i_{0};u+v)\right)$ (where each of the $I$ steps incurs an error $E$), to eventually obtain the inequality 
\begin{equation}
\label{eq:prob intersect <= prod prob}
 \Pb \Big( \bigcap\limits_{i \le I} \left\{ \widetilde{\Dc}_{T,u}(x_i; \mu \bar{r}_T) \le - \eps \right\} \Big)   \le \prod_{i \le I} \Pb\Big(  \widetilde{\Dc}_{T,u+1}(x_i; \mu \bar{r}_T) \le - \eps  \Big) +  I\cdot E ,
\end{equation}
with $E=c_{0}\frac{\bar{k}}{v^{2}}$, as in \eqref{eq:E=c sup/v^2}. Mind that our inductive treatment is suboptimal in terms of the level $u+1$, but will do the job for us.

Therefore, we may bound 
$$I\cdot E \le c_{0}I^{3}\bar{k} \le c_{0}T^{3\delta_{1}/4}\cdot T^{-\delta_{1}} = C_{2}\cdot T^{-\delta_{1}/4}$$ with $C_{2}=c_{0}\cdot C_{1}$,
and \eqref{eq:prob intersect <= prod prob} reads 
\begin{equation*}
\Pb \Big( \bigcap\limits_{i \le I} \left\{ \widetilde{\Dc}_{T,u}(x_i; \mu \bar{r}_T) \le - \eps \right\} \Big)  \le \prod_{i \le I} \Pb\left(  \widetilde{\Dc}_{T,u+1}(x_i; \mu \bar{r}_T)\le -\eps\right) +  C_{2}\cdot T^{-\delta_{1}/4},
\end{equation*}
and further
\begin{equation}
\label{eq:prob <= (1-T)^I+o(1)}
\Pb \Big( \bigcap\limits_{i \le I} \left\{ \widetilde{\Dc}_{T,u}(x_i; \mu \bar{r}_T) \le - \eps \right\} \Big) \le \left(1- T^{-C  \mu^d }  \right)^{I} + C_{2}\cdot T^{-\delta_{1}/4},
\end{equation}
on invoking the complement inequality to \eqref{e:lbfix}. We bound the r.h.s. of \eqref{eq:prob <= (1-T)^I+o(1)} as
\begin{align*}
\left( 1-T^{-C  \mu^d }  \right)^{I} + C_{2}\cdot T^{-\delta_{1}/4}& \le e^{-T^{\delta_{1}/4}\cdot T^{-C  \mu^d }} + C_{2}\cdot T^{-\delta_{1}/4}= e^{-T^{\delta_{1}/4}\cdot T^{-C  \mu^d }} + C_{2}\cdot T^{-\delta_{1}/4} \\&=e^{-T^{\delta_{1}/4-C  \mu^d}} + C_{2}\cdot T^{-\delta_{1}/4},
\end{align*}
since $\log(1-x) \le -x$ for $|x|<1$, and on recalling the choice \eqref{eq:I well-separated points} for $I=I(T)$. 
Then, one has 
\begin{equation*}
\Pb \Big( \bigcap\limits_{i \le I} \left\{ \widetilde{\Dc}_{T,u}(x_i; \mu \bar{r}_T) \le - \eps \right\} \Big) \rightarrow 0,
\end{equation*}
provided that $C\mu^{d}< \frac{\delta_{1}}{4}$. Passing to the complement, this implies \eqref{eq:prob(exists bad x)->1} for $u\ne 0$.

\smallskip
We now turn to proving case $u = 0$, appealing to a similar argument to the above, but the analysis is different, and the differences are highlighted. Recall $\ubar{r}_T$ defined in \eqref{eq:rT crit rad rand wav}, satisfying (cf.\ \eqref{e:barr})
 \begin{equation}
 \label{e:barr2}
  (\ubar{r}_T T)^{2(d-1) } \cdot \max\{1,  \ubar{r}_T^2 T \cdot \eta \} = \log T .
  \end{equation}
An application of Theorem \ref{t:clb2} (instead of Theorem \ref{t:clb}) yields (cf.\ \eqref{e:lbfix})
 \begin{equation}
 \begin{split}
 \label{e:lbfix2}
 \Pb\left(  \widetilde{\Dc}_{T,1}(x; \mu\cdot \ubar{r}_T) >  \eps\right)  &> e^{-C  (\mu\cdot \ubar{r}_T\cdot T)^{2(d-1)} \max\{1,  \mu^2 \cdot \ubar{r}_T^2 T \cdot \eta \} }    \\&\ge     e^{-C \cdot \mu^{2d} \cdot ( \ubar{r}_T T)^{2(d-1)}
  \max\{1,  \ubar{r}_T^2 T \cdot \eta \}}  =     
 T^{-C  \mu^{2d} }.
 \end{split}
\end{equation}
  
Now let $(x_i)_{1 \le i \le I}$ be points satisfying the same properties \eqref{eq:I well-separated points} and \eqref{eq:dist > T^1-del1}, as in case $u \neq 0$. 
Using the same sprinkled decoupling procedure as in the course of the above proof of case $u\ne 0$, one may obtain the inequality
\[ \Pb \Big( \bigcap\limits_{i \le I} \left\{ \widetilde{\Dc}_{T}(x_i; \mu \ubar{r}_T) \le \eps \right\} \Big)  \le \left( 1 -   T^{-C  \mu^{2d} } \right)^{I} +  c_0 I^3 \bar k \le e^{ - I \cdot T^{-C \mu^{2d}} }   + c_0 I^3 \bar k ,\]
where, $\bar k$ has the same meaning as above, and, therefore, may be bounded by \eqref{eq:bar(k) bound decay}, thanks to Corollary \ref{cor:decay correlations}. Finally, we recall that $I$ is given by \eqref{eq:I well-separated points}, and choose $\mu$ so that $C \mu^{2d} < \delta_{1}/4$ to conclude, in a manner similar to case $u \neq 0$, that 
\[ \Pb \Big( \bigcup\limits_{i \le I} \left\{ \widetilde{\Dc}_{T,0}(x_i; \mu \ubar{r}_T) > \eps \right\} \Big)  \to 1 \]
completing the proof.

\bigskip
\appendix

\section{Sign-balance at multiple scales: Proof of Proposition \ref{prop:no scale increase}}
\label{apx:no scale increase}

We first reduce Proposition \ref{prop:no scale increase} to a statement about the existence of a certain continuous function $\xi:\R_{\ge 0}\rightarrow [-1,1]$, depending on $r > 1$, satisfying
\begin{equation}
\label{e:add}
 \int\limits_{B_{1}(x)}\xi(\|y\|) dy = 0 , \  \forall x \in \R^2, \qquad \text{and} \qquad \int\limits_{B_{r}(0)}\xi(\|y\|) dy > 0. 
 \end{equation}
Indeed, the next two lemmas guarantee that, for \textit{every} continuous $\xi:\R_{\ge 0}\rightarrow [-1,1]$, we can construct a sequence of functions $\{f_j\}$ on $\R^2$ whose \textit{asymptotic defect density} at $x \in \R^2$ is prescribed by $\xi(\|x\|)$, in the sense that, for every $B_r(x) \subseteq \R^2$,
\begin{equation}
\label{eq:asymp density prop}
\lim\limits_{j\rightarrow \infty}   \int \limits_{B_r(x)} H( f_j(y) ) \, dy  = \int \limits_{B_r(x)}  \xi( \|y\| )   dy  .
\end{equation}
The first lemma is standard, and its proof is omitted. 

 \begin{lemma}
\label{lem:red func to sets}
Let $\Ac$ be a finite collection of disjoint smooth compact domains. Then there exists a smooth function $f:\R^{2}\rightarrow \R$ such that $f^{-1}([0,\infty)) = \Ac$.
\end{lemma}

\begin{lemma}
\label{lem:red comp dens}
Let $\xi(\rho):\R_{\ge 0}\rightarrow [-1,1]$ be a continuous function. Then there exists a sequence $(\Ac_{j})_j$ of collections of disjoint smooth compact domains such that, for every $B_r(x) \subseteq \R^2$, 
\begin{equation}
\label{eq:asymp density prop2}
\lim\limits_{j\rightarrow \infty} | \Ac_{j} \cap B_r(x) |  = \int \limits_{B_r(x)}  \frac{1}{2}\left(1+ \xi(\|y\|)\right)  dy .
\end{equation}
\end{lemma}
\begin{proof}
Let us suppose that $\xi(\rho) \in (-1,1)$ for simplicity (otherwise we can apply a density argument to reduce to this case). Let $(R_j)_j$ be a sequence such that $R_j \to \infty$ as $j \to \infty$. For every $j \in \mathbb{N}$ and integers $-R_j^2 \le i_1,i_2 \le R_j^2$, let $\Ac_{j; i_1,j_2}$ be a smooth domain contained in the interior of the box 
\[ R^{-1}_j (i_1,i_2) +  R_j^{-1} \cdot [0,1]^2  \subseteq \R^2 \]
such that its volume satisfies
\[ |\Ac_{j; i_1,j_2}| =   \frac{R_j^{-2} }{2} \Big(1+ \xi \Big( R^{-1}_j \sqrt{i_1^2 + i_2^2} \Big) \Big) \in (0, R_j^{-2} )  ,\]
which is possible by the assumption that $\xi(\rho) \in (-1,1)$. Observe that $(\Ac_{j; i_1,j_2})_{i_1,i_2}$ are disjoint, and define $\Ac_j = \bigcup\limits_{i_1,i_2} \Ac_{j; i_1,j_2}$. It is easy to check that the indicator of $\Ac_j$ converges vaguely to $\frac{1}{2}(1+\xi)$ in the sense of integration against continuous bounded and compactly supported functions. Hence the sequence $\Ac_j$ satisfies \eqref{eq:asymp density prop2} by the Portmanteau Theorem.
\end{proof}
 
In fact, a simple modification of the proof shows that one could, in addition, impose the condition that $\Ac_j$ is a single smooth compact domain for every $j \ge 1$. We next confirm the existence of a suitable asymptotic defect density, with the help of the following lemma about the positive zeros $\{z_{k}\}_{k \in \mathbb{N}}$ of the Bessel $J_{1}$ function (in increasing order). 

\begin{lemma}
\label{lem:bess}
For every $r > 1$ there exists a number $k \in \mathbb{N}$ such that $r\cdot z_{k}$ is not contained in $\{z_{k}\}_{k \in \mathbb{N}}$. 
\end{lemma}
\begin{proof}
Let $r > 1$, and suppose for contradiction that for every $k\ge 1$ there exists $m=m_k\in\Z$ with $r\cdot z_{k} = z_{m}$. Recall the asymptotics 
\begin{equation}
\label{eq:Bessel zeros asymp}
z_{k} = \Big( k+\frac{1}{4}  \Big)\pi - \frac{3}{8\pi \Big( k+\frac{1}{4}  \Big)} + O\Big( \frac{1}{k^{3}} \Big) .
\end{equation}
By \eqref{eq:Bessel zeros asymp} we have
\begin{equation}
\label{eq:offset 1/4}
4 r\cdot ( k+ 1 ) = 4 m_k + 1 + o(1).
\end{equation}
Subtracting the corresponding asymptotics for $k$ and $k+1$ we obtain
\begin{equation*}
r = m_{k+1}-m_{k} + o(1),
\end{equation*} 
which forces $r\in \Z_{\ge 2}$. Since this ensures that both sides of \eqref{eq:offset 1/4} are integers, the asymptotic equality in \eqref{eq:offset 1/4} must eventually be exact, i.e.\ for sufficiently large $k$
\begin{equation}
\label{eq:m=rk+C}
m_k = r\cdot k + C
\end{equation}
with $C=C(r) := \frac{1}{4}(r-1)$. Substituting \eqref{eq:m=rk+C} into \eqref{eq:Bessel zeros asymp} yields the asymptotic equality
\[ r\cdot \Big( k+\frac{1}{4}  \Big)\pi - r\cdot \frac{3}{8\pi \left( k+\frac{1}{4}  \right)} + O\Big( \frac{1}{k^{3}} \Big) = \Big( rk+C+\frac{1}{4}  \Big)\pi - \frac{3}{8\pi \left( rk+C+\frac{1}{4}  \right)} + O\Big( \frac{1}{k^{3}} \Big). \]
Comparing the coefficients of $\frac{1}{k}$, this forces $r=\frac{1}{r} = 1$, a contradiction.
\end{proof}

Now fix $r > 1$, and let $k = k_r$ be such that $r \cdot z_k$ is not contained in $\{z_{k}\}_{k \in \mathbb{N}}$ (possible by Lemma \ref{lem:bess}).  Abbreviate $s = s_r := z_k$ and define
\begin{equation}
\label{eq:xi J0 def}
\xi(\rho) = \xi_r(\rho) := (-1)^{\id_{J_1(rs) < 0}} J_{0} ( s\cdot \rho ).
\end{equation}

\begin{lemma}
\label{lem:dens exist}
For every $r > 1$, the function $\xi = \xi_r$ defined in \eqref{eq:xi J0 def} satisfies \eqref{e:add}.
\end{lemma}

\begin{proof}
For $t > 0$, we write
\begin{equation*}
\psi_t(x) = \int\limits_{B_{t}(x)}\xi(|y|) dy = \left(\xi \star \chi_{t}\right)(x),
\end{equation*}
where $\star$ denotes convolution, $\chi_{t}$ is the characteristic function of the ball $B_t(0)$, and we abuse notation by writing $\xi(\cdot)$ for the radial function $\xi(\|y\|)$. Formally, the Fourier transform of $\psi_t$ is 
\begin{equation*}
\widehat{\psi_t}(\zeta) = \widehat{\xi}(\zeta)\cdot ( t \widehat{\chi_{1}}(t \zeta) ),
\end{equation*}
where (cf.\ \eqref{eq:Ad Fourier Bessel})
$$\widehat{\chi_{1}}(\zeta) = 2\pi\cdot \frac{J_{1}(\|\zeta\|)}{\|\zeta\|}, $$
 and, recalling the definition \eqref{eq:xi J0 def} of $\xi$, $\widehat{\xi}$ is formally $(-1)^{\id_{J_1(rs) < 0}} \times \frac{1}{s}$ times the uniform measure on $\partial B_{s}(0)$. Since $s$ is a zero of $J_{1}(\cdot)$, it follows that $\widehat{\psi}_1(\zeta)\equiv 0$, which is the first statement in \eqref{e:add}. Similarly for the second statement in \eqref{e:add}, we formally have by Plancherel 
\begin{equation*}
\int\limits_{B_{r}(0)}\xi(|y|) dy = \psi_r(0) = \left\langle  \xi,\, \chi_{1}\left(\frac{\cdot}{r}\right) \right\rangle = \left\langle  \widehat{\xi},\, r\widehat{\chi_{1}}\left(r\cdot \right) \right\rangle =  2\pi\cdot \frac{|J_{1}(rs)|}{s^{2}} > 0 .
\end{equation*}
To rigorously justify these arguments one can mollify $\xi$ and pass to the limit, but we omit the details.
\end{proof}

We can now conclude the proof of Proposition \ref{prop:no scale increase}:

\begin{proof}[Proof of Proposition \ref{prop:no scale increase}]
Let $\xi$ be the function in Lemma \ref{lem:dens exist}. Applying Lemmas \ref{lem:red func to sets} and \ref{lem:red comp dens} we obtain a sequence $\{f_{j}\}_{j\ge 1}$ of smooth functions whose positive nodal domains $A_{j}=f^{-1}([0,\infty))$ satisfy \eqref{eq:asymp density prop2}, and which therefore also satisfy  \eqref{eq:asymp density prop}. Recalling \eqref{e:add}, this completes the proof of the pointwise convergence in Proposition \ref{prop:no scale increase}. It remains to argue that the limit in \eqref{eq:sign balan rad=1} is uniform w.r.t.\ $x$ in compact sets. Indeed, this follows directly from the {\em equicontinuity} of the function  $x\mapsto \int_{B_{1}(x)}H(f_{j}(y))dy$ since the convergence of equicontinuous functions on a compact set to a continuous limit is {\em uniform}.
\end{proof}

\section{Defect derivatives w.r.t.\ deformation: Proof of Lemma \ref{lem:defect derivatives}}
\label{sec:defect derivatives}
For simplicity, we assume that the nodal line $\phi_{0}^{-1}(0)$ does not intersect the boundary $\partial\Pi$, as otherwise we may apply a density argument to reduce to this case. The function $\phi_{0}$ naturally gives rise to the partition 
\begin{equation}
\label{e:part}
\Pi = \Cc_{1}^{+}\cup\ldots \cup \Cc^{+}_{m}\cup \Cc^{-}_{1}\cup\ldots\cup\Cc^{-}_{k}
\end{equation}
of $\Pi$ into the {\em nodal domains} of $\phi_{0}$ intersecting $\Pi$, where $\Cc^{+}_{j}$, $1\le j\le m$ (resp.\ $\Cc^{-}_{j}$, $1\le j\le k$) are the {\em positive} (resp.\ negative) nodal domains. Then,
\begin{equation}
\label{eq:D(0) sum areas}
D(0) = \frac{1}{|\Pi|} \Big(\sum\limits_{j=1}^{m}|\Cc_{j}^{+}|- \sum\limits_{j=1}^{k}|\Cc_{j}^{-}| \Big) = \frac{2}{|\Pi|}\sum\limits_{j=1}^{m}|\Cc_{j}^{+}|- 1.
\end{equation}
Since, by assumption, $\phi_{0}$ does not have critical zeros, by Morse theory the partition of $\Pi$ \eqref{e:part} into nodal domains of $\phi_{t}$ is locally constant for $|t|$ sufficiently small. Moreover, for every $1\le j\le m$ there is a smooth bijective map $$p:\partial\Cc^{+}_{j}\rightarrow \partial\Cc^{+}_{j}$$ between the corresponding boundaries, so that 
\begin{equation}
\label{eq:p(x)=x+rN}
p(x) = x+ r(x;t)\cdot \vec{N}(x),
\end{equation}
where $\vec{N}(x)= \frac{\nabla \phi_{0}(x)}{\|\nabla \phi_{0}(x)\|}$ is a {\em inward} unit normal vector to $\partial \Cc_{j}^{+}$ at $x\in\partial\Cc_{j}^{+}$, in accordance to Lemma \ref{lem:defect derivatives}(ii.) 
(with the same, suitably adjusted, holding for the negative nodal domains),
and $$r(x;t):\partial\Cc_{j}^{+}\times [-\delta_{0},\delta_{0}]\rightarrow \R$$ is a smooth $2$-variable function so that $r(x;0)\equiv 0$, cf. ~\cite[Proof of Lemma 4.7]{BW}.
With the help of the implicit equation $$\phi_{t}(p(x)) = \phi_{0}\left(x+r(x;t)\cdot\vec{N}(x)\right) + t\psi\left(x+r(x;t)\cdot\vec{N}(x)\right)=0$$ w.r.t.\ $t$ in a neighbourhood of $0$, we may recover 
\begin{equation}
\label{eq:r(x;t) asymp}
r(x;t) = -t\cdot \frac{\psi(x)}{\|\nabla \phi_{0}(x)\|} + O(t^{2}).
\end{equation}
Therefore, for $j\le m$, $$|\Cc_{j}^{+}(t)| = |\Cc_{j}^{+}| + t\cdot \int\limits_{\partial\Cc_{j}^{+}} \frac{\psi(x)}{\|\nabla \phi_{0}(x)\|}   dx + O(t^{2}),$$
with the correction term, stemming from the curvature of $\partial\Cc_{j}$, absorbed in the $O(t^{2})$ error term. Inserting into \eqref{eq:D(0) sum areas}, and comparing it to the analogous expression for $D(t)$, shows that
\begin{equation}
\label{eq:D'(0) expr int2} 
D'(0) = \frac{2}{|\Pi|}\sum\limits_{j=0}^{m}\int\limits_{\partial\Cc_{j}^{+}} \frac{\psi(x)}{\|\nabla \phi_{0}(x)\|}   dx = \frac{2}{|\Pi|}\int\limits_{\phi_{0}^{-1}(0)\cap \Pi} \frac{\psi(x)}{\|\nabla \phi_{0}(x)\|}   dx .
\end{equation}

\vspace{2mm}
To evaluate $D''(0)$, we observe that the same argument leading to \eqref{eq:D'(0) expr int2} gives
\begin{equation}
\label{eq:D'(t) expr int}
\begin{split}
D'(t) &= \frac{2}{|\Pi|}\int\limits_{\phi_{t}^{-1}(0)\cap \Pi} \frac{\psi(x)}{\|\nabla  \phi_{t}(x)\|} dx ,
\end{split}
\end{equation}
provided $|t|$ is sufficiently small. Let us denote for brevity 
\begin{equation}
\label{eq:Upsilon func def}
\Upsilon_{t}(x):=\frac{\psi(x)}{\|\nabla  \phi_{t}(x)\|},
\end{equation}
and $\gamma_{t}:=\phi_{t}^{-1}(0)\cap \Pi$. Writing 
\begin{equation*}
\int\limits_{\gamma_{t}} \Upsilon_{t}(x)dx = \int\limits_{\gamma_{0}} \Upsilon_{t}(x)dx + \Big(\int\limits_{\gamma_{t}} \Upsilon_{t}(x)dx-\int\limits_{\gamma_{0}} \Upsilon_{t}(x)dx   \Big) ,
\end{equation*}
we have $D''(0) = \frac{2}{|\Pi|}\cdot \left( I_{1}(0)+I_{2}(0)\right)$
where 
\begin{equation}
\label{eq:I2 def Upsilon gamma}
I_{1}(t) := \int\limits_{\gamma_{0}} \frac{\partial}{\partial t}\Upsilon_{t}(x)dx \qquad \text{and} \qquad I_{2}(t): = \frac{\partial}{\partial t}\Big[\int\limits_{\gamma_{t}} \Upsilon_{t}(x)dx-\int\limits_{\gamma_{0}} \Upsilon_{t}(x)dx   \Big] .
\end{equation}
Recalling \eqref{eq:Upsilon func def},
\begin{align*}
I_{1}(t)&=-\int\limits_{\gamma_{0}}   \frac{\psi(x)}{\|\nabla  \phi_{t}(x)\|^{2}} \cdot \frac{\partial}{\partial t}\left[  \|\nabla  \phi_{t}(x)\|  \right]        dx = -\int\limits_{\gamma_{0}}   \frac{\psi(x)}{\|\nabla  \phi_{t}(x)\|^{3}}  \cdot\frac{\partial}{\partial t}\left[\left\langle \nabla\phi_{t}(x),\nabla\phi_{t}(x)\right\rangle \right]        dx\\
&=-\int\limits_{\gamma_{0}}   \frac{\psi(x)}{\|\nabla  \phi_{t}(x)\|^{3}}  \cdot\left\langle \frac{\partial}{\partial t}\left[\nabla\phi_{t}(x)\right],\nabla\phi_{t}(x)\right\rangle        dx.
\end{align*}
Since $\frac{\partial}{\partial t}\left[\nabla\phi_{t}(x)\right]=\nabla\psi(x)$, we obtain
\begin{equation}
\label{eq:I10 expl}
I_{1}(0) = -\int\limits_{\gamma_{0}}   \frac{\psi(x)}{\|\nabla  \phi_{t}(x)\|^{3}}  \cdot \left\langle \nabla\psi(x),\nabla\phi_{0}(x)\right\rangle        dx.
\end{equation}
Next, we evaluate $I_{2}(0)$. We have
\begin{equation*}
\int\limits_{\gamma_{t}} \Upsilon_{t}(x)dx = \int\limits_{\gamma_{0}} \Upsilon_{t}\left(x+r(x;t)\cdot \vec{N}(x)\right) \cdot \left(1-\kappa(x)\cdot r(x;t)\right)dx,
\end{equation*}
where $r(x;t)$ and $\vec{N}(x)$ are as in \eqref{eq:p(x)=x+rN}, and $\kappa(x)$ is as in Lemma \ref{lem:defect derivatives}(ii.). Inserting into \eqref{eq:I2 def Upsilon gamma} gives
\begin{equation*}
I_{2}(0)=\int\limits_{\gamma_{0}}  \frac{\partial}{\partial t}\left[\Upsilon_{t}\left(x+r(x;t)\cdot \vec{N}(x)\right) \cdot \left(1-\kappa(x)\cdot r(x;t)\right)\right]\bigg|_{t=0} dy.
\end{equation*}
Since $\Upsilon_{t}(\cdot)$ is given by \eqref{eq:Upsilon func def}, and $r(x;t)$ is given by \eqref{eq:r(x;t) asymp}, we may write
\begin{equation}
\label{eq:I2(0) int gamma0}
I_{2}(0)=  -\int\limits_{\gamma_{0}}  \frac{\psi(x)}{\|\nabla\phi_{0}(x)\|}\cdot \partial_{\vec{N}(x)}\left[\frac{\psi(x)}{\|\nabla  \phi_{0}(x)\|}  \right]dx
+\int\limits_{\gamma_{0}} \frac{\psi(x)}{\|\nabla  \phi_{0}(x)\|} \cdot \kappa(x)\frac{\psi(x)}{\|\nabla \phi_{0}(x)\|}  dx .
\end{equation}
As above, we evaluate
\begin{align*}
\partial_{\vec{N}(x)} \Big[\frac{\psi(x)}{\|\nabla  \phi_{0}(x)\|}  \Big] &= \Big\langle \vec{N}(x),  \nabla\Big[\frac{\psi(x)}{\|\nabla  \phi_{0}(x)\|}\Big] \Big\rangle \\
& =\frac{1}{\|\nabla\phi_{0}(x)\|^{3}} \left\langle \nabla\phi_{0}(x), \nabla\psi(x)\cdot \|\nabla\phi_{0}(x)\| - \psi(x)\cdot\nabla\|\nabla\phi_{0}(x)\|\right\rangle
\\&= \frac{\left\langle \nabla\phi_{0}(x),\nabla\psi(x)\right\rangle}{\|\nabla\phi_{0}(x)\|^{2}} - \frac{\psi(x)}{\|\nabla\phi_{0}(x)\|^{3}}\partial_{\nabla\phi_{0}(x)}\left[\left\|\nabla\phi_{0}(x)\right\|\right].
\end{align*}
Substituting into \eqref{eq:I2(0) int gamma0} yields the expression
\begin{equation}
\label{eq:I20 expl}
\begin{split}
I_{2}(0) &= -\int\limits_{\gamma_{0}}\frac{\psi(x)\cdot\left\langle \nabla\phi_{0}(x),\nabla\psi(x)\right\rangle}{\|\nabla\phi_{0}(x)\|^{3}}dx + \int\limits_{\gamma_{0}}\frac{\psi(x)^{2}}{\|\nabla\phi_{0}(x)\|^{4}}\cdot \partial_{\nabla\phi_{0}(x)}\left[\left\|\nabla\phi_{0}(x)\right\|\right] dx\\& \qquad \qquad +\int\limits_{\gamma_{0}} \frac{\psi(x)^{2}}{\|\nabla  \phi_{0}(x)\|^{2}} \cdot \kappa(x)dx.
\end{split}
\end{equation}
Consolidating \eqref{eq:I10 expl} and \eqref{eq:I20 expl}, and noticing that $I_{1}(0)$ coincides with the first term on the r.h.s.\ of \eqref{eq:I20 expl}, concludes the proof.


\bibliographystyle{halpha-abbrv}


\end{document}